\documentclass[12pt,leqno]{article}
\usepackage{makeidx}

\usepackage{amssymb,amsfonts,stmaryrd,mathrsfs}
\usepackage{amsmath,latexsym,amsthm,mathtools}

\usepackage{float}

\usepackage{hyperref}
\usepackage{manfnt}
\usepackage{verbatim}
\usepackage[utf8]{inputenc} 
\usepackage[cyr]{aeguill}
\usepackage[russian,francais,english]{babel}
\usepackage{frcursive}

\usepackage{url}
\let\urlorig\url
\renewcommand{\url}[1]{%
  \begin{otherlanguage}{english}\urlorig{#1}\end{otherlanguage}%
}
\usepackage{graphicx}

\makeindex

\newtheorem{lemma}{Lemma}
\newtheorem{corollary}{Corollary}
\newtheorem{proposition}{Proposition}
\newtheorem{theorem}{Theorem}
\newtheorem{remark}{Remark}

\newtheorem{assumption}{Assumption}
\newtheorem{notation}{Notation}

\renewcommand{\U}{\mathscr{U}}
\newcommand{\V}{\mathscr{V}}
\newcommand{\Potential}{\mathcal{W}}
\newcommand{\Redpot}{\mathscr{R}}

\newcommand{\Redpotc}{\check{\mathscr{R}}}

\newcommand{\dif}{{\rm d}}

\newcommand{\vits}{{u}}

\newcommand{\vol}{{v}}

\newcommand{\f}{{\bf f}}
\newcommand{\g}{{\bf g}}
\newcommand{\fc}{\check{\bf f}}

\newcommand{\Y}{{\mathscr Y}}
\newcommand{\Yc}{\check{\mathscr Y}}
\newcommand{\Ss}{{\mathscr S}}
\newcommand{\ZZ}{{\mathscr Z}}

\newcommand{\bu}{{u}}
\newcommand{\ubu}{\underline{u}}
\newcommand{\bv}{{v}}
\newcommand{\ubv}{\underline{v}}
\newcommand{\h}{{h}}
\newcommand{\bw}{{w}}
\newcommand{\bz}{{z}}
\newcommand{\bzc}{{\check{z}}}
\newcommand{\bZ}{{\bf Z}}
\newcommand{\bY}{{\bf Y}}
\newcommand{\bW}{{\bf W}}

\newcommand{\bF}{{\bf F}}

\newcommand{\bX}{{\bf X}}

\newcommand{\blambda}{\boldsymbol{\lambda}}

\newcommand{\Sb}{{\mathbb{S}}}

\newcommand{\Pb}{{\mathbb{P}}}
\newcommand{\bB}{{\bf B}}
\newcommand{\bS}{{\bf S}}

\newcommand{\bA}{{\bf A}}
\newcommand{\bD}{{\bf D}}

\newcommand{\bU}{{\bf U}}
\newcommand{\ubU}{\underline{\bf U}}

\newcommand{\bmu}{\boldsymbol{\mu}}

\newcommand{\bE}{{\bf E}}

\newcommand{\bV}{{\bf V}}
\newcommand{\bT}{{\bf T}}

\newcommand{\tot}{\!\mbox{\scriptsize tot}}

\newcommand{\R}{{\mathbb R}}

\newcommand{\Z}{{\mathbb Z}}
\newcommand{\N}{{\mathbb N}}

\newcommand{\D}{{\bf D}}

\newcommand{\hess}{\delta^2}

\newcommand{\Euler}{\delta}

\newcommand{\Legendre}{{\sf L}}

\newcommand{\Hess}{\nabla^2}

\newcommand{\Lag}{\mathscr{L}}

\renewcommand{\cap}{{\kappa}}

\newcommand{\free}{{f}}
\newcommand{\En}{\mathscr{E}}
\newcommand{\en}{e}
\newcommand{\Ec}{\mathcal{I}}
\newcommand{\Ham}{\mathcal{H}}

\newcommand{\Mom}{\mathcal{M}}

\newcommand{\Impulse}{\mathcal{Q}}

\newcommand{\impulse}{{q}}

\renewcommand{\d}{\partial}

\def\transp#1{{#1}^{{\sf T}}}
\newcommand{\negsign}{{\sf n}}

\newcommand{\speed}{{c}}

\newcommand{\bC}{{\bf C}}

\newcommand{\EKE}{\mbox{(EKE)}}
\newcommand{\EKL}{\mbox{(EKL)}}
\newcommand{\qKdV}{\mbox{(qKdV)}}

\newcommand{\Action}{{\Theta}}
\newcommand{\Iphi}{I({\phi})}

\begin{document}
\title{Stability of periodic waves in Hamiltonian PDEs of either long wavelength or small amplitude}
\author{S. Benzoni-Gavage\thanks{Universit\'e de Lyon,
CNRS UMR 5208,
Universit\'e Claude Bernard Lyon 1,
Institut Camille Jordan,
43 bd du 11 novembre 1918;
F-69622 Villeurbanne cedex, FRANCE}, C. Mietka and L.M. Rodrigues\thanks{Universit\'e de Rennes 1,
CNRS UMR 6625,
263 av. du General Leclerc;
F-35042 Rennes cedex, FRANCE}}
\maketitle

\begin{abstract}
Stability criteria have been derived and investigated in the last decades for many kinds of periodic traveling wave solutions to Hamiltonian PDEs. They
turned out to depend in a crucial way on the negative signature - or Morse index - of the Hessian matrix of action integrals associated with those waves. 
In a previous paper (published in Nonlinearity in 2016), the authors addressed the characterization of stability of periodic waves for a rather large class of Hamiltonian partial differential equations that includes quasilinear generalizations of the Korteweg--de Vries equation and dispersive perturbations of the Euler equations for compressible fluids, either in Lagrangian or in Eulerian coordinates. They derived a sufficient condition for orbital stability with respect to co-periodic perturbations, and a necessary condition for spectral stability, both in terms of the negative signature of the Hessian matrix of the action integral. Here the asymptotic behavior of this matrix is investigated in two asymptotic regimes, namely for small amplitude waves and for waves approaching a solitary wave 
as their wavelength goes to infinity. The special structure of the matrices involved in the 
expansions makes possible to actually compute the negative signature of the action Hessian both in the harmonic limit and in the soliton limit. As a consequence, it is found that nondegenerate small amplitude waves are 
orbitally stable with respect to co-periodic perturbations in this framework. For waves of long wavelength, the negative signature of the action Hessian is found to be exactly governed by $\Mom''(\speed)$, the second derivative with respect to the wave speed of the Boussinesq momentum associated with the limiting solitary wave. 
This gives an alternate proof of Gardner's result [J. Reine Angew. Math. 1997] according to which the instability of the limiting 
solitary wave, 
when $\Mom''(\speed)$ is negative, implies the instability of nearby periodic waves. Interestingly enough, it is found here that in the inverse situation, when $\Mom''(\speed)$ is positive, nearby periodic waves are orbitally stable with respect to co-periodic perturbations.
\end{abstract}



\tableofcontents

\section{Introduction}
In our earlier work \cite{BMR}, we considered a rather large class of Hamiltonian - systems of - partial differential equations (PDEs), and derived stability conditions for their periodic traveling wave solutions. The main purpose of this paper is to investigate these conditions in two asymptotic regimes, namely for waves of small amplitude and for waves of long wavelength. The former are by definition almost harmonic waves close to constant states, while the latter are `close' to solitary waves, in the sense that their trajectories in the phase plane of the governing ordinary differential equations (ODEs) are close to a homoclinic loop.

Roughly speaking, our stability conditions consist of a \emph{necessary} condition \emph{for spectral stability}, and of a \emph{sufficient} condition \emph{for orbital stability} with respect to \emph{co-periodic perturbations} - that is, perturbations of the same spatial period as the underlying wave. Because of a gap in between those conditions our main result in \cite{BMR} can be viewed as a \emph{pseudo-alternative}. 
As will be clear when we actually recall this result (see Theorem~\ref{thm:stabcrit} hereafter), the gap lies mostly in the stability conditions themselves, which are not the reciprocal of each other even if we eliminate degenerate critical situations. The gap is however also present in the kind of stability that is being addressed. On the one hand, orbital stability deals with \emph{nonlinear stability} by overcoming the translation invariance issue, while spectral stability only has to do with stability at the \emph{linearized level}. On the other hand, stability with respect to \emph{co-periodic perturbations} is quite restrictive, whereas when we talk about spectral stability we allow somehow arbitrary, \emph{localized} perturbations, typically in $H^s(\R)$ for $s$ large enough. Our spectral instability criterion provides 
instability with respect to both
co-periodic perturbations and 
localized ones. 

Up to our knowledge, the nonlinear stability of periodic traveling waves with respect to localized perturbations is a widely open question for dispersive PDEs like the ones we are considering. This is in sharp contrast with the breakthrough that was made recently for dissipative PDEs \cite{JNRZ-Inventiones}. Addressing the tough issue of  
nonlinear stability with respect to localized perturbations is not our purpose here. See however \cite{R_linearKdV} for a first step in this direction.

The present paper has a two-fold motivation
. First of all we want to exemplify how conditions derived in \cite{BMR} may be investigated analytically in some asymptotic regimes. We see this as a welcome complement to the development of numerical tools for `arbitrary' periodic waves in the bulk \cite{BMR,Mietka_PhD}, in particular because these tools can hardly attain asymptotic regimes or cover the whole range of parameters. Even more importantly, the asymptotic expansions derived here in the small amplitude and 
solitary wave limits to elucidate stability conditions are \emph{per se} involved in many related issues \cite{BNR-GDR-AEDP,BNR-JNLS14,BMR2-II}, including the understanding of dispersive shocks to be investigated elsewhere. 
An initially unexpected outcome of this work is that it fills the gap in between the stability conditions mentioned above, in those asymptotic regimes.

To give a glimpse of the nature of our pseudo-alternative we can draw a parallel - as in \cite[Remarks~3 \&~6]{BMR} - between our infinite-dimensional stability analysis and the classical stability theory for steady states of finite-dimensional Hamiltonian systems of ODEs. For a Hamiltonian system of size $2d$, $d\in\N^*$, say in canonical coordinates, either the Hamiltonian $H$ or its opposite may serve as a Lyapunov function for the stability analysis of steady states at which it has a local extremum. Speaking in terms of the negative signature - or Morse index - of the Hessian matrix $\nabla^2H$ of that function $H$, that is, the number of negative directions for the associated quadratic form in $\R^{2d}$, we find a local minimum when this signature equals zero and a local maximum when it equals $2d$ - provided that $\nabla^2H$ is nonsingular. In both cases we can thus infer the nonlinear stability of the associated steady state from Lyapunov theory. 
In addition, the linearized system about a nondegenerate steady state - where by definition $\nabla H$ vanishes with $\nabla^2H$ being nonsingular - has an unstable mode as soon as the negative signature of $\nabla^2H$ is odd at that point. This merely follows from a mean value argument between zero and $+\infty$  for the characteristic polynomial of the linearized system matrix $J\nabla^2H$, where $J$ is the skew-adjoint matrix associated with the underlying symplectic form on $\R^{2d}$, since $\det (J\nabla^2H)=(\det J) \,(\det \nabla^2H)$ with
 $\det J=1$ by definition of canonical coordinates and $\det \nabla^2H<0$ if the negative signature of $\nabla^2H$ is odd and $\nabla^2H$ is nonsingular. In summary, at nondegenerate steady states an odd negative signature of $\nabla^2H$ implies spectral instability, whereas a negative signature 
equal to either zero or $2d$ implies nonlinear stability. This is already a kind of pseudo-alternative, which reduces to a true alternative only in the special case $d=1$.
We cannot hope a priori for a better situation in infinite dimensions.

The main achievement in \cite{BMR} has been to turn the infinite-dimensional problem of stability of periodic traveling wave solutions to systems of $N$ PDEs (with actually $N=1$ or $2$) into a pseudo-alternative in dimension $N+2$.
This effective dimension $N+2$ corresponds to the number of independent parameters that determine the periodic waves profiles - in the class of systems 
we are considering - either up to a trivial phase shift parameter or including this phase shift parameter but holding fixed the spatial period.
 
It turns out that a further dimension reduction occurs in the two distinguished limits under consideration here, namely the small amplitude and solitary wave limits, which we shall indifferently refer to as  the \emph{harmonic} and \emph{soliton}\footnote{Throughout the text we endow the term `soliton' with the very broad meaning of solitary wave, disregarding any interacting properties.} limits, respectively.
This reduction eventually yields a true alternative - in nondegenerate cases - about the stability of periodic traveling waves of sufficiently small amplitude or large wavelength.

\bigskip

To be more specific, let us begin with the large wavelength limit. Intuitively, and this is confirmed by our findings, the co-periodic stability of periodic waves of large wavelength is related to the stability of solitary waves under localized perturbations, that is 
stability in the class of functions with fixed endstate. 
Solitary waves are in general parametrized by their endstate (in $\R^N$), phase speed, and phase shift, which make
$(N+2)$-dimensional families. However, when we fix their endstate we are left with two-dimensional families. By analogy with the ODE situation described above this 
hints at the possibility of a true alternative for the corresponding stability problem. This is indeed what is known to happen in most cases.
Under assumptions expected to be generic,
the stability of a solitary wave solution to an Hamiltonian PDE is characterized by means of the \emph{Boussinesq momentum of stability}, which is defined as an \emph{action integral} along the solitary wave orbit - a homoclinic loop in the phase plane - and can be viewed as a function of the speed of the wave only, once its endstate has been fixed. If we denote it by $\Mom(\speed)$,  with $\speed$ being the wave speed, we know that
\begin{itemize}
\item if $\Mom''(\speed)<0$ then the solitary wave is spectrally unstable, as follows from Evans functions calculations, see \cite{PegoWeinstein} for scalar equations, \cite{JMAA10} for Euler--Korteweg systems (also see \cite{SBG-DIE} for a survey and further references);
\item if $\Mom''(\speed)>0$ then the solitary wave is orbitally stable by the general theory of Grillakis, Shatah and Strauss \cite{GrillakisShatahStrauss}, which subsumes earlier results and has been widely used since then to produce numerous specific results.
\end{itemize}
One may argue that this is still a pseudo-alternative - even if we preclude the critical case $\Mom''(\speed)=0$ - in that there is still a gap in between the kinds of stability that are considered. Actually, \cite{GrillakisShatahStrauss} also contains a result that fills this gap for some class of equations, namely by showing that $\Mom''(\speed)<0$ implies \emph{orbital instability}, but it does not apply to our Hamiltonian framework because the involved skew-adjoint operator is not onto. However, this difficulty can also be overcome in some cases, as was done by Bona, Souganidis and Strauss for the Korteweg--de Vries equation \cite{BonaSouganidisStrauss}. For simplicity we may think of the stability problem for solitary waves as a true alternative.

The pseudo-alternative introduced in \cite{BMR} for the co-periodic stability of periodic waves  is a neat analogue of what is known for solitary waves, in that it gives a sufficient condition for spectral instability and a sufficient condition for orbital stability in terms of a single tool, which plays the role of the Boussinesq momentum for periodic waves. 
This tool is still an \emph{action integral} along the wave orbit in the phase plane, but compared with the Boussinesq momentum it depends on more parameters. Actually, for the systems of $N$ PDEs we are considering ($N=1$ or $2$), the periodic waves action $\Action$ depends on $N+2$ parameters (as it is independent of the phase shift parameter). This is just one more parameter than for the Boussinesq momentum of a solitary wave, if we count its speed for one parameter and its endstate for $N$ parameters. However, while the stability of solitary waves is governed by the convexity properties of a function of a single variable $\Mom(\speed)$ - since the endstate is held fixed - the stability of periodic waves depends on the \emph{whole} Hessian $\nabla^2\Action$ of the action
with respect to all parameters - a $(N+2)\times (N+2)$ matrix. 
Our pseudo-alternative is expressed indeed in terms of the negative signature of $\nabla^2\Action$, similarly as what happens for finite-dimensional Hamiltonian systems of ODEs. One of the goals here is to make the connection with the alternative for solitary waves.

A connection was already made by Gardner between the instability of solitary waves and that of periodic waves of long wavelength.
He proved indeed in \cite{Gardner97} that periodic waves close to spectrally unstable solitary waves are also spectrally unstable. 
An outcome of our present work is a similar result 
showing that periodic waves close to solitary waves for which $\Mom''(\speed)<0$ are also spectrally unstable. The latter result 
might be obtained by combining \cite{Gardner97} with \cite{GrillakisShatahStrauss}. However our proof 
differs significantly from \cite{Gardner97} and does not rely on any known results for solitary waves. 
It hinges on the first part of our own pseudo-alternative - the sufficient condition for spectral instability - together with a careful asymptotic expansion of the second derivatives of the action, which yields leading order terms for the matrix $\nabla^2\Action$ having a very peculiar form. 

From the second part of our pseudo-alternative - the sufficient condition for orbital stability - together with our asymptotic expansion we derive the newer, striking upshot that periodic waves close to solitary waves for which $\Mom''(\speed)>0$ are orbitally stable with respect to co-periodic perturbations. Thus in the large wavelength limit this completes the picture and provides a true alternative for nondegenerate cases.

\bigskip

Regarding small amplitude waves, the dimensional reduction is even stronger. Our approach actually leads to their orbitally stability with respect to co-periodic perturbations in \emph{any} nondegenerate case, regardless of the stability of limiting states viewed as steady constant solutions, hereafter merely referred to as constant states.

This may be thought of as coming from two combined facts. The first one is that constant states form an $N$-dimensional family (since any constant function is a steady solution). The second one is more subtle and is linked to an important feature of co-periodic stability, namely that it is associated with several constraints.

Indeed, as is clearly apparent in \cite{BMR} the co-periodic stability of periodic waves is related to evolutions and minimization under constraints, essentially of fixed period and of fixed averages for the components of the solution and the impulse - or momentum, associated with the invariance of equations under spatial translations. The constraint of a fixed endstate in the solitary wave limit, and of a fixed period and fixed averages in the harmonic limit are limiting traces of those\footnote{The constraint on the impulse becomes redundant in both the harmonic and soliton limits.}.

Our stability problem reduces in the small amplitude limit to the stability of  the limiting constant state in the class of functions with period equal to the harmonic period and with the same average values as the constant state under consideration. Those extra constraints reduce to zero the dimension of the family of relevant constant solutions.

The latter observation is indeed crucial. By contrast, constant states 
may not be spectrally stable under more general perturbations, like arbitrary localized ones or periodic ones that are neither of zero mean nor of harmonic period.
This does not happen in the scalar case $N=1$ but does happen when $N=2$. More details are given in Appendix~\ref{app:harmonic}, in which we also point out a few facts about the \emph{spectral stability} of constant states.
These are indeed spectrally stable under arbitrary perturbations in the scalar case ($N=1$), or for $N=2$ when they are endstates of solitary waves. However, in the case $N=2$ states generating a family of small-amplitude periodic waves need to satisfy an extra condition for begin spectrally stable. This condition happens to be the \emph{hyperbolicity} of the underlying dispersionless system.
Fortunately this is irrelevant for our purpose because those constant states are always spectrally stable under relevant perturbations. 

A direct computation of the relevant constrained infinite-dimensional Hessian shows that its negative signature is zero and that its kernel is $2$-dimensional, consistently with the extra two dimensions of the family of periodic waves. This hints at the orbital stability of small amplitude periodic waves proved below. 

As in the soliton limit, those pieces of information are given here just to motivate what follows, and provide \emph{a posteriori} enlightenment. We shall actually not use any of those.

\bigskip

Besides the divergence of final outcomes, there is also a notable discrepancy in the nature of expansions obtained in the harmonic and soliton limits. The former is regular and the leading-order term in the expansion of $\nabla^2\Action$ is sufficient to obtain all required pieces of information. By contrast, the latter is highly singular and a three-steps expansion of $\nabla^2\Action$ is needed to reach the stage where the expansion involves bounded terms and - more importantly -  conveys enough information to conclude. However once this multi-scale expansion has been obtained, the separation of scales helps in the actual computation of the negative signature of $\nabla^2\Action$ in the large wavelength limit. On the contrary, even though $N$ is small, the computation of the signature is slightly tricky in the small amplitude regime. We think that it would be hardly doable without gaining first some insight - not provided by the expansion - on the geometric nature of the problem at hand. In this respect, important orthogonality properties of vectors involved in the expansions of $\nabla^2\Action$ are gathered in Subsection~\ref{s:ortho}. We actually also use the related geometric upshots in the soliton limit in order to spare some computational effort and point out explicitly the single entry that is really important in the last matrix-valued coefficient of the three-steps expansion. Remarkably enough it is in the computation of this entry at the very last stage of the expansion that appears the decisive $\Mom''(\speed)$.

\bigskip

The present paper is organized as follows. In Section \ref{s:framework} we introduce the general framework, recall  from \cite{BMR} our pseudo-alternative, and state our main, present results. Section \ref{s:expansion} is devoted to the derivation of asymptotic expansions for $\nabla^2\Action$ in the small amplitude and in the soliton limits. In Section \ref{s:conclusions} we draw consequences from these expansions and prove our stability/instability results.

In a companion paper \cite{BMR2-II}, we consider the modulated equations associated with our class of Hamiltonian PDEs, and use the asymptotic expansions derived here to gain insight on modulated equations in the small amplitude and in the soliton limits, which are crucial to the understanding of dispersive shocks. 
In further work, including the forthcoming \cite{BS} focused on scalar equations, we plan to also address the joint limit and the existence of small amplitude dispersive shocks for Hamiltonian PDEs.

\section{Framework and main results}\label{s:framework}

\subsection{A class of Hamiltonian PDEs}\label{ss:framework}
As in our earlier work \cite{BMR}, we consider abstract systems of the form
\begin{equation}
\label{eq:absHamb}
\partial_t\bU = \partial_x (\bB\,\Euler \Ham[\bU])\,.\end{equation}
where the unknown $\bU$ takes values in $\R^N$, $\bB$ is a \emph{symmetric} and \emph{nonsingular} matrix, and $\Euler \Ham$ denotes the variational gradient of $\Ham=\Ham(\bU,\bU_x)$~-~the letter $\Euler$ standing for the \emph{Euler} operator
 (see explicit expression below). 
In practice, we are most concerned with the case $N=1$, which includes quasilinear, generalized versions of the {Korteweg}-{de} {Vries} equation
$$\qKdV\ \quad \partial_tv=\partial_x(\Euler \en[v])\,,\quad \en=\en(v,v_x)\,,$$
and the case $N=2$, which includes
the Euler--Korteweg system, either in Eulerian coordinates, 
\begin{equation*}\label{eq:EKabs1d}
\EKE\ \quad \left\{\begin{array}{l}\partial_t\varrho +\partial_x (\varrho \vits)\,=\,0\,,\\ [5pt]
\partial_t \vits + \vits\partial_x\vits \,+\,\partial_x(\Euler \En [\varrho])\,=\,0\,,
\end{array}\right.\qquad\qquad \En=\En(\varrho,\varrho_x)\,,
\end{equation*}
or in mass Lagrangian coordinates, 
\begin{equation*}\label{eq:EKabsLag}
\EKL\ \quad \left\{\begin{array}{l}\partial_t{\vol} \,=\,\partial_x{\vits}\,,\\ [5pt]
\partial_t {\vits} \,=\, \partial_x( \Euler \en [\vol])\,,
\end{array}\right.\qquad\qquad \en=\en(\vol,\vol_x)\,.
\end{equation*}
Readers interested in the underlying physics should keep in mind that the significance of the independent variables $(x,t)$ is not the same in \EKE\  and in \EKL\ - in \EKE , $x$ denotes a spatial coordinate (homogeneous to a length), whereas in \EKL , $x$  denotes a mass Lagrangian coordinate (homogeneous to a mass per squared length). The systems \EKE\  and \EKL\ are equivalent models - at least regarding their smooth solutions - for the motion of a compressible fluid of energy density 
$\En$ and specific energy $\en$, as long as its density $\varrho$ and specific volume $\vol=1/\varrho$ remain bounded away from zero. However, this equivalence is out of purpose here. We just consider \EKE\  and \EKL\  as two instances of our abstract framework \eqref{eq:absHamb}. 
We stress that it is important to keep the form of nonlinearities as general as possible to cover both cases arising directly from fluid mechanics considerations alluded to above and hydrodynamic formulations of other classical models, such as nonlinear Schr\"odinger equations or vortex filaments models. See \cite{SBG-DIE} for a detailed discussion. 

We might as well consider equations of the form \eqref{eq:absHamb} where $\bB\partial_x$ is replaced by a more complicated skew-adjoint pseudodifferential operator. This occurs for instance in the Benjamin--Bona--Mahony (BBM) equation, involving the nonlocal operator $(1-\partial_x^2)^{-1}$ instead of just the matrix $\bB$. This type of equation will be addressed elsewhere \cite{BS}.

Regarding $\bB$ and $\Ham=\Ham(\bU,\bU_x)$, we will make a few assumptions that are met by our main examples \qKdV, \EKE\  and \EKL , and ensure that periodic traveling waves arise as families parametrized by their speed, 
$(N+1)$ constants of integration (and a spatial phase shift that is immaterial to the present discussion and that hence will be omitted from now on). 
However, since our analysis is local in nature, we expect that assumptions on $\Ham$ - in particular quadraticity assumptions - could also be further relaxed to reach the generality considered in \cite{BNR-JNLS14}.

The properties of \qKdV, \EKE\  and \EKL\  regarding periodic traveling waves have been investigated and compared in detail in \cite{BMR}.
The scalar equation \qKdV\  is obviously of the form \eqref{eq:absHamb}, with
$$\bU=v\,,\;\bB=1\,,\;\Ham=\en(\bv,\bv_x)\,.$$ 
Both \EKE\  and \EKL\  fall into the special situation in which
\begin{equation}\label{eq:Ham}\Ham=\Ham[\bU]=\Ham(\bU,\bU_x)=\Ec(\bv,\bu)+\en(\bv,\bv_x)\,,\;\mbox{with\,}\;\bU=\left(\begin{array}{c} \bv \\ \bu\end{array}\right)\,,
\end{equation}
with $\Ec$ quadratic in $\bu$ - which is reminiscent of a kinetic energy - and
\begin{equation}\label{eq:B}\bB=\,\left(\begin{array}{c|c}0& b \\ \hline b &0\end{array}\right)\,,\;b\neq 0\,.\end{equation}
More precisely, for  \EKL\  we have $\Ec=\frac12 u^2$ and $b=1$,
and substituting the notations $\bv$ and $\en$ for $\varrho$ and 
$\En$ in \EKE\ we see that this system is of the form \eqref{eq:absHamb}-\eqref{eq:Ham}-\eqref{eq:B} with $b=-1$
and $\Ec=\frac12 \bv u^2$ - which would read $\frac12 \varrho u^2$ in a more standard notation for fluids.

From now on, we consider an abstract system \eqref{eq:absHamb} under the following restrictions.
\begin{assumption}
\label{as:ham}
Either $N=1$, $\bU=v$, $\Ham=\en(\bv,\bv_x)$, and $\bB=b$ is a nonzero real number, 
or $N=2$ with the Hamiltonian $\Ham$ being decomposed as in \eqref{eq:Ham}, 
and the symmetric matrix $\bB$ being as in \eqref{eq:B}. 

In both cases $\en$ is assumed to be quadratic in $\bv_x$ with $\kappa:=\partial_{\bv_x}^2 \en$ taking only positive values, while in the case $N=2$ the other part of the Hamiltonian, $\Ec$ is assumed to be quadratic in $\bu$, with $\tau:=\partial_{\bu}^2 \Ec$ taking only positive values as well.
\end{assumption}

This framework is designed to include  \qKdV, \EKE\  and \EKL\ at the same time.

To some extent the case $N=1$ can be viewed as the sub-case obtained by merely dropping the component $u$ and the kinetic energy $\Ec$ in \eqref{eq:absHamb}-\eqref{eq:Ham}-\eqref{eq:B}.
Whenever this point of view is indeed legitimate, to avoid repetition, computations will be carried out only in the case $N=2$. 
However when we come to conclusions regarding stability we will need to distinguish between the two cases.  

Here above and throughout the paper except in Section \ref{ss:sollim}, square brackets $[\cdot]$ signal a function of not only the dependent variable $\bU$ but also of its derivatives $\bU_x$, $\bU_{xx}$, etc.

By definition we have
$$\Euler \Ham[\bU]= \left(\begin{array}{c} \partial_\bv \Ec(\bv,\bu) +  \Euler \en[\vol] \\  \partial_\bu \Ec(\bv,\bu) \end{array}\right)\,,\qquad\qquad
\Euler \en[\bv] := \partial_\bv \en (\bv,\bv_x)\,-\,\partial_x(\partial_{\bv_x} \en (\bv,\bv_x))\,.$$
In addition, we introduce the \emph{impulse}
$$\Impulse(\bU)=\frac{1}{2} \bU\cdot \bB^{-1} \bU\,,$$
which is linked to the invariance of \eqref{eq:absHamb} with respect to $x$-translations in that
\begin{equation}\label{eq:translx}
\partial_x \bU = \partial_x (\bB\,\Euler \Impulse[\bU]) \,.
\end{equation}
Another connection with space translation invariance lies in the fact that the impulse satisfies a local conservation law, which plays a crucial role in the stability analysis developed in \cite{BMR,BNR-JNLS14}. 
This conservation law will not be used  explicitly in the present paper though. 
By definition, the impulse is just the well-known conserved quantity $\frac12 \bv^2$ for \qKdV, while $\Impulse=-\rho\vits$ is also well-known for \EKE, and $\Impulse= \vol\vits$ is slightly less usual for \EKL.

\subsection{Wave parameters and action integral}\label{ss:waveparam}

For a traveling wave $\bU=\ubU(x-\speed t)$ of speed $\speed$ to be solution to \eqref{eq:absHamb}, one must have  by \eqref{eq:translx}  that
$$\partial_x(\Euler (\Ham+\speed\Impulse)[\ubU])=0\,,$$
or equivalently, there must exist $\blambda\in \R^N$ such that
\begin{equation}
\label{eq:EL}
\Euler (\Ham+\speed\Impulse)[\ubU] \,+\,\blambda\,=\,0\,.
\end{equation} 
Equation \eqref{eq:EL} is the \emph{Euler--Lagrange equation} $\Euler\Lag[\ubU]=0$ associated with the \emph{Lagrangian}
$$\Lag= \Lag(\bU,\bU_x;\speed,\blambda):= \Ham(\bU,\bU_x)+\speed\Impulse(\bU)\,+\,\blambda\cdot \bU\,.$$
We thus see that  $\ubU_x$ is an integrating factor, and $$\Legendre\Lag[\ubU]:= \ubU_x\cdot \nabla_{\bU_x} \Lag(\ubU,\ubU_x) -\Lag(\ubU,\ubU_x)$$ is  a first integral of the profile ODEs in \eqref{eq:EL}
(The letter $\Legendre$ here above stands for a formal `\emph{Legendre} transform'.) 
Therefore, the possible profiles $\ubU$ are determined by the equations in \eqref{eq:EL} together with
\begin{equation}
\label{eq:ELham}
\Legendre\Lag[\ubU]\,=\,\mu\,,
\end{equation}
where $\mu$ is a constant of integration.  The scalar equation \eqref{eq:ELham} 
specifies the value of the above mentionned first integral for the system \eqref{eq:EL}. `Conversely', by differentiation of  \eqref{eq:ELham}, we receive $\ubU_x\cdot \Euler\Lag[\ubU]=0$. In the scalar case ($N=1$), this implies that all non constants solutions to \eqref{eq:ELham} automatically solve \eqref{eq:EL}. However, this is not true for $N=2$, and 
we consider both \eqref{eq:EL}  and \eqref{eq:ELham} as profile equations, despite some redundancy.

A traveling profile $\ubU$ is thus `generically' parametrized by $(\mu, \blambda, \speed)\in \R^{N+2}$. This is what happens for periodic profiles under Assumption \ref{as:ham} and a few other assumptions regarding the dependence of $\Ham$ on $\bv$.
See \cite{BMR,BNR-GDR-AEDP} for a more detailed discussion on the profile equations. Given such a periodic profile $\ubU$ of period $\Xi$, we define as in \cite{BMR,BNR-GDR-AEDP} the abbreviated action integral\footnote{This action integral has often been denoted by $K$ in the related KdV literature \cite{BronskiJohnson,BronskiJohnsonKapitula,Johnson}. We have refrained from doing so when dealing with problems in which the letter $K$ is usually reserved for capillarity - or sometimes a wave number.} 
\begin{equation}\label{eq:defTheta}
\Action(\mu, \blambda, \speed):= \int_{0}^{\Xi} 
(\Ham[\ubU]+\speed \Impulse(\ubU) + \blambda\cdot \ubU +\mu)\,\dif \xi\,,
\end{equation}
which because of  \eqref{eq:Ham}\eqref{eq:ELham} equivalently reads
\begin{equation}\label{eq:Theta}\Action(\mu, \blambda, \speed)\,=\,
\int_{0}^{\Xi}  \ubv_x \partial_{\bv_x} \en(\ubv,\ubv_x) \,\dif x =
\oint\partial_{\bv_x} \en(\bv,\bv_x)\,\dif \bv\,,
\end{equation}
where the symbol $\oint$ stands for the integral along the orbit described by $(\ubv,\ubv_x)$ in the $(\bv,\bv_x)$-plane, referred to as the phase plane. This quantity $\Action$ is known to play a crucial role in the stability analysis of periodic traveling waves, see \cite{BMR} and references therein. It also yields a closed form of the modulated equations associated with 
\eqref{eq:absHamb}, see \cite{BNR-GDR-AEDP,BMR2-II} and references therein.

For waves of small amplitude about a fixed point $\bv_0$, we have $(\ubv,\ubv_x)\to (\bv_0,0)$ when their amplitude goes to zero, while their period has a finite limit $\Xi_0$, the period of harmonic waves. This implies in particular that  the action integral $\Action$ goes to zero in the small amplitude limit. 

For a family of waves of large period approaching a loop homoclinic to a saddle-point $\bv_s$ in the phase plane, we see that
the action integral 
$$\int_{-\Xi/2}^{\Xi/2}  \Ham[\ubU]+\speed \Impulse(\ubU) + \blambda\cdot \ubU +\mu)\,\dif \xi\,=\,\int_{-\Xi/2}^{\Xi/2}  \ubv_x \partial_{\bv_x} \en(\ubv,\ubv_x) \,\dif x =
\oint\partial_{\bv_x} \en(\bv,\bv_x)\,\dif \bv
$$
tends to
$$\int_{-\infty}^{+\infty} (\Ham[\ubU^s]+\speed \Impulse(\ubU^s) + \blambda_s \cdot \ubU^s +\mu_s)\,\dif \xi\,=:\,\Mom(\speed,\bU_s)\,,$$
the Boussinesq moment of instability associated with the solitary wave 
of speed $\speed$ and endstate $\bU_s=(\bv_s,\bu_s)$, with 
$$\blambda_s := - \nabla_{\bU}(\Ham+\speed\Impulse)(\bU_s,0)\,,\quad 
\mu_s := - \blambda\cdot \bU_s\,-\,(\Ham+\speed\Impulse)(\bU_s,0)\,,$$
and  $\ubU^s$ denoting the profile of the corresponding solitary wave.

In general, a crucial observation is that the first-order derivatives of the action are merely integrals along the profile $\ubU$ obtained as though neither $\ubU$ nor the period would depend on the parameters $(\mu,\blambda,\speed)$. More precisely, as pointed out in \cite[Prop.~1]{BNR-GDR-AEDP} we have
\begin{equation}\label{eq:derTheta}
\partial_\mu\Action = \Xi= \int_{0}^{\Xi} \,\dif \xi\,,\qquad
\nabla_{\blambda} \Action =  \int_{0}^{\Xi}\ubU \,\dif \xi\,, \qquad
\partial_\speed\Action = \int_{0}^{\Xi} \Impulse(\ubU)\,\dif \xi\,.
\end{equation}

Let us now recall from \cite{BMR} the following stability/instability criteria, expressed in terms of the $(N+2)\times (N+2)$, Hessian matrix $\Hess\Action$ of $\Action$ as a function of $(\mu,\blambda,\speed)$ and its \emph{negative signature} $\negsign(\Hess\Action)$,
that is, the number of negative directions for the associated quadratic form in $\R^{N+2}$.

\begin{theorem}\label{thm:stabcrit} 
In the framework described in \S\ref{ss:framework}, in particular under Assumption \ref{as:ham}, for a wave locally parametrized by $(\mu,\blambda,\speed)$ such that $\partial_\mu^2\Action$ is nonzero and $\Hess\Action$ is nonsingular,
\begin{itemize}
\item if $\negsign(\Hess\Action)-N$ is zero then the wave is conditionally orbitally stable with respect to co-periodic perturbations;
\item if $\negsign(\Hess\Action)-N$ is odd then the wave is spectrally unstable.
\end{itemize}
\end{theorem}

The words `\emph{conditionally} orbitally stable' in the first statement mean that a solution to \eqref{eq:absHamb} stays - \emph{as long as it exists} - as close as we want in the energy space $H^1(\R)\times L^2(\R)^N$ to some translate of the underlying wave provided it starts close enough. In general, this does not encompass global existence of the solution, since local existence for this type of quasilinear PDE is usually known only in smoother functional spaces.

Spectral instability means that the linearized equations have exponentially growing solutions. 
We recall that here this may be realized with both co-periodic solutions and with smooth localized solutions.

Applying the stability/instability criteria in Theorem~\ref{thm:stabcrit}  to either small amplitude waves or long wavelength waves, and using the asymptotic expansions for $\Hess\Action$ derived in the next section, we can prove the following.
For the sake of conciseness, we give only rough statements here. 
Further details on nondegeneracy assumptions are provided in Section~\ref{s:conclusions}. For more on the nature of conclusions we refer to \cite{BMR}.

\begin{theorem}\label{thm:stabsmall} 
`Nondegenerate' periodic waves of sufficiently small amplitude approaching a constant steady state are such that
\begin{itemize}
\item Their period $\Xi$ is monotonically depending on $\mu$, that is, $\partial_\mu\Xi=\partial_\mu^2\Action\neq0$.
\item The matrix $\Hess\Action$ is nonsingular and
\begin{itemize}
\item $\det\Hess\Action\,<0$ in the case $N=1$;
\item $\det\Hess\Action\,>0$ in the case $N=2$.
\end{itemize}
\item In any case, these waves are conditionally orbitally stable with respect to co-periodic perturbations.
\end{itemize}
\end{theorem}

Note that the nature of the monotonicity of the period $\Xi$ with respect to $\mu$ does depend on the limiting constant state but that the result covers both possible cases.

\begin{theorem}\label{thm:stabsol} 
`Nondegenerate' periodic waves of sufficiently large period approaching a solitary wave are such that
\begin{itemize}
\item Their period $\Xi$ is monotonically increasing with $\mu$, that is, $\partial_\mu\Xi=\partial_\mu^2\Action>0$.
\item If for the limiting solitary wave the second derivative $\Mom''(\speed)$ of the Boussinesq momentum with respect to velocity at fixed endstate 
is nonzero, then the matrix $\Hess\Action$ is nonsingular for nearby periodic orbits 
\begin{itemize}
\item with  $\Mom''(\speed)\,\det\Hess\Action\,<0$ in the case $N=1$;
\item and $\Mom''(\speed)\,\det\Hess\Action\,>0$ in the case $N=2$.
\end{itemize}
\item If in addition $\Mom''(\speed)>0$ then 
the negative signature of  $\Hess\Action$ is equal to $N$ for nearby periodic orbits, and we have conditional orbital stability with respect to co-periodic perturbations of the corresponding waves.
\item If on the contrary $\Mom''(\speed)<0$ then
the negative signature of  $\Hess\Action$ is equal to $N+1$ and the corresponding wave is spectrally unstable.
\end{itemize}
\end{theorem}

\section{Derivation of asymptotic expansions}\label{s:expansion}

\subsection{Qualitative properties of profile equations}

Now our purpose is to investigate the small amplitude and soliton limits in the ODEs governing the traveling profiles $\ubU=(\ubv,\ubu)$.

Let us start by making these ODEs more explicit. A traveling profile $\ubU$ of speed $\speed$ is characterized by the existence of $(\blambda,\mu)\in \R^{N+1}$ such that \eqref{eq:EL}-\eqref{eq:ELham} hold true.
Let us first look at the most complicated case
 $N=2$. Recalling the forms in \eqref{eq:Ham} and \eqref{eq:B} of $\Ham$ and $\bB$
we see that the equations in \eqref{eq:EL}-\eqref{eq:ELham} are equivalent to
$$\left\{\begin{array}{l}\partial_\bv\en(\ubv,\ubv_x)\,-\,\partial_x(\partial_{\bv_x}\en(\ubv,\ubv_x))\,+ \,\partial_\bv \Ec(\ubv,\ubu)\,+\,(\speed/b)\, \ubu\,+\,\lambda_1\,=\,0\,,\\ [2\jot]
\partial_\bu \Ec(\ubv,\ubu)\,+\,(\speed/b)\, \ubv\,+\,\lambda_2\,=\,0\,,\\ [2\jot]
\ubv_x\,\partial_{\bv_x}\en(\ubv,\ubv_x)\,-\,\en(\ubv,\ubv_x)\,-\,\Ec(\ubv,\ubu)\,-\,(\speed/b)\,\ubv\,\ubu\,-\,\lambda_1\,\ubv\,-\,\lambda_2\,\ubu\,=\,\mu\,,
\end{array}\right.$$
Under the further assumption (see Assumption \ref{as:ham}) that $\en$ is quadratic in $\bv_x$ and $\Ec$ is  
quadratic\footnote{Assumptions of quadraticity on $\Ec$ is made for pure convenience to enable the explicit algebraic solving of \eqref{eq:profalg} resulting in an explicit formula for $g$ below. Relaxing the assumption on $\Ec$ would require only notational changes and invoking the implicit function theorem. Likewise we also expect quadraticity of $\En$ to be unnecessary, but we anticipate the computational cost to be heavier.} in $\bu$ with
$$\kappa:=\partial_{\bv_x}^2 \en\,>0\,,\qquad \tau:=\partial_{\bu}^2 \Ec\,>\,0\,,$$
and the normalization $\Ec(\cdot,0)\equiv0$, these profile equations are 
 equivalent\footnote{The proof of the equivalence, omitted here as classical, requires some care at points where $\ubv_x$ vanish.} for non constant solutions 
to the reduced algebro-differential system
\begin{spreadlines}{2\jot}
\begin{align}
& \tau(\ubv) \ubu\,+\,(\speed/b)\, \ubv\,+\,\lambda_2\,=\,0\,,\label{eq:profalg}\\
& \tfrac12 \kappa(\ubv)\,\ubv_x^2\,+\,\Potential(\ubv;\speed,\blambda)\,=\,\mu\,,\label{eq:profedo}
\end{align}
\end{spreadlines}
where 
$$\Potential(\bv;\speed,\blambda):=\,-\,\free(\bv)\,-\,\tfrac12 \tau(\bv)\,g(\bv;\speed,\lambda_2)^2
\,-\,
(\speed/b)\,\bv\,g(\bv;\speed,\lambda_2)\,-\,\lambda_1\,\bv\,-\,\lambda_2\,g(\bv;\speed,\lambda_2)\,,$$
$$\free(\bv):= \en(\bv,\bv_x)-\tfrac12 \kappa(\bv)\,\bv_x^2\,,\qquad\qquad g(\bv;\speed,\lambda):=-((\speed/b)\, \bv\,+\,\lambda)/\tau(\bv)\,.$$
For stationary solutions, \eqref{eq:EL}-\eqref{eq:ELham} is actually equivalent to \eqref{eq:profalg}-\eqref{eq:profedo} and $\partial_\bv \Potential(\ubv;\speed,\blambda)=0$. 
Using the definition of $g(\bv;\speed,\blambda)$ here above we can simplify a little bit the expression of the potential and write it as
$$\Potential(\bv;\speed,\blambda)=\,-\,\free(\bv)\,+\,\tfrac12 \tau(\bv)\,g(\bv;\speed,\lambda_2)^2\,-\,\lambda_1\,\bv\,.$$
In the case $N=1$, there is no component $\bu$ to eliminate, and the sole profile equation  \eqref{eq:ELham} is readily of the form \eqref{eq:profedo} with
$$\Potential(\bv;\speed,\lambda)=\,-\,\free(\bv)\,-\,\tfrac12 \tfrac{\speed}{b}\, \bv^2\,-\,\lambda\,\bv\,.$$

Given a reference state $\bU_*\in \R^N$, for all $\speed\in \R$ we can always find $(\blambda,\mu)\in \R^{N+1}$ such that $\bU_*$ is a stationary solution to \eqref{eq:EL}-\eqref{eq:ELham}. Indeed, in the case $N=2$ it suffices to define $\lambda_2$ from \eqref{eq:profalg} by
$$\lambda_2=-\tau(\bv_*)\bu_*-(\speed/b)\bv_*\,,$$
then take $$\lambda_1\,=\,-\,\free'(\bv_*)-\tfrac12 \tau'(\bv_*)\bu_*^2-(\speed/b)\bu_*\,,$$ to ensure that
$$\partial_\bv \Potential(\bv_*;\speed,\blambda)=0\,,$$ and finally choose
$$\mu= \Potential(\bv_*;\speed,\blambda)$$
because of \eqref{eq:profedo}. In the case $N=1$ we just take
$$\lambda\,=\,-\,\free'(\bv_*)-\tfrac{\speed}{b}\bv_*\,,\qquad\mu= \Potential(\bv_*;\speed,\lambda)\,.$$
The nature of the steady state $\bU_*$ as a solution to \eqref{eq:EL}-\eqref{eq:ELham} is the same as the nature of $\bv_*$ as a solution to \eqref{eq:profedo}, which merely depends on the sign of 
$\partial_v^2\Potential(\bv_*;\speed,\lambda)$. Since $\kappa$ is positive, the mapping 
$$(\bv,\bv_x)\mapsto \tfrac12 \kappa(\bv)\,\bv_x^2\,+\,\Potential(\bv;\speed,\blambda)$$
has 
\begin{itemize}
\item a \emph{center point} at $(\bv_*,0)$ if  $\partial_v^2\Potential(\bv_*;\speed,\blambda)>0$,
\item a \emph{saddle point} at $(\bv_*,0)$ if $\partial_v^2\Potential(\bv_*;\speed,\blambda)<0$.
\end{itemize}
In the former case, there are small amplitude solutions to \eqref{eq:EL}-\eqref{eq:ELham} oscillating around $\bU_*$. In the latter case, $\bU_*$ is a candidate for the endstate of a soliton, the existence of which requires that $\Potential(\cdot;\speed,\blambda)$ achieves the value $\mu$ at some other point than $\bv_*$, say $\bv^*$ such that $\partial_\bv \Potential(\bv^*;\speed,\blambda)\neq 0$.

At fixed $(\speed,\blambda)$, the most basic phase portrait for the planar system associated with \eqref{eq:profedo} arises when the potential $\Potential(\cdot;\speed,\blambda)$ admits a local minimum $\bv_0$ (center point), and a local maximum $\bv_s$ (saddle point) associated with 
another point $\bv^s$ at which $\Potential(\cdot;\speed,\blambda)$ achieves the same value $\mu_s:=\Potential(\bv_s;\speed,\blambda)$ and $\partial_\bv\Potential(\bv^s;\speed,\blambda)\neq 0$,
in such a way that $\Potential(\cdot;\speed,\blambda)$ is monotonically varying from $\mu_0:=\Potential(\bv_s;\speed,\blambda)$ to 
$\mu_s$ in between $\bv_0$ and $\bv_s$, and in between $\bv_0$ and $\bv^s$ as well. In this case, the saddle point $(\bv_s,0)$ is associated with a homoclinic loop corresponding to a solitary wave 
of endstate $\bv_s$, inside which - in the phase plane - there are closed orbits corresponding to periodic profiles. The amplitude of periodic waves tends to zero when $\mu\searrow \mu_0$, and their period tends to infinity when $\mu\nearrow \mu_s$. 

The reader may refer to Figure \ref{fig:portrait-generic}, which shows the phase portrait associated with KdV traveling waves, for which $\Potential (v;\lambda,\speed)= \tfrac16 v^3 \,-\,\tfrac{\speed}{2} v^2\,-\,\lambda\,v$, and is qualitatively the same for all potentials that vary as in Table \ref{tb:var} below.
\begin{table}[hb]
$$\begin{array}{c|cccccccccccccc}
\bv &  & & v_1 & & v_s & & v_2  &  & v_0 & & v_3 & & v^s & \\ \hline
\Potential' & & & + & & 0 & & - & & 0 & & + & &   \\  \hline
& & &  & & \mu_s & & & & & &  & &   \mu_s & \\
 & &  & & \nearrow &  & \searrow  & & & & & & \nearrow &   \\ [-5pt]
\Potential & & & \mu & & & & \mu & & & & \mu & &   \\
 & &  \nearrow &  &  &  &  &  & \searrow  & &  \nearrow & &  &   \\
  & &   &  &  &  &  &   &  & \mu_0 & & &  &   \\
\end{array}$$
\caption{Variations of potential}\label{tb:var}
\end{table}

\begin{figure}[H]
\begin{center}
\includegraphics[width=70mm]{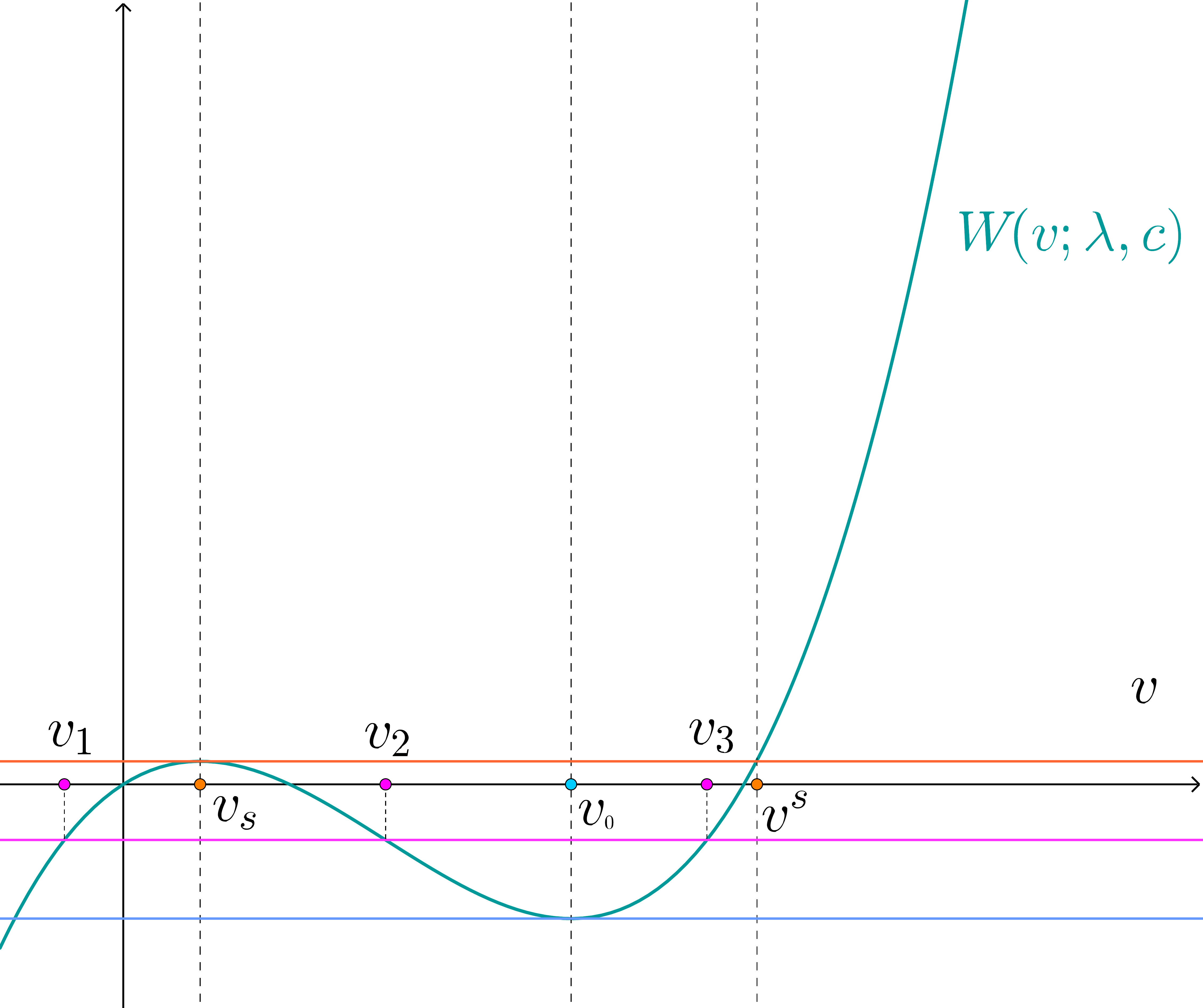}

\vspace{2mm}

\includegraphics[width=70mm]{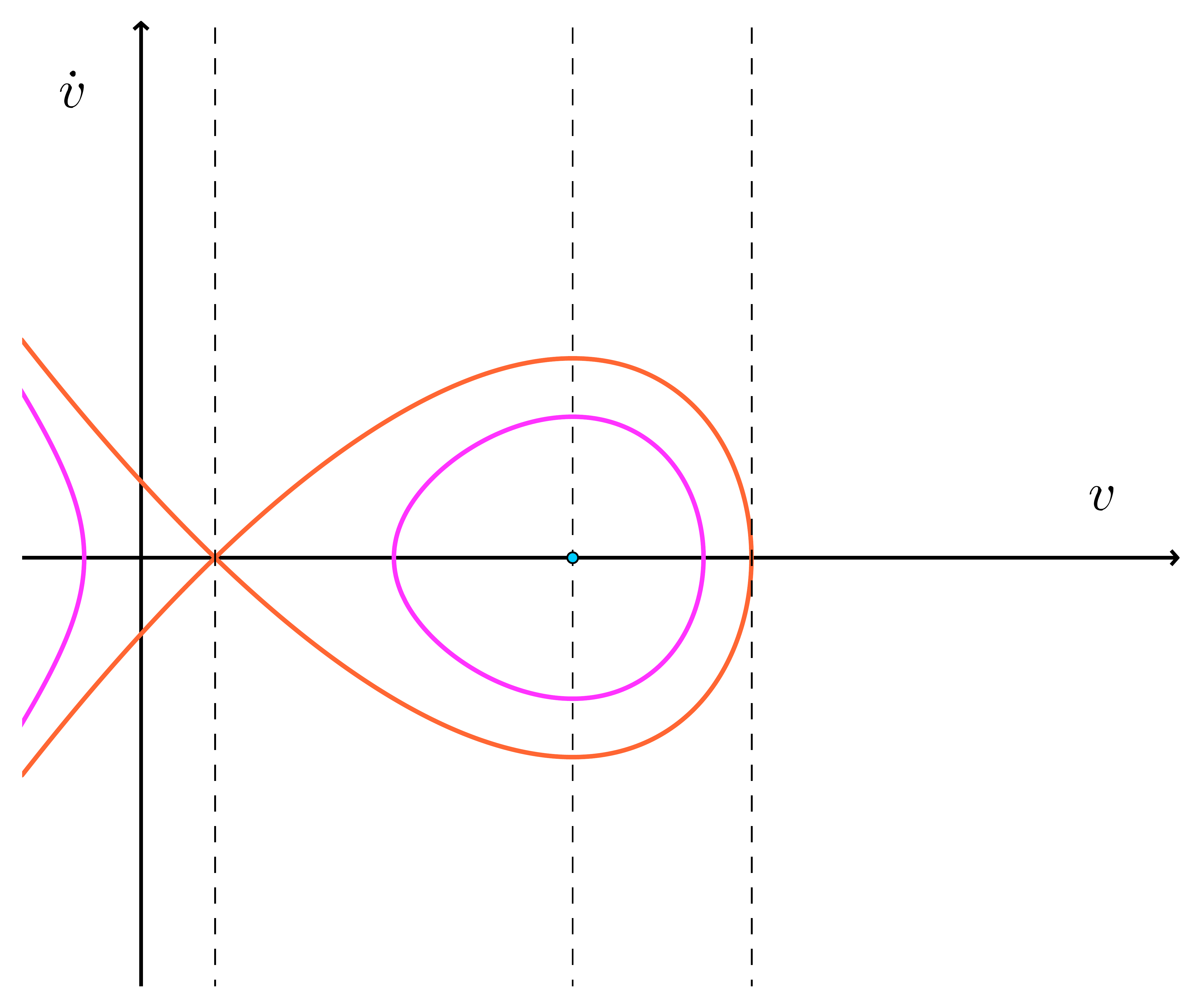}

\caption{`Generic' potential (top) and phase portrait (bottom) for traveling waves}\label{fig:portrait-generic}
\end{center}
\end{figure}

Phase portraits like on Figure \ref{fig:portrait-generic} display a `bright soliton', whose profile has a bump.
A symmetric phase portrait arises when the point $\bv^s$ is found to be lower than $\bv_s$, which yields a `dark soliton', or hole. 
When $\free'''$ changes sign there may be a larger number of stationnary solutions and the phase portrait may be quite complicated, see \cite{BNR-JNLS14}. For a simple instance, by enlarging the view one may find both sorts of solitary waves, dark and bright, corresponding to loops homoclinic to the same endstate, and closed orbits as well surrounding them in the phase plane, decribing periodic waves asymptotically obtained by gluing both solitary waves together.

It is important to note however that our theorem relies on local considerations and does not require any topological assumptions. It applies to any family of periodic waves collapsing to either a stationnary state or a single homoclinic loop. We also expect - but do not pursue this goal here - that the same techniques could be applied to analyze the case, alluded to above, where periodic loops shrink to two homoclinics. In contrast, cases involving the concatenation of homoclinics and heretoclinics are not expected to persist under variation of parameters $(\speed,\blambda)$ and do not fit our present framework. 

Concerning stability of the phase portrait under small perturbations of the parameters $(\speed,\blambda)$, for the sake of clarity and concision, let us concentrate on the basic situation described in Table~\ref{tb:var}.

\begin{proposition}\label{prop:perturbpp} Assume that $(\speed_*,\blambda_*)\in \R^{N+1}$ and 
$\bv_{s*}<\bv_{0*}<\bv^{s*}$ are such that
$$\mu_{0*}:=\Potential(\bv_{0*};\speed_*,\blambda_*)<\mu_{s*}:=\Potential(\bv_{s*};\speed_*,\blambda_*)=\Potential(\bv^{s*};\speed_*,\blambda_*)\,,$$
$$\partial_\bv\Potential(\bv_{0*};\speed_*,\blambda_*)=\partial_\bv\Potential(\bv_{s*};\speed_*,\blambda_*)=0\,,$$
$$
\partial_\bv^2\Potential(\bv_{s*};\speed_*,\blambda_*)<0\,,\;
\partial_\bv^2\Potential(\bv_{0*};\speed_*,\blambda_*)>0\,,$$
$$\forall \bv\in (\bv_{s*},\bv_{0*})\,,\quad \partial_\bv\Potential(\bv;\speed_*,\blambda_*)<0\,,$$
$$\forall \bv\in (\bv_{0*},\bv^{s*}]\,,\quad \partial_\bv\Potential(\bv;\speed_*,\blambda_*)>0\,.$$
Then there is a neighborhood of $(\speed_*,\blambda_*)$ and $\bv_s$, $\bv^s$, $\bv_0$ smoothly parametrized by $(\speed,\blambda)$ in this neighborhood such that
$$
\Potential(\bv_{0};\speed,\blambda)<
\Potential(\bv_{s};\speed,\blambda)=\Potential(\bv^{s};\speed,\blambda)\,,$$
$$\partial_\bv\Potential(\bv_{0};\speed,\blambda)=\partial_\bv\Potential(\bv_{s};\speed,\blambda)=0\,,$$
$$
\partial_\bv^2\Potential(\bv_{s};\speed,\blambda)<0\,,\;
\partial_\bv^2\Potential(\bv_{0};\speed,\blambda)>0\,,$$
$$\forall \bv\in (\bv_{s},\bv_{0})\,,\quad \partial_\bv\Potential(\bv;\speed,\blambda)<0\,,$$
$$\forall \bv\in (\bv_{0},\bv^{s}]\,,\quad \partial_\bv\Potential(\bv;\speed,\blambda)>0\,.$$
\end{proposition}

\begin{proof}
By the implicit function theorem applied to $\partial_\bv\Potential$ we find $\bv_0$ such that $\partial_\bv\Potential(\bv_{0};\speed,\blambda)=0$,
and by the implicit function theorem applied to 
$$(v,w)\mapsto(\partial_\bv\Potential(\bv;\speed,\blambda), \Potential(w;\speed,\blambda)-\Potential(\bv;\speed,\blambda))$$
we find $(\bv_s,\bv^s)$ such that 
$\partial_\bv\Potential(\bv_{s};\speed,\blambda)=0$ and $\Potential(\bv_{s};\speed,\blambda)=\Potential(\bv^{s};\speed,\blambda)$.
The inequalities 
$$\partial_\bv\Potential(\bv^{s};\speed,\blambda)>0\,,\;
\partial_\bv^2\Potential(\bv_{s};\speed,\blambda)<0\,,\;
\partial_\bv^2\Potential(\bv_{0};\speed,\blambda)>0\,,$$
which are valid at $(\speed_*,\blambda_*)$,
are preserved in a neighborhood of  $(\speed_*,\blambda_*)$, and they imply $\partial_\bv\Potential(\bv;\speed,\blambda)<0$ for all
$\bv\in (\bv_{s},\bv_{0})$ close to either  $\bv_{s}$ or $\bv_{0}$, and $\partial_\bv\Potential(\bv;\speed,\blambda)>0$ for all 
$\bv\in (\bv_{0},\bv^{s}]$ close to either  $\bv_{0}$ or $\bv^{s}$.
Away from $\bv_{s}$ and $\bv_{0}$ and for small variations of the parameters $(\speed,\blambda)$, $\partial_\bv\Potential(\bv;\speed,\blambda)$ does not vanish and thus keep the same sign.
\end{proof}

We can now investigate `the small amplitude limit' and `the soliton limit' in the following way.
If $(\speed_*,\blambda_*,\bv_{s*},\bv_{0*},\mu_{0*},\mu_{s*})$ satisfy the assumptions of Proposition~\ref{prop:perturbpp},
the small amplitude limit consists in considering the closed orbits around $\bv_0$ when $(\mu,\blambda,\speed)$ goes to $(\speed_*,\blambda_*,\mu_{0*})$ with $\mu>\mu_{0}(\speed,\blambda):=\Potential(\bv_0(\speed,\blambda);\speed,\blambda)$, and the soliton limit concerns closed orbits of large period that are inside the loop homoclinic to $\bv_s$ when  $(\mu,\blambda,\speed)$ goes to $(\speed_*,\blambda_*,\mu_{s*})$ with $\mu<\mu_{s}(\speed,\blambda):=\Potential(\bv_s(\speed,\blambda);\speed,\blambda)$.

For $\mu\in (\mu_0(\speed,\blambda),\mu_s(\speed,\blambda))$ there are two roots of $\mu-\Potential(\cdot;\speed,\blambda)$ in $(\bv_s(\speed,\blambda),\bv^s(\speed,\blambda))$, which we denote by $\bv_k(\mu,\blambda,\speed)$, $k=2,3$, such that
$$\bv_s<\bv_2<\bv_0<\bv_3\,.$$
At some point, for $\mu$ close to $\mu_s$, it will be useful to introduce a third root $\bv_1<\bv_s$ of $\mu-\Potential$, close to $\bv_s$. The existence of such a root is ensured as long as the potential $\Potential$ achieves the value $\mu$ on the left of $\bv_s$ and follows from 
$$
\partial_\bv\Potential(\bv_{s*};\speed_*,\blambda_*)=0\,,\qquad
\partial_\bv^2\Potential(\bv_{s*};\speed_*,\blambda_*)<0\,,
$$
for $(\speed,\blambda,\mu)$ in a neighborhood of $(\speed_*,\blambda_*,\mu_s)$, with $\mu<\mu_s$.

\subsection{Asymptotics of basic quantities}\label{s:asbasic}
\subsubsection{Preparation}
We are primarily interested in the asymptotic behaviors of the following quantities.
\begin{itemize}
\item The action integral 
$$\Action= \int_{-\Xi/2}^{\Xi/2}
(\Ham[\ubU]+\speed \Impulse(\ubU) + \blambda\cdot \ubU +\mu)\,\dif \xi\,,$$
where $g(\ubv;\speed,\lambda_2)$ is to be substituted for $\ubu$ in $\ubU=(\ubv,\ubu)$ in the case $N=2$ because of \eqref{eq:profalg}.  
In view of \eqref{eq:profedo} we may equivalently write
$$\Action=\oint\partial_{\bv_x} \en(\bv,\bv_x)\,\dif \bv\,=\,2\,\int_{\bv_2(\mu,\blambda,\speed)}^{\bv_3(\mu,\blambda,\speed)} {\sqrt{2\cap(\bv)\,(\mu-\Potential(\bv;\speed,\blambda))}}\,\dif \bv\,.$$

\item The period $\Xi\,=\,\partial_\mu\Action$ of the waves, given by
$$\Xi= \int_{-\Xi/2}^{\Xi/2}\dif \xi= \oint \frac{\dif \bv}{\bv_x} = 2\int_{\bv_2(\mu,\blambda,\speed)}^{\bv_3(\mu,\blambda,\speed)} \frac{\dif \bv}{\sqrt{2(\mu-\Potential(\bv;\speed,\blambda))/\cap(\bv)}}\,,$$

\item More generally, nonlinear integrals along the wave profiles
$$\int_{-\Xi/2}^{\Xi/2} G(\ubv(\xi))\,\dif \xi\,=\ \oint \frac{G(\bv)\,\dif \bv}{\bv_x} \,=\,2\int_{\bv_2(\mu,\blambda,\speed)}^{\bv_3(\mu,\blambda,\speed)} 
\frac{G(\bv)\,\dif \bv}{\sqrt{2(\mu-\Potential(\bv;\speed,\blambda))/\cap(\bv)}}\,,$$
and in particular the other first order derivatives of the action integral, which by \eqref{eq:derTheta} read
$$\partial_{\blambda_1}  \Action(\mu,\blambda,\speed) \,=\,\int_{-\Xi/2}^{\Xi/2}\ubv(\xi)\,\dif \xi= \oint \frac{\bv\,\dif \bv}{\bv_x} \,=\,2\int_{\bv_2(\mu,\blambda,\speed)}^{\bv_3(\mu,\blambda,\speed)} 
\frac{\bv\,\dif \bv}{\sqrt{2(\mu-\Potential(\bv;\speed,\blambda))/\cap(\bv)}}\,,$$
$$\partial_{\blambda_2}  \Action(\mu,\blambda,\speed)\, =\,\int_{-\Xi/2}^{\Xi/2}g(\ubv(\xi);\speed,\lambda_2)\,\dif \xi
\,=\,2\int_{\bv_2(\mu,\blambda,\speed)}^{\bv_3(\mu,\blambda,\speed)} 
\frac{g(\bv;\speed,\lambda_2)\,\dif \bv}{\sqrt{2(\mu-\Potential(\bv;\speed,\blambda))/\cap(\bv)}}\,,$$
$$\partial_\speed\Action(\mu,\blambda,\speed)\,=\, \int_{-\Xi/2}^{\Xi/2}q(\ubv(\xi);\speed,\lambda_2)\,\dif \xi
 \,=\,2\int_{\bv_2(\mu,\blambda,\speed)}^{\bv_3(\mu,\blambda,\speed)} 
\frac{
q(\bv;\speed,\lambda_2)
\,\dif \bv}{\sqrt{2(\mu-\Potential(\bv;\speed,\blambda))/\cap(\bv)}}\,,$$
with
$$\impulse(\bv;\speed,\lambda):=\Impulse(\bv,g(\bv;\speed,\lambda))\,$$
in the case $N=2$ and $$\impulse(\bv;\speed,\lambda):=\Impulse(\bv)\,,$$
actually independent of $(\speed,\lambda)$, in the case $N=1$.
\end{itemize}

Introducing in addition the function
$$\ZZ(\bv;\mu,\blambda,\speed):=\mu-\Potential(\bv;\speed,\blambda)\,,$$
we may notice that its gradient with respect to the parameters $(\mu,\blambda,\speed)$ reads
\begin{equation}\label{eq:derZNtwo}
\nabla\ZZ\,=\,\left(\begin{array}{c}1 \\ \bv \\ g(\bv;\speed,\lambda_2) \\ q(\bv;\speed,\lambda_2) \end{array}\right)\,
\end{equation}
in the case $N=2$, while 
\begin{equation}\label{eq:derZNone}
\nabla\ZZ\,=\,\left(\begin{array}{c}1 \\ \bv \\ \Impulse(\bv) \end{array}\right)\,
\end{equation}
in the simpler case $N=1$. This can be checked by direct inspection or derived from the original Euler-Lagrange structure.

Therefore, using either \eqref{eq:derZNtwo} or \eqref{eq:derZNone}, we can write as a shortcut for the first order derivatives of the action
\begin{equation}\label{eq:gradaction}
\nabla\Action(\mu,\blambda,\speed)\,=\,\int_{\bv_2(\mu,\blambda,\speed)}^{\bv_3(\mu,\blambda,\speed)} \nabla \ZZ(\bv;\mu,\blambda,\speed)\,\sqrt{\frac{2\kappa(\bv)}{\ZZ(\bv;\mu,\blambda,\speed)}}\,\dif\bv\,.
\end{equation}
This formula can of course be obtained directly by differentiation under the integral of 
$$\Action\,=\,2\,\int_{\bv_2(\mu,\blambda,\speed)}^{\bv_3(\mu,\blambda,\speed)} {\sqrt{2\cap(\bv)\,\ZZ(\bv;\mu,\blambda,\speed)}}\,\dif \bv\,,$$
recalling that by definition
$$\ZZ(\bv_2(\mu,\blambda,\speed);\mu,\blambda,\speed)\,=\,\ZZ(\bv_3(\mu,\blambda,\speed);\mu,\blambda,\speed)\,=\,0\,.$$
However, pointwise relations \eqref{eq:derZNtwo} and \eqref{eq:derZNone} will turn to be useful in later arguments.

\begin{notation}\label{not:derivatives}
Here above and in what follows $\nabla$ denotes the gradient with respect to $(\mu,\blambda,\speed)$. For functions depending also on other variables, we will denote their other partial derivatives with subscripts. For instance,
$\partial_\bv\ZZ$ will be denoted by $\ZZ_\bv$. 
\end{notation}

The derivation of the asymptotic behaviors of our basic integrals begins with a change of variables that eliminates singularities at endpoints, which is of interest even for the less singular of them all, namely the action integral
$$\Action=\,2\,\int_{\bv_2(\mu,\blambda,\speed)}^{\bv_3(\mu,\blambda,\speed)} {\sqrt{2\cap(\bv)\,\ZZ(\bv;\mu,\blambda,\speed)}}\,\dif \bv\,.$$
A rather natural change of variables consists in writing, for $\bv\in (\bv_2,\bv_3)$,
$$\bv=\bv_2+\sigma (\bv_3-\bv_2)\,,\;\sigma=(1+\sin\theta)/2\in (0,1)\,,\;\theta\in (-\pi/2,\pi/2)\,.$$
Noticing that, by definition of $\bv_2$ and $\bv_3$ and by Taylor's formula
$$
\ZZ(\bv;\mu,\blambda,\speed)=(\bv_3- \bv)(\bv-\bv_2) \Redpot(\bv,\bv_2,\bv_3;\speed,\blambda)\,,$$
with
\begin{equation}\label{eq:redpot}
\Redpot(\bv,\bw,\bz;\speed,\blambda):=
\int_{0}^{1}\int_{0}^{1}t\partial_\bv^2\Potential(\bw+t(\bz-\bw)+ts(\bv-\bz);\speed,\blambda)\,\dif s\dif t\,,
\end{equation}
we see that, on the one hand,
$$\Action=\,(\bv_3-\bv_2)^2\,\int_{-\pi/2}^{\pi/2} {\sqrt{\cap(\V(\theta;\bv_2
,\bv_3
))\Redpot(\V(\theta;\bv_2
,\bv_3
),\bv_2
,\bv_3
;\speed,\blambda)/2}}\;\cos^2\theta \;\dif \theta\,,$$
and on the other hand, for any smooth function $\phi$,
$$\Iphi:=\int_{\bv_2(\mu,\blambda,\speed)}^{\bv_3(\mu,\blambda,\speed)} 
\frac{\phi(\bv)\,\dif \bv}{\sqrt{
\ZZ(\bv;\mu,\blambda,\speed)}}\,=\,
\int_{-\pi/2}^{\pi/2}\frac{\phi(\V(\theta;\bv_2
,\bv_3
))\,\dif \theta}{\sqrt{\Redpot(\V(\theta;\bv_2
,\bv_3
),\bv_2
,\bv_3
;\speed,\blambda)}}\,,$$
where 
 \begin{equation}\label{eq:V}
\V(\theta;\bv_2,\bv_3):=\bv_2+(1+\sin\theta) (\bv_3-\bv_2)/2\,,
\end{equation}
and $\bv_2=\bv_2(\mu,\blambda,\speed)$, $\bv_3=\bv_3(\mu,\blambda,\speed)$ in the integral expressions of $\Action$ and $\Iphi$ here above.
By definition, $\V(\theta;\bv_2,\bv_3)$ belongs to $(\bv_2,\bv_3)$ for $\theta\in (-\pi/2,\pi/2)$, and in particular goes to $\bv_{0*}$, the limiting state of both $\bv_2$ and $\bv_3$ in the small amplitude limit. We also see from the definition of $\Redpot$ that $\Redpot(\V(\theta;\bv_2(\mu,\blambda,\speed),\bv_3(\mu,\blambda,\speed)),\bv_2(\mu,\blambda,\speed),\bv_3(\mu,\blambda,\speed);\speed,\blambda)$ goes to $\frac12 \partial_\bv^2\Potential(\bv_{0*};\speed_*,\blambda_*)>0$  in the small amplitude limit. This readily implies that
$$\int_{\bv_2(\mu,\blambda,\speed)}^{\bv_3(\mu,\blambda,\speed)} 
\frac{\phi(\bv)\,\dif \bv}{\sqrt{\mu-\Potential(\bv;\speed,\blambda)}}\,\rightarrow\,
\frac{\pi \,\phi(\bv_{0*})}{\sqrt{\partial_\bv^2\Potential(\bv_{0*};\speed_*,\blambda_*)/2}}\,$$
in the small amplitude limit. In particular, by applying this to $\phi(\bv)=\sqrt{2\kappa(\bv)}$, we recover in the small amplitude limit the period 
$$\Xi_{0*}:={2\pi}\sqrt{\kappa(\bv_{0*})/\partial_\bv^2\Potential(\bv_{0*};\speed_*,\blambda_*)}$$
of \emph{harmonic} waves, which are by definition solutions to the quadratic equation
$$\tfrac12 \cap(\bv_{0*}) \ubv_x^2+\tfrac{1}{2}\partial_\bv^2\Potential(\bv_{0*};\speed_*,\blambda_*)\, \ubv^2= \mbox{constant}\,.$$

In order to study the soliton limit, we had better use the factorization
$$\mu-\Potential(\bv;\speed,\blambda)=-(\bv- \bv_1)(\bv-\bv_2) \Redpot(\bv,\bv_1,\bv_2;\speed,\blambda)\,,$$
and write
$$\Action=\,2\,(\bv_3-\bv_2)^2\int_{0}^{1} {\sqrt{-2\cap(\bv_2+\sigma(\bv_3-\bv_2))\Redpot(\bv_2+\sigma(\bv_3-\bv_2)),\bv_1,\bv_2;\speed,\blambda)}}\,\sqrt{\sigma\,(\sigma+\rho)} \;\dif \sigma
$$
and
$$\Iphi
=
\int_{0}^{1}\frac{\phi(\bv_2+\sigma(\bv_3-\bv_2))}{\sqrt{
-\Redpot(\bv_2+\sigma(\bv_3-\bv_2),\bv_1,\bv_2;\speed,\blambda)
}}\,\frac{\dif \sigma}{\sqrt{\sigma(\sigma+\rho)}}
$$
with
$$\rho:=\frac{\bv_2-\bv_1}{\bv_3-\bv_2}$$
where for simplicity we have omitted to write the dependence of 
$\bv_1$, $\bv_2$, $\bv_3$, and $\rho$ on $(\mu,\blambda,\speed)$.
In the soliton limit, $\rho$ goes to zero and the integral here above goes to infinity unless $\phi(\bv_{s*})=0$. 
This is in particular consistent with the fact that the period $\Xi$ - given by $I(\phi)$ with $\phi(\bv)=\sqrt{2\kappa(\bv)}$ - goes to infinity in the soliton limit.

\subsubsection{Asymptotics of roots and their gradients}
The starting point of our asymptotic expansions lies in the expansion of roots $\bv_1$, $\bv_2$, $\bv_3$ of $\ZZ(\cdot;\mu,\blambda,\speed)= \mu -\Potential(\cdot;\speed,\blambda)$, and of their gradients. These expansions are gathered in Proposition~\ref{prop:asvgradv} below.

Before stating this series of expansions, let us formulate our main assumptions and introduce a few additional pieces of notation.

\begin{assumption}\label{as:main}
We consider a potential $\Potential$ that is smoothly defined on  $I \times \Lambda$, where $I$ is an open interval in $\R$, and $\Lambda$ is an open subset of $\R^{N+1}$. We also assume the existence of smooth functions $\bv_0$, $\bv_s$, $\bv^s$ on $\Lambda$ 
valued in $I$ and 
such that for all $(\speed,\blambda)\in \Lambda$,
$$
\bv_s(\speed,\blambda)<\bv_0(\speed,\blambda)
<\bv^s(\speed,\blambda)\,,$$
$$\mu_{0}(\speed,\blambda):=
\Potential(\bv_{0}(\speed,\blambda);\speed,\blambda)<\mu_{s}(\speed,\blambda):=
\Potential(\bv_{s}(\speed,\blambda);\speed,\blambda)=\Potential(\bv^{s}(\speed,\blambda);\speed,\blambda)\,,$$
$$\partial_\bv\Potential(\bv_{0}(\speed,\blambda);\speed,\blambda)=\partial_\bv\Potential(\bv_{s}(\speed,\blambda);\speed,\blambda)=0\,,$$
$$
\partial_\bv^2\Potential(\bv_{s}(\speed,\blambda);\speed,\blambda)<0\,,\;
\partial_\bv^2\Potential(\bv_{0}(\speed,\blambda);\speed,\blambda)>0\,,$$
$$\forall \bv\in (\bv_{s}(\speed,\blambda),\bv_{0}(\speed,\blambda))\,,\quad \partial_\bv\Potential(\bv;\speed,\blambda)<0\,,$$
$$\forall \bv\in (\bv_{0}(\speed,\blambda),\bv^{s}(\speed,\blambda)]\,,\quad \partial_\bv\Potential(\bv;\speed,\blambda)>0\,.$$
We denote by $\Omega$ the open subset of $\R^{N+2}$ defined by
$$\Omega= \{\bmu=(\mu,\blambda,\speed)\,;\;(\speed,\blambda)\in \Lambda\,,\;
\mu_0(\speed,\blambda)<\mu<\mu_s(\speed,\blambda)\}\,,$$
and consider  $\bv_1$, $\bv_2$, $\bv_3$, smoothly defined on $\Omega$ and satisfying
$$\bv_1(\mu,\blambda,\speed)<\bv_s(\speed,\blambda)<\bv_2(\mu,\blambda,\speed)<\bv_0(\speed,\blambda)<\bv_3(\mu,\blambda,\speed)<\bv^s(\speed,\blambda)\,,$$
$$\mu=\Potential(\bv_1(\mu,\blambda,\speed);\speed,\blambda)=\Potential(\bv_2(\mu,\blambda,\speed);\speed,\blambda)=\Potential(\bv_3(\mu,\blambda,\speed);\speed,\blambda)\,,$$
$$\forall \bv\in [\bv_{1}(\speed,\blambda),\bv_{s}(\speed,\blambda))\,,\quad \partial_\bv\Potential(\bv;\speed,\blambda)>0\,,$$
for all $\bmu=(\mu,\blambda,\speed)\in \Omega$. 
\end{assumption}
For all $(\speed_*,\blambda_*)\in \Lambda$, we consider 
$$\bmu_{0*}:=(\speed_*,\blambda_*,\mu_0(\speed_*,\blambda_*))\,,\qquad\bmu_{s*}:=(\speed_*,\blambda_*,\mu_s(\speed_*,\blambda_*))\,,$$
which both belong to $\overline{\Omega}$.

\begin{remark}
For expository purposes, Assumption \ref{as:main} here above is formulated so as to deal with both limits, the small amplitude - or harmonic - limit ($\bmu 
\to \bmu_{0*}$) and the soliton limit ($\bmu 
\to \bmu_{s*}$), at the same time. In particular, it supposes the existence of exactly three roots of $\ZZ$ all the way between these two asymptotic regimes, that is for $\mu$ varying from $\mu_0$ to $\mu_s$. We recall that as far as we seek expansions that are by definition specific to one or the other of those regimes, we could localize our assumptions at either $\bmu_{0*}$ or $\bmu_{s*}$ and derive respectively the existence of $(v_2,v_3)$ when $0<\mu-\mu_0\ll 1$ and of $(v_1,v_2,v_3)$ when $0<\mu_s-\mu\ll 1$.
\end{remark} 

\begin{notation}\label{not:superscripts}
From now on, the value of functions of $(\bv;\mu,\blambda,\speed)$ at 
$\bv=\bv_s(\speed,\blambda)$, $\mu=\mu_s(\speed,\blambda)$, respectively $\bv=\bv_0(\speed,\blambda)$, $\mu=\mu_0(\speed,\blambda)$, will be denoted by a superscript $s$, respectively $0$.
For instance,
$$\nabla\ZZ^s:=\nabla\ZZ(\bv_s(\speed,\blambda);\mu_s(\speed,\blambda),\blambda,\speed)\,,
\quad \nabla\ZZ^0:=\nabla\ZZ(\bv_0(\speed,\blambda);\mu_0(\speed,\blambda),\blambda,\speed)\,.$$
The main exception is for the soliton maximal state $\bv^s$, which has its own 
meaning, and the related $\bV^s$, $\bU^s$, $\impulse^s$. Occasionally we shall use similar shorthand superscripts for evaluation at $\bv_1$, $\bv_2$, or $\bv_3$.
\end{notation}

\begin{notation}\label{not:UVWZ}
Recalling that
$$\ZZ(\bv;\mu,\blambda,\speed):=\mu-\Potential(\bv;\speed,\blambda)$$
satisfies either \eqref{eq:derZNtwo} (case $N=2$) or \eqref{eq:derZNone} (case $N=1$),
we can write
$$\nabla\ZZ^s=
\left(\begin{array}{c} 1 \\ \bU_s \\ \impulse_s\end{array}\right)\,,\qquad\quad
\nabla\ZZ^0=
\left(\begin{array}{c} 1 \\ \bU_0 \\ \impulse_0\end{array}\right)\,,
$$
with 
$\bU_s$ being the solitary wave endstate associated with parameters $(\mu_s(\speed,\blambda),\blambda,\speed)$,  
$\bU_0$ being the limiting constant state associated with parameters $(\mu_0(\speed,\blambda),\blambda,\speed)$,
 defined by $$\transp{\bU}_s:=(\bv_s(\speed,\blambda),g(\bv_s(\speed,\blambda);\speed,\lambda_2))\,,\qquad\transp{\bU}_0:=(\bv_0(\speed,\blambda),g(\bv_0(\speed,\blambda);\speed,\lambda_2))\,$$
in the case $N=2$, and merely as $$\bU_s:=\bv_s(\speed,\blambda)\;,\qquad \bU_0:=\bv_0(\speed,\blambda)$$ in the case $N=1$,
and
$$\impulse_s:=\Impulse(\bU_s)\,,\qquad \impulse_0:=\Impulse(\bU_0)\,$$
in both cases.
Furthermore, recalling $\impulse(\bv;\speed,\lambda):=\Impulse(\bv,g(\bv;\speed,\lambda))$ in the case $N=2$ and $\impulse(\bv;\speed,\lambda):=\Impulse(\bv)$ in the case $N=1$, we have for the 
derivatives of $\nabla\ZZ$ with respect to $\bv$ up to second order 
$$
\begin{array}{c}\displaystyle
\nabla\ZZ^s=\bV_s\,,\qquad
\nabla\ZZ^s_\bv=\bW_s\,,\qquad \nabla\ZZ^s_{\bv\bv}=\bZ_s\,,\\\displaystyle
\nabla\ZZ^0=\bV_0\,,\qquad\nabla\ZZ^0_\bv=\bW_0\,,\qquad \nabla\ZZ^0_{\bv\bv}=\bZ_0\,,
\end{array}$$
with for $i=0$ or $s$
\begin{equation}\label{eq:vectstwo}
\bV_i:=\left(\begin{array}{c} 1 \\ \bU_i \\ \impulse_i\end{array}\right)\,,\qquad 
\bW_i:=\left(\begin{array}{c} 0 \\ 1 \\  g_\bv^i \\ \impulse_\bv^i\end{array}\right)\,,
\qquad 
\bZ_i:=\left(\begin{array}{c} 0 \\ 0 \\  g_{\bv\bv}^i \\ \impulse_{\bv\bv}^i\end{array}\right)\,
\end{equation}
in the case $N=2$, and 
\begin{equation}\label{eq:vectsone}
\bV_i:=\left(\begin{array}{c} 1 \\ \bv_i \\ \impulse_i\end{array}\right)\,,\qquad 
\bW_i:=\left(\begin{array}{c} 0 \\ 1  \\ \impulse_\bv^i\end{array}\right)\,,
\qquad 
\bZ_i:=\left(\begin{array}{c} 0 \\   0 \\ \impulse_{\bv\bv}^i\end{array}\right)\,
\end{equation}
in the case $N=1$. 
Likewise we introduce 
\begin{equation}\label{eq:vectsalt}
\bV^s:=\nabla\ZZ(\bv^s(\speed,\blambda);\mu_s(\speed,\blambda),\blambda,\speed)
=\left(\begin{array}{c} 1 \\ \bU^s \\ \impulse^s\end{array}\right)\,,\qquad\quad
\impulse^s:=\Impulse(\bU^s)\,,
\end{equation}
with 
$$\transp{(\bU^s)}:=(\bv^s(\speed,\blambda),g(\bv^s(\speed,\blambda);\speed,\lambda_2))\,,\quad N=2\,,\qquad\qquad
\transp{(\bU^s)}:=\bv^s(\speed,\blambda)\,,\quad N=1\,.$$
For later use, we also notice that the Hessian of $\ZZ$ is identically zero in the case $N=1$, and that in the case $N=2$,  for $i=0$ or $s$ we have
\begin{equation}\label{eq:vectT}
\nabla^2\ZZ^i=-\bT_i\otimes \bT_i\,,\qquad\quad
\bT_i:=\,
\frac{1}{\sqrt{\tau(\bv_i)}}
\left(\begin{array}{c}
0 \\
0 \\
1\\
\bv_i/b
\end{array}\right)\,.
\end{equation}
For consistency's sake, we set $\bT_i=0$, $i=0$ or $s$, when $N=1$.
\end{notation}
For any two vectors $\bV$ and $\bW$ in $\R^d$ thought of as column vectors the notation $\bV\otimes \bW$ stands for the rank-one, square matrix of  size $d$
$$\bV\otimes \bW\,=\,\bV \,\transp{\bW}$$
whatever $d$. Note that its range is spanned by $\bV$ if $\bW$ 
is nonzero - otherwise the matrix is null.

We detail the derivation of \eqref{eq:vectT} in the case $N=2$. Since 
$$\nabla\ZZ\,=\,\left(\begin{array}{c}
1 \\
\bv \\
g(\bv;\speed,\lambda_2)\\
\impulse(\bv;\speed,\lambda_2)
\end{array}\right)
$$
we have
$$\nabla^2\ZZ^i\,=\,\left(\begin{array}{cccc}
0 & 0 & 0 & 0 \\
0 & 0 & 0 & 0 \\
0 & 0 & g_{\lambda_2}^i & g_{\speed}^i\\
0 & 0 & \impulse_{\lambda_2}^i & \impulse_{\speed}^i
\end{array}\right)\,,
$$
and by definition of $g$ and $\impulse$ we see that
$$g_{\lambda_2}^i\,=\,-\frac{1}{\tau(\bv_i)}\,,\qquad\quad
g_\speed^i = \frac{\bv_i\,g_{\lambda_2}^i}{b}\,=\,\impulse_{\lambda_2}^i\,,\qquad\quad
\impulse_\speed^i\,=\,\frac{(\bv_i)^2\,g_{\lambda_2}^i}{b^2}\,,$$
with 
$$\tau(\bv_i)\,=\,\partial_\bu^2\Ec(\bU_i)\,,\qquad\quad
\frac{\bv_i}{b}=\partial_\bu\Impulse(\bU_i)\,.$$

\begin{notation}\label{not:abcp} Coefficients of expansions are to be expressed in terms of
\begin{equation}\label{eq:abcps}
\left\{\begin{array}{l}\frak{a}_s:= \sqrt{-{2}/{\partial_v^2\Potential(\bv_s;\speed;\blambda)}}\,,\qquad\quad
\frak{b}_s:=\,\frac13\,\partial_v^3\Potential(\bv_s;\speed;\blambda)/(\partial_v^2\Potential(\bv_s;\speed;\blambda))^2\,,\\ [10pt]
\frak{c}_s:= \frac{1}{6\sqrt{2}} \left(\frac53 (\partial_v^3\Potential(\bv_s;\speed;\blambda))^2-\partial_v^2\Potential(\bv_s;\speed;\blambda)\partial_v^4\Potential(\bv_s;\speed;\blambda)\right)/(-\partial_v^2\Potential(\bv_s;\speed;\blambda))^{7/2} ,\\ [10pt]
\frak{p}_s:=1/{\partial_v\Potential(\bv^s;\speed;\blambda)}\,,
\end{array}\right.
\end{equation}
\begin{equation}\label{eq:abcp0}
\left\{\begin{array}{l}\frak{a}_0:= \sqrt{{2}/{\partial_v^2\Potential(\bv_0;\speed;\blambda)}}\,,\qquad\quad
\frak{b}_0:=-\,\frac13\,\partial_v^3\Potential(\bv_0;\speed;\blambda)/(\partial_v^2\Potential(\bv_0;\speed;\blambda))^2\,,\\ [10pt]
\frak{c}_0:= \frac{1}{6\sqrt{2}}\,\left(\frac53 (\partial_v^3\Potential(\bv_0;\speed;\blambda))^2-\partial_v^2\Potential(\bv_0;\speed;\blambda)\partial_v^4\Potential(\bv_0;\speed;\blambda)\right)/(\partial_v^2\Potential(\bv_0;\speed;\blambda))^{7/2}\,.
\end{array}\right.
\end{equation}
\end{notation}
We stress differences in signs above and warn the reader that despite their similarities coefficients in \eqref{eq:abcps} and \eqref{eq:abcp0} are not obtained one from the other by exchanging the subscripts $0$ and $s$.

\begin{proposition}\label{prop:asvgradv}  
With notational conventions \ref{not:derivatives}--\ref{not:UVWZ}--\ref{not:abcp}, under Assumption \ref{as:main} we have

$\bullet$ In the small amplitude limit $\bmu \stackrel{\Omega}{\to} \bmu_{0*}$
\begin{equation}\label{eq:asharmv} 
\left\{\begin{array}{l}
\bv_2\,=\,\bv_0-\,\frak{a}_0\,\sqrt{\mu-\mu_0}\,+\, \frak{b}_0\,(\mu-\mu_0)\,-\,\frak{c}_0 \,(\mu-\mu_0)^{3/2}\,+\,{\mathcal O}(\mu-\mu_0)^2\,,\\
\bv_3\,=\,\bv_0+\frak{a}_0\,\sqrt{\mu-\mu_0}\,+\, \frak{b}_0\,(\mu-\mu_0)\,+\,\frak{c}_0 \,(\mu-\mu_0)^{3/2}\,+\,{\mathcal O}(\mu-\mu_0)^2\,,
\end{array}\right.
\end{equation}
\begin{equation}\label{eq:asharmgradv}
\left\{\begin{array}{lll}\displaystyle
\nabla\bv_2=-\frac{\frak{a}_0}{2\sqrt{\mu-\mu_0}}\,\bV_0\,&\displaystyle+\,\frak{b}_0\,\bV_0\,+\,\frac{\frak{a}_0^2}{2}\,\bW_0\,& \\  [10pt]
 & \displaystyle -\, \Big(\frac32 \frak{c}_0\, \bV_0 +\frac32\frak{a}_0\frak{b}_0\, \bW_0 + \frac{\frak{a}_0^3}{4}\, \bZ_0\Big) \,\sqrt{\mu-\mu_0}\,& +\,{\mathcal O}(\mu-\mu_0)\,,\\
\displaystyle\nabla\bv_3=\,\quad\frac{\frak{a}_0}{2\sqrt{\mu-\mu_0}}\,\bV_0\,&\displaystyle+\,\frak{b}_0\,\bV_0\,+\,\frac{\frak{a}_0^2}{2}\,\bW_0\,& \\  [10pt]
 & \displaystyle +\, \Big(\frac32 \frak{c}_0\, \bV_0 +\frac32\frak{a}_0\frak{b}_0\, \bW_0 + \frac{\frak{a}_0^3}{4}\, \bZ_0\Big) \,\sqrt{\mu-\mu_0}\,& +\,{\mathcal O}(\mu-\mu_0)\,.\\ [15pt]
\end{array}
\right.
\end{equation}
$\bullet$ In the soliton limit $\bmu \stackrel{\Omega}{\to} \bmu_{s*}$
\begin{equation}\label{eq:assolv}
\left\{\begin{array}{lll}\bv_1=\bv_s\,-\,\frak{a}_s\,\sqrt{\mu_s-\mu}\,&+\, \frak{b}_s\,(\mu_s-\mu)\,-\,\frak{c}_s \,(\mu_s-\mu)^{3/2}\,&+\,{\mathcal O}(\mu_s-\mu)^2\,,\\[5pt]
\bv_2=\bv_s\,+\,\frak{a}_s\,\sqrt{\mu_s-\mu}\,&+\, \frak{b}_s\,(\mu_s-\mu)\,+\,\frak{c}_s \,(\mu_s-\mu)^{3/2}\,&+\,{\mathcal O}(\mu_s-\mu)^2\,,\\[5pt]
\bv_3=\bv^s\,&-\, \frak{p}_s\,(\mu_s-\mu)\,&+\,{\mathcal O}(\mu_s-\mu)^2\,,
\end{array}\right.
\end{equation}
\begin{equation}\label{eq:assolgradv}
\left\{\begin{array}{lll}\displaystyle
\nabla\bv_1=\,\frac{\frak{a}_s}{2\sqrt{\mu_s-\mu}}\,\bV_s\,&\displaystyle-\,\frak{b}_s\,\bV_s\,-\,\frac{\frak{a}_s^2}{2}\,\bW_s\,& \\  [10pt]
 & \displaystyle +\, \Big(\frac32 \frak{c}_s\, \bV_s +\frac32\frak{a}_s\frak{b}_s\, \bW_s + \frac{\frak{a}_s^3}{4}\, \bZ_s\Big) \,\sqrt{\mu_s-\mu}\,& +\,{\mathcal O}(\mu_s-\mu)\,,\\ [15pt]
\displaystyle
\nabla\bv_2=-\frac{\frak{a}_s}{2\sqrt{\mu_s-\mu}}\,\bV_s\,&\displaystyle-\,\frak{b}_s\,\bV_s\,-\,\frac{\frak{a}_s^2}{2}\,\bW_s\,& \\  [10pt]
 & \displaystyle -\, \Big(\frac32 \frak{c}_s\, \bV_s +\frac32\frak{a}_s\frak{b}_s\, \bW_s + \frac{\frak{a}_s^3}{4}\, \bZ_s\Big) \,\sqrt{\mu_s-\mu}\,& +\,{\mathcal O}(\mu_s-\mu)\,,\\ [15pt]
\displaystyle
\nabla\bv_3=\,&\displaystyle\,\frak{p}_s\,\bV^s\,&+\,{\mathcal O}(\mu_s-\mu)\,,
\end{array}
\right.
\end{equation}
\end{proposition}

\begin{proof}[Proof of Proposition~\ref{prop:asvgradv}]
The asymptotic expansions of the zeros of $\ZZ(\cdot;\mu,\blambda,\speed)=\mu-\Potential(\cdot;\speed;\blambda)$ follow from Lemma~\ref{lem:as} in Appendix \ref{app:elas}.
Applied to $$W(z)=\Potential(z+\bv_0(\speed,\blambda);\speed;\blambda)-\Potential(\bv_0(\speed,\blambda);\speed;\blambda)$$ and $\varepsilon=\mu-\mu_0(\speed,\blambda)$, this lemma gives 
\eqref{eq:asharmv} and 
$$
\begin{array}{lllllllll}\displaystyle
\partial_\mu\bv_{2}&=&\displaystyle
-\,\frac{\frak{a}_0}{2\sqrt{\mu-\mu_0}}
&+&\displaystyle\frak{b}_0
&-&\displaystyle
\frac32 \frak{c}_0 \,\sqrt{\mu-\mu_0}
&+&\displaystyle
{\mathcal O}(\mu-\mu_0)\,,\\\displaystyle
\partial_\mu\bv_{3}&=&\displaystyle
\quad\frac{\frak{a}_0}{2\sqrt{\mu-\mu_0}}
&+&\displaystyle\frak{b}_0
&+&\displaystyle
\frac32 \frak{c}_0 \,\sqrt{\mu-\mu_0}
&+&\displaystyle
{\mathcal O}(\mu-\mu_0)\,,
\end{array}
$$
where coefficients $\frak{a}_0$, $\frak{b}_0$, and $\frak{c}_0$ are precisely as defined in Notation \ref{not:abcp}. %
Similarly, Lemma~\ref{lem:as}$)$ applied to 
$$W(z)=\Potential(\bv_s(\speed,\blambda);\speed;\blambda)-\Potential(z+\bv_s(\speed,\blambda);\speed;\blambda)$$ and $\varepsilon=\mu_s(\speed,\blambda)-\mu$,  gives 
the first two rows in \eqref{eq:assolv} and
$$
\begin{array}{lllllllll}\displaystyle
\partial_\mu\bv_{1}&=&\displaystyle
\quad\frac{\frak{a}_s}{2\sqrt{\mu_s-\mu}}
&-&\displaystyle\frak{b}_s
&+&\displaystyle
\frac32 \frak{c}_s \,\sqrt{\mu_s-\mu}
&+&\displaystyle{\mathcal O}(\mu_s-\mu)\,,\\\displaystyle
\partial_\mu\bv_{2}
&=&\displaystyle
-\frac{\frak{a}_s}{2\sqrt{\mu_s-\mu}}
&-&\displaystyle\frak{b}_s
&-&\displaystyle
\frac32 \frak{c}_s \,\sqrt{\mu_s-\mu}
&+&{\mathcal O}(\mu_s-\mu)\,.
\end{array}
$$

The third row in  \eqref{eq:assolv} merely follows from the fact that $0$ is a simple root of $$\Potential(\bv_s(\speed,\blambda);\speed;\blambda)-\Potential(z+\bv_s(\speed,\blambda);\speed;\blambda)-\mu_s(\speed,\blambda)+\mu\,.$$
By differentiating with respect to $\mu$ the identity $\Potential(\bv_3(\mu,\blambda, \speed);\speed;\blambda)=\mu$ we also readily find that 
$$\partial_\mu\bv_3\,=\,
\frak{p}_s\,+\,{\mathcal O}(\mu_s-\mu)\,.$$

The asymptotic expansions of other derivatives  of the $v_i$'s can be inferred from expansions of $\partial_\mu\bv_{i}$ by differentiating the identity
$\ZZ(\bv_i;\mu,\blambda,\speed)=0$. This gives
$$\nabla\bv_i\,=\,-\,\nabla\ZZ^i/\ZZ_\bv^i\,,\qquad\quad\mbox{in particular }\quad\bv_{i,\mu}\,=\,-\,1/\ZZ_\bv^i\,,$$
so that 
by Taylor-expanding these derivatives about $\bv_0$, $\bv_s$ and $\bv^s$ as
$$
\begin{array}{lllllllll}
\nabla\ZZ&=&\bV^s&+&{\mathcal O}(\bv-\bv^s)\,,&&&&\\
\nabla\ZZ&=&\bV_l&+&(\bv-\bv_l)\,\bW_l&+&\frac12 \,(\bv-\bv_l)^2\,\bZ_l&+&{\mathcal O}(\bv-\bv_l)^3\,,
\end{array}$$
for $l=0$ or $s$, and  collecting all the terms we eventually obtain \eqref{eq:asharmgradv} and \eqref{eq:assolgradv}.
\end{proof}

\subsubsection{Asymptotics of action and period}\label{ss:actionperiod}
We concentrate now on the asymptotics of the action $\Action$ and of the spatial period $\Xi$, as a warm-up for the expansion of the other derivatives of the action, which will be derived in \S~\ref{ss:ashessaction}.

\begin{proposition} \label{prop:asactionperiod} 
With notational conventions \ref{not:UVWZ}--\ref{not:abcp}, under Assumption \ref{as:main} we have

$\bullet$ In the small amplitude limit $\bmu \stackrel{\Omega}{\to} \bmu_{0*}$
$$\begin{array}{ll}\Xi\,&=\, 
\Xi_0\,\,+\,{\mathcal O}(\mu-\mu_0)\,,\qquad\quad
\Xi_0:=\pi \,\frak{a}_0\,\sqrt{2\cap(\bv_0)}\,,\\ [10pt]
\Action \,& =\,\Xi_0\;(\mu-\mu_0)\,+\,{\mathcal O}(\mu-\mu_0)^2\,. \end{array}$$

$\bullet$ In the soliton limit $\bmu \stackrel{\Omega}{\to} \bmu_{s*}$
$$\begin{array}{ll}
\Xi&=\,-\,\frak{a}_s\,\sqrt{\cap(\bv_s)/2}\,
\,\ln(\mu_s-\mu)\,+\,{\mathcal O}(1)\\ [10pt]
\Action&=\,\Mom
\,+\,\frak{a}_s\,\sqrt{\cap(\bv_s)/2}\,(\mu_s-\mu)\,\ln(\mu_s-\mu)\,+\,{\mathcal O}(\mu_s-\mu)
\end{array}
$$
with
$$ \Mom
= 2\,\int_{\bv_s}^{\bv^s} {\sqrt{2\cap(\bv)\,\ZZ(\bv;\mu_s,\blambda,\speed)}}\,\dif \bv\,.$$
\end{proposition}

\begin{remark}\label{rk:def-Mom}
As in Section \ref{ss:waveparam},  
$\Mom=\Mom(\speed,\bU_s)$ stands for the Boussinesq moment of instability associated with the solitary wave 
of speed $\speed$ and endstate $\bU_s$. In terms of the profile $\ubU^s$ of this solitary wave, 
it equivalently reads
$$\Mom(\speed,\bU_s)=\int_{-\infty}^{+\infty} (\Ham[\ubU^s]+\speed \Impulse(\ubU^s) + \blambda \cdot \ubU^s +\mu_s)\,\dif \xi\,,$$
where 
$$\blambda = - \nabla_{\bU}(\Ham+\speed\Impulse)(\bU_s,0)\,,\qquad\quad
\mu_s = - \blambda\cdot \bU_s\,-\,(\Ham+\speed\Impulse)(\bU_s,0)\,,$$
that is,
$$\Mom(\speed,\bU_s)=
\int_{-\infty}^{+\infty} \big((\Ham+\speed \Impulse)(\ubU^s,0) - (\Ham+\speed \Impulse)(\bU_s,0) - \nabla_{\bU}(\Ham+\speed\Impulse)(\bU_s,0)\cdot (\ubU^s-\bU_s)\big)\,\dif \xi\,.$$
\end{remark}

\begin{remark}
The factor 
$$\frak{a}_s\,\sqrt{\cap(\bv_s)/2}\,=\,\sqrt{\frac{-\cap(\bv_s)}{\partial_\bv^2\Potential(\bv_s;\speed,\blambda)}}$$
 in the leading order term in the expansion of $\Xi$ in the soliton limit is, up to a factor $2\pi$, 
the small amplitude limit of the period of waves 
$\widetilde{\Xi}$ associated with the \emph{opposite capillarity coefficient} $-\kappa$, and whose profile oscillates between $\bv_1$ and $\bv_2$. This symmetry has led some authors to analyze the soliton limit in terms of the small parameter $2\pi/\widetilde{\Xi}$, that has been coined as the \emph{conjugate wave number}. See for instance \cite[Formula~(54)]{El1} and the surrounding discussion. This is particularly convenient for leading-order asymptotics but becomes akward as soon as expansions starts involving $\bv^s$.
\end{remark}

\begin{proof}[Proof of Proposition~\ref{prop:asactionperiod}]
Let us deal with small amplitude asymptotics first. We first infer from expansions of $\bv_2$ and $\bv_3$ in \eqref{eq:asharmv} that the function 
$$\V=\V(\theta;\bv_2,\bv_3)=\bv_2+(1+\sin\theta) (\bv_3-\bv_2)/2$$ 
defined in \eqref{eq:V} expands as
$$\V= \bv_0\,+\,
(\sin\theta)\; \frak{a}_0\,\sqrt{\mu-\mu_0}\,+\,\frak{b}_0\,(\mu-\mu_0) \,+\,(\sin\theta)\;\frak{c}_0 \,(\mu-\mu_0)^{3/2}\,+\,{\mathcal O}(\mu-\mu_0)^2\,.
$$
Together with \eqref{eq:asharmv} , this enables us to expand 
$$\begin{array}[l]{l}\sqrt{\cap(\V(\theta;\bv_2,\bv_3))\Redpot(\V(\theta;\bv_2,\bv_3),\bv_2,\bv_3;\speed,\blambda)}\\ [10pt]
\ \ =
(\sqrt{\cap\Redpot})^0
+(\sqrt{\cap\Redpot})_\bv^0\,\frak{a}_0\,\sqrt{\mu-\mu_0}\,\sin\theta
+((\sqrt{\cap\Redpot})_\bz^0-(\sqrt{\cap\Redpot})_\bw^0)\,\frak{a}_0\,\sqrt{\mu-\mu_0}+{\mathcal O}(\mu-\mu_0)
\end{array}
$$
where the superscript $0$ stands for evaluation at $\bv=\bv_0$, $\bw=\bv_0$, $\bz=\bv_0$
by extension of the convention adopted in Notation \ref{not:superscripts}.
 
Since $\Redpot$ - as defined in \eqref{eq:redpot} - is symmetric in $\bw$ and $\bz$, $\sqrt{\cap \Redpot}$ inherits this symmetry and therefore 
$$(\sqrt{\cap\Redpot})_\bz^0=(\sqrt{\cap\Redpot})_\bw^0\,.$$
Incidentally, note that $\Redpot$ is actually a symmetric function of its first three arguments, see Lemma~\ref{lem:symmetry} in Appendix~\ref{app:alg}. Furthermore, since $\int_{-\pi/2}^{\pi/2} \sin\theta \cos^2\theta\;\dif\theta=0$ the other ${\mathcal O}(\sqrt{\mu-\mu_0})$ term cancels out upon integration of $\sqrt{\cap\Redpot}$ over $(-\pi/2,\pi/2)$. Noticing in addition that
$\int_{-\pi/2}^{\pi/2} \cos^2\theta\;\dif\theta=\pi/2$ and 
$$
\Redpot(\bv_0,\bv_0,\bv_0;\speed,\blambda)=\partial_\bv^2\Potential(\bv_0;\speed,\blambda)/2\,,$$ 
we thus obtain the simple expansion
$$
\begin{array}{l}\displaystyle
\int_{-\pi/2}^{\pi/2} {\sqrt{\cap(\V(\theta;\bv_2
,\bv_3
))\Redpot(\V(\theta;\bv_2
,\bv_3
),\bv_2
,\bv_3
;\speed,\blambda)}}\,\cos^2\theta \;\dif \theta\\\displaystyle
\qquad\qquad\qquad\qquad\,=\,\frac{\pi}{2}\,\sqrt{\cap(\bv_0)\,\partial_\bv^2\Potential(\bv_0;\speed,\blambda)/2}\,+\,{\mathcal O}(\mu-\mu_0)\,.
\end{array}$$
Hence, by combining with an expansion of $(\bv_3-\bv_2)^2$, as expected
$$\Action=2 {\pi}\,\frak{a}_0^2\;\sqrt{\cap(\bv_0)\,\partial_\bv^2\Potential(\bv_0;\speed,\blambda)/2}\;
(\mu-\mu_0)\,+\,{\mathcal O}(\mu-\mu_0)^2\,.$$
In a similar way, since $\int_{-\pi/2}^{\pi/2} \sin\theta \;\dif\theta=0$ we find that for any smooth function
$\phi$,
$$\int_{-\pi/2}^{\pi/2}\frac{\phi(\V(\theta;\bv_2
,\bv_3
))\,\dif \theta}{\sqrt{\Redpot(\V(\theta;\bv_2
,\bv_3
),\bv_2
,\bv_3
;\speed,\blambda)}}\,=\,\pi\,\frac{\phi(\bv_0)}{\sqrt{\partial_\bv^2\Potential(\bv_0;\speed,\blambda)/2}}\,+\,{\mathcal O}(\mu-\mu_0)\,.$$
In particular for $\phi(\bv)=\sqrt{2\cap(\bv)}$ this gives that
$$\Xi=\Xi_0\,+\,{\mathcal O}(\mu-\mu_0)\,,\qquad\quad
\Xi_0:=2\pi\,\sqrt{\cap(\bv_0)/\partial_\bv^2\Potential(\bv_0;\speed,\blambda)}\,=\,
\pi \, \frak{a}_0\,\sqrt{2\cap(\bv_0)}\,.$$

Let us now turn to the soliton limit. This asymptotic regime is trickier. As a warm-up and a preparation for higher-order expansions, we already introduce now, to handle it, pieces of notation that will be most relevant later.  
We recall that we can write the action integral as
$$\Action=\,2\,(\bv_3-\bv_2)^2\int_{0}^{1} {\sqrt{-2\cap(\V)\Redpot(\V,\bv_1,\bv_2;\speed,\blambda)}}\,\sqrt{\sigma\,(\sigma+\rho)} \;\dif \sigma\,,$$
and for any smooth function $\phi$,
$$\Iphi=
\int_{0}^{1}\frac{\phi(\V)}{\sqrt{
-\Redpot(\V,\bv_1,\bv_2;\speed,\blambda)
}}\,\frac{\dif \sigma}{\sqrt{\sigma(\sigma+\rho)}}\,,\qquad\quad
\rho=\frac{\bv_2-\bv_1}{\bv_3-\bv_2}\,,$$
where $\V$ is redefined here as 
$$\V(\sigma;\bv_2,\bv_3):=\bv_2+\sigma(\bv_3-\bv_2)$$
- actually, we do not change the dependent variable $\V$ itself, but we view it as a function of $\sigma=\sin \theta$ instead of $\theta$.
In particular, the period of the wave corresponds to $\phi(\bv)=\sqrt{2\kappa(\bv)}$, that is,
$$\Xi=I(\sqrt{2\kappa})=
\int_{0}^{1}\,\Y(\V,\bv_1,\bv_2;\speed,\blambda)\;\frac{\dif \sigma}{\sqrt{\sigma(\sigma+\rho)}}\,,$$
where $\Y$ is defined by
\begin{equation}\label{eq:defY}
\Y(\bv,\bw,\bz;\speed,\blambda)\,:=\,\sqrt{2 \cap(\bv)/|\Redpot(\bv,\bw,\bz;\speed,\blambda)|}\,.
\end{equation}
The absolute value in \eqref{eq:defY} is inserted here to allow us the use of the notation $\Y$ in both the soliton limit - with $\bw=\bv_1$ and $\bz=\bv_2$ - and the small amplitude limit - with $\bw=\bv_2$ and $\bz=\bv_3$. 
Recall that by construction of $\Redpot$ and Assumption \ref{as:main} we have $\Redpot(\bv,\bv_1,\bv_2;\speed,\blambda)<0$ hence
$$\Y(\bv,\bv_1,\bv_2;\speed,\blambda)\,=\,\sqrt{-2 \cap(\bv)/\Redpot(\bv,\bv_1,\bv_2;\speed,\blambda)}$$
for all $\bv\in [\bv_2,\bv_3)$. 

Since
$$\ZZ(\bv;\mu,\blambda,\speed)\,=\,\mu-\Potential(\bv;\speed,\blambda)=-(\bv- \bv_2)(\bv-\bv_1) \Redpot(\bv,\bv_1,\bv_2;\speed,\blambda)$$
for all $\bv$, observe that by definition of $\bv_3$ as another root of $\ZZ=\mu-\Potential$ we have
$$\Redpot(\bv_3,\bv_1,\bv_2;\speed,\blambda)\,=\,0\,.$$
Thus some care is still needed with this endpoint. More, at some point it will become important to also desingularize it. To do so, we proceed as follows. First, we can write
$$\Redpot(\bv,\bv_1,\bv_2;\speed,\blambda)\,=\,(\bv-\bv_3)\,\Redpotc(\bv,\bv_1,\bv_2,\bv_3;\speed,\blambda)\,,$$
with
\begin{equation}\label{eq:redpotc}
\begin{array}[t]{l}\Redpotc(\bv,\bw,\bz,\bzc;\speed,\blambda)\\ [10pt]
\displaystyle
\qquad\quad
:=\,\int_{0}^{1}\!\!\!\int_{0}^{1}\!\!\!\int_{0}^{1}t^2s\,\partial_\bv^3\Potential(\bw+t(\bz-\bw)+ts(\bzc-\bz)+tsr \,(\bv-\bzc);\speed,\blambda)\,\dif r\dif s\dif t\,,
\end{array}
\end{equation}
which gives the factorization
$$
\ZZ(\bv;\mu,\blambda,\speed)=(\bv- \bv_1)(\bv-\bv_2) (\bv_3-\bv) \Redpotc(\bv,\bv_1,\bv_2,\bv_3;\speed,\blambda)\,.$$
By Assumption \ref{as:main} again, we have 
$$\Redpotc(\bv,\bv_1,\bv_2,\bv_3;\speed,\blambda)\,\neq\,0$$
for all $\bv \in [\bv_2,\bv_3]$.
Accordingly, we define
$$\Yc(\bv,\bw,\bz,\bzc;\speed,\blambda):=\sqrt{\frac{\bzc-\bv}{\bzc-\bz}}\;\Y(\bv,w,z;\speed,\blambda)\,=\,
\frac{1}{\sqrt{\bzc-\bz}}\,\sqrt{\frac{2 \cap(\bv)}{|\Redpotc(\bv,\bw,\bz,\bzc;\speed,\blambda)|}}\,,$$
which is a smooth function of $\bv \in [\bv_2,\bv_3]$ such that
$$\Y(\bv_2+\sigma (\bv_3-\bv_2),\bv_1,\bv_2;\speed,\blambda)\,=\,\frac{1}{\sqrt{1-\sigma}}\,
\Yc(\bv_2+\sigma (\bv_3-\bv_2),\bv_1,\bv_2,\bv_3;\speed,\blambda)\,$$
for all $\sigma \in [0,1)$.
At some point we shall use the desingularized function $\Yc$ instead of $\Y$ in our asymptotic expansions.

To begin expansions in the soliton limit, we derive from \eqref{eq:assolv}
$$\rho=\frac{\bv_2-\bv_1}{\bv_3-\bv_2}\,=\,\frac{2\frak{a}_s}{\bv^s-\bv_s}\,\sqrt{\mu_s-\mu}\,+\,\frac{2\frak{a}_s^2}{(\bv^s-\bv_s)^2}\,({\mu_s-\mu})\,+\,{\mathcal O}(\mu_s-\mu)^{3/2}\,,$$
and
$$\V\,=\,\V_s\,+\,
(1-\sigma)\; \frak{a}_s\,\sqrt{\mu_s-\mu}\,
+\,((1-\sigma)\;\frak{b}_s-\sigma\;\frak{p}_s)
\,(\mu_s-\mu) \,+\,{\mathcal O}(\mu_s-\mu)^{3/2}\,,
$$
with 
$$\V_s=\V_s(\sigma):=\bv_s+\sigma (\bv^s-\bv_s)\,.$$ 
By taking the limit under the integral sign we 
infer readily 
that $\Action$ goes to
$$2\,(\bv^s-\bv_s)^2\int_{0}^{1} {\sqrt{-2\cap(\V_s(\sigma))\Redpot(\V_s(\sigma),\bv_s,\bv_s
;\speed,\blambda)}}\;\sigma \;\dif \sigma\,,$$
which 
is seen to coincide with the Boussinesq momentum $\Mom$ by the change of variable $v=\V_s(\sigma)$.

To derive the full expansion of $\Action$, observe that half 
the difference reads
$$\begin{array}[t]{l} 
\frac12(\Action-\Mom)= \displaystyle
\,(\bv^s-\bv_s)^2\,\int_{0}^{1} \Ss(\V_s,\bv_s,\bv_s;\speed,\blambda)\;({\sqrt{\sigma(\sigma+\rho)}}-\sigma) \;\dif \sigma
\\ [10pt] \displaystyle
\qquad\qquad+
\int_{0}^{1}\left((\bv_3-\bv_2)^2\,\Ss(\V,\bv_1,\bv_2;\speed,\blambda)
 -(\bv^s-\bv_s)^2\Ss(\V_s,\bv_s,\bv_s;\speed,\blambda)\right)\;{\sqrt{\sigma(\sigma+\rho)}} \;\dif \sigma
\end{array}$$
where 
$$ \Ss(v,w,z;\speed,\blambda):=\,\sqrt{2 \cap(\bv) |\Redpot(\bv,w,z;\speed,\blambda)|}\,=\,|\Redpot(\bv,w,z;\speed,\blambda)|\,
 \Y(v,w,z;\speed,\blambda)\,.$$
From the expansions here above we have
$$
\begin{array}{l}\displaystyle
(\bv_3-\bv_2)^2\,\Ss(\V,\bv_1,\bv_2;\speed,\blambda)\,-\,(\bv^s-\bv_s)^2\Ss(\V_s,\bv_s,\bv_s;\speed,\blambda)\\[10pt]
\displaystyle\,=\,
\frak{a}_s \,(\bv^s-\bv_s)\, \Big(-2 \Ss^s\,+\,(\bv^s-\bv_s)\,\big((1-\sigma) \Ss^s_{\bv}\,-\,\Ss^s_{w}\,+\,\Ss^s_{z}\big)\Big)\,\sqrt{\mu_s-\mu} \,+\,{\mathcal O}(\mu_s-\mu)\,,
\end{array}
$$
with 
 $$\Ss^s:=\Ss(\V_s,\bv_s,\bv_s;\speed,\blambda)$$
and similar notation for the partial derivatives of $\Ss$ with respect to $\bv$, $w$, and $z$. 
We point out that, by definition, $\V_s$ depends on $\sigma$, and so does $\Ss^s$ through $\V_s$. In addition, the derivatives $\Ss^s_{w}$ and $\Ss^s_{z}$ cancel out in the expansion above by symmetry of $\Redpot$ and thus also of $\Ss$ with respect to $w$ and $z$.
  
\begin{remark}\label{rem:superscripts}
The meaning of the superscript $s$ differs slightly here from Notation \ref{not:superscripts} since we take $\bv=\V_s$  and not just $\bv=\bv_s$. The case $\bv=\bv_s$ actually corresponds to $\sigma=0$ in $\V_s$ and will be signaled by an additional subscript $0$. For instance 
$$\Ss^s_0:=\Ss(\bv_s,\bv_s,\bv_s;\speed,\blambda)\,.$$
The subscript $0$ here above, standing for $\sigma=0$, should not be confused with notation related to the state $\bv_0$. As the latter state is irrelevant in the soliton limit, we hope that this will not be too confusing. We adopt the same convention for all functions of $(\bv,\bw,\bz;\speed,\blambda)$, such as $\Y$, $\Redpot$ and their derivatives.
\end{remark}
 
By Proposition~\ref{prop:asol} we thus find that 
$$\begin{array}[t]{l} \displaystyle
\int_{0}^{1}\left((\bv_3-\bv_2)^2\,\Ss(\V,\bv_1,\bv_2;\speed,\blambda)
 -(\bv^s-\bv_s)^2\Ss(\V_s,\bv_s,\bv_s;\speed,\blambda)\right)\;{\sqrt{\sigma(\sigma+\rho)}} \;\dif \sigma\\ [10pt]
\displaystyle
\qquad\quad\,=\, \frak{a}_s \,(\bv^s-\bv_s)\, \sqrt{\mu_s-\mu} \,
\int_{0}^{1}\,\Big(-2 \Ss^s\,+\,(\bv^s-\bv_s)\,(1-\sigma) \Ss^s_{\bv}\Big)\,\sigma\,\dif \sigma 
\,+\,{\mathcal O}(\mu_s-\mu)\,,
\end{array}$$
and
$$\begin{array}[t]{l}\displaystyle
\int_{0}^{1}  \Ss(\V_s,\bv_s,\bv_s;\speed,\blambda)\;({\sqrt{\sigma(\sigma+\rho)}}-\sigma) \;\dif \sigma\\[10pt]
\displaystyle
\qquad\quad
=\frac{\rho}{2}\,\int_{0}^{1}\Ss^s\,\dif \sigma \,+\,\frac{\rho^2\ln\rho}{8}\,\Ss^s_0\,+\,{\mathcal O}(\rho^2)\\[10pt] 
\displaystyle
\qquad\quad=\frac{\frak{a}_s}{\bv^s-\bv_s}\,\sqrt{\mu_s-\mu}\,\int_{0}^{1}\Ss^s\,\dif \sigma\,+\,
\frac{\frak{a}_s^2}{4(\bv^s-\bv_s)^2}\,(\mu_s-\mu)\,\ln(\mu_s-\mu)\,\Ss^s_0\,+\,{\mathcal O}(\mu_s-\mu)\,
\end{array}
$$
with
$$\Ss^s_0=\Ss(\bv_s,\bv_s,\bv_s;\speed,\blambda)$$
being defined as in Remark \ref{rem:superscripts}.

Therefore, we have 
$$\begin{array}[t]{l} 
\frac12(\Action-\Mom)\\ \displaystyle
\qquad=\frak{a}_s \,(\bv^s-\bv_s)\, \sqrt{\mu_s-\mu} \,
\int_{0}^{1}\,\Big((1-2\sigma) \Ss^s\,+\,(\bv^s-\bv_s)\sigma\,(1-\sigma) \Ss^s_{\bv}\Big)\,\dif \sigma \\ [10pt]\displaystyle
\qquad\quad\,+\,
\frac14\,\Ss^s_0 \frak{a}_s^2\,(\mu_s-\mu)\,\ln(\mu_s-\mu)\,\,+\,{\mathcal O}(\mu_s-\mu)
\,.
\end{array}$$
We notice  that 
$$\Ss^s_0 \frak{a}_s^2\,=\,\frak{a}_s\,\sqrt{2\cap(\bv_s)}\,.
$$
Most importantly, the integral in factor of $\sqrt{\mu_s-\mu}$ is zero, since it reads
$$\int_{0}^{1}\,\partial_\sigma \Big( \sigma \,(1-\sigma) \Ss(\bv_s+\sigma (\bv^s-\bv_s), \bv_s,\bv_s;\speed,\blambda) \Big)\,\dif \sigma\,.$$
This eventually implies that 
$$\Action \,=\,\Mom\,+\,\frak{a}_s\,\sqrt{\cap(\bv_s)/2}\,(\mu_s-\mu)\,\ln(\mu_s-\mu)\,+\,{\mathcal O}(\mu_s-\mu)\,,
$$
as claimed.

The expansion of $\Xi$ is also obtained by applying Proposition~\ref{prop:asol}, first observing that
$$
\Y\,=\,
\Y^s\,+\,{\mathcal O}(\rho)\,.
$$
Then, we have by Proposition~\ref{prop:asol}
$$\Xi \begin{array}[t]{rcl} 
&=&\displaystyle  
\int_{0}^{1}\,\frac{\Y\;\dif \sigma}{\sqrt{\sigma(\sigma+\rho)}}
\,=\,
\int_{0}^{1}\,\frac{\Y^s\;\dif \sigma}{\sqrt{\sigma(\sigma+\rho)}}\,+\,{\mathcal O}(\rho\,\ln\rho)
\\[15pt]
&=&\displaystyle   - \Y^s_0 \,\ln \rho\,+\,
{\mathcal O}(1)\,.
\end{array}$$ 
In terms of $\mu_s-\mu$, the expansion of $\Xi$  thus reads
$$\Xi\,=\,- \frac12 \Y^s_0 \,\ln (\mu_s-\mu)\,+\,{\mathcal O}(1)\,,
$$
as expected.
\end{proof}

\subsection{Asymptotics of the Hessian of the action}\label{ss:ashessaction}
As has been shown in earlier work (see \emph{e.g.}~\cite{BronskiJohnson,BronskiJohnsonKapitula,Johnson,BNR-JNLS14,BNR-GDR-AEDP,BMR}), the Hessian $\nabla^2\Action$ of the action plays a major role in the understanding of the stability of periodic waves and their modulations. 

The aim of this section is to derive asymptotic expansions for $\nabla^2\Action$ in both the small amplitude limit  and the solitary wave limit. The starting point is the expression of the gradient of $\Action$ in \eqref{eq:gradaction}, which we repeat here for convenience
$$\nabla\Action(\mu,\blambda,\speed)\,=\,\int_{\bv_2(\mu,\blambda,\speed)}^{\bv_3(\mu,\blambda,\speed)} \nabla \ZZ(\bv;\mu,\blambda,\speed)\,\sqrt{\frac{2\kappa(\bv)}{\ZZ(\bv;\mu,\blambda,\speed)}}\,\dif\bv\,.$$
As in Section \ref{s:asbasic}, we can use changes of variables that `desingularize' the denominator $\sqrt{\ZZ}$. 
This would be in fact required if we were to differentiate $\nabla\Action$ under the integral sign. 
We will see that it is more efficient to expand $\nabla\Action$ and differentiate the expansion term by term rather than compute $\nabla^2\Action$ and then expand numerous terms arising.

The choice of the desingularizing change of variables must be adapted to the limit under consideration. 
As a result, the treatment of the two limits being quite different, we deal with them separately.

\subsubsection{Small amplitude limit}
The starting point of the small amplitude limit expansion of $\nabla^2\Action$ consists in rewriting
$$\nabla\Action\,=\,\int_{-\pi/2}^{\pi/2} \f \,\dif \theta\,,\qquad\quad \mbox{where } \f:=\Y\,\nabla\ZZ$$
is evaluated at
$v=\V(\theta;\bv_2,\bv_3)=\bv_2+(1+\sin\theta) (\bv_3-\bv_2)/2=m+\delta\,\sin\theta$, with 
\begin{equation}
\label{eq:defmdelta}
m:=\frac{\bv_2+\bv_3}{2}\,,\qquad\qquad 
\delta:=\frac{\bv_3-\bv_3}{2}\,.
\end{equation}

\begin{proposition}\label{prop:assmall}
Under Assumption \ref{as:main} we have the following
$$\begin{array}{rcl}
m&=&\bv_0\,+\,\frak{b}_0\,(\mu-\mu_0)\,+\,{\mathcal O}(\mu-\mu_0)^2\,,\\[10pt]
\delta&=&\frak{a}_0\,\sqrt{\mu-\mu_0}\,+\,\frak{c}_0\,(\mu-\mu_0)^{3/2}\,+\,{\mathcal O}(\mu-\mu_0)^{5/2}
\end{array}
$$ 
in the small amplitude limit, where coefficients 
 $\frak{a}_0$, $\frak{b}_0$ and $\frak{c}_0$ are those defined in \eqref{eq:abcp0} (Notation \ref{not:abcp}). In particular, we can use $\delta$ as a small parameter to describe the small amplitude limit and obtain 
$$
\frac{1}{\sqrt{\mu-\mu_0}}\,=\,\frak{a}_0\left(\frac1\delta+\frac{\frak{c}_0}{\frak{a}_0^3}\delta\right)\,+\,{\mathcal O}(\delta^3)\,,\qquad 
\sqrt{\mu-\mu_0}\,=\,\frac{1}{\frak{a}_0}\left(\delta-\frac{\frak{c}_0}{\frak{a}_0^3}\delta^3\right)\,+\,{\mathcal O}(\delta^5)\,,
$$
and
\begin{equation}\label{eq:devmgradm}
m\,=\, v_0\,+\,\frac{\frak{b}_0}{\frak{a}_0^2}\delta^2\,+\,{\mathcal O}(\delta^4)\,,
\qquad\quad
\nabla m\,=\,\frak{b}_0\,\bV_0\,+\,\frac{\frak{a}_0^2}{2}\,\bW_0\,+\,{\mathcal O}(\delta^2)\,,
\end{equation}
\begin{equation}\label{eq:devgraddelta}
\nabla \delta\,=\,\left(\frac{\frak{a}_0^2}{2\delta}+\frac{2\frak{c}_0}{\frak{a}_0}\delta\right)\,\bV_0\,+\,\frac{3}{2}\frak{b}_0\,\delta\,\bW_0
\,+\,\frac{\frak{a}_0^2}{4}\delta\,\bZ_0
\,+\,{\mathcal O}(\delta^2)\,,
\end{equation}
the vectors $\bV_0$, $\bW_0$ and $\bZ_0$ being those defined in Notation \ref{not:UVWZ}.
\end{proposition}

\begin{proof}
Expansions of $m$, $\delta$ and $\nabla m$ are derived by taking sums and differences in \eqref{eq:asharmv} and \eqref{eq:asharmgradv} (Proposition~\ref{prop:asvgradv}) and noting that the ${\mathcal O}(\mu-\mu_0)^2$ term in the expansion of $\delta$ cancels out at leading order for the same symmetry reason as the ${\mathcal O}(\mu-\mu_0)$ does. 
The expansion of $\sqrt{\mu-\mu_0}$ follows readily by inserting $\sqrt{\mu-\mu_0}=\delta/\frak{a}_0\,+\,{\mathcal O}(\delta^3)$ in the second and third term of the expansion of $\delta$. Then the expansion of $1/\sqrt{\mu-\mu_0}$ stems directly from it. At last, the proof is achieved by the substitution of those in
$$
\nabla \delta\,=\,\left(\frac{\frak{a}_0}{2}\frac{1}{\sqrt{\mu-\mu_0}}+\frac{3}{2}\frak{c}_0\sqrt{\mu-\mu_0}\right)\,\bV_0\,+\,\frac{3}{2}\frak{a}_0\frak{b}_0\,\sqrt{\mu-\mu_0}\,\bW_0
\,+\,\frac{\frak{a}_0^3}{4}\,\sqrt{\mu-\mu_0}\,\bZ_0\,+\,{\mathcal O}(\mu-\mu_0)
$$
obtained from \eqref{eq:asharmgradv}.
\end{proof}

Arguments used to prove Proposition~\ref{prop:asactionperiod} may then also be used to foresee the form of the asymptotic expansion of $\nabla^2\Action$. Indeed from
$$
\begin{array}{rcccccc}
\nabla^2\Action&=&\displaystyle
\int_{-\pi/2}^{\pi/2} \nabla \f\,\dif\theta
&+&\displaystyle
\int_{-\pi/2}^{\pi/2} \f_v\,\dif\theta \,\otimes\,\nabla m
&+&\displaystyle
\int_{-\pi/2}^{\pi/2} \f_v \sin \theta\,\dif\theta \,\otimes\,\nabla \delta\\[10pt]
&&&+&\displaystyle
\int_{-\pi/2}^{\pi/2} (\f_w+\f_z)\,\dif\theta \,\otimes\,\nabla m
&+&\displaystyle
\int_{-\pi/2}^{\pi/2} (-\f_w+\f_z)\,\dif\theta \,\otimes\,\nabla \delta
\end{array}
$$
stems the existence of an expansion in powers of $\delta$, starting at order $\delta^{-1}$ because of the presence of $\nabla\delta$. Yet, now that this expansion is known to exist, it is in fact quicker to expand $\nabla\Action$ and then differentiate term by term to obtain its coefficients.

The expansion of $\nabla\Action$ is merely based on a Taylor expansion of $\f$ and on the expansion of $m$ in \eqref{eq:devmgradm} through $\bv_2\,=\, m\,-\,\delta$ and $\bv_3\,=\, m\,+\,\delta$. Note that the function $\f=\Y\,\nabla\ZZ$ inherits the symmetries of $\Redpot$ through its factor $\Y$, so that
$$\f_\bw^0=\f_\bz^0\,,\qquad
\f_{\bv\bw}^0=\f_{\bv\bz}^0\,,\qquad
\f_{\bw\bw}^0=\f_{\bz\bz}^0\,,$$
where the superscript $0$ stands for evaluation at $\bv=\bv_0$, $\bw=\bv_0$, $\bz=\bv_0$. Since $m$ is at an ${\mathcal O}(\delta^2)$ distance from $\bv_0$, we thus have
$$
\begin{array}{l}
\f=\f^0\,+\,\f_\bv^0\,(m-\bv_0+\delta \sin\theta) \,+\,2\,\f_\bz^0\,(m-\bv_0)\,+\,
\tfrac12 \f_{\bv\bv}^0\,(m-\bv_0+\delta \sin\theta)^2
\\[5pt]
\qquad\quad\,+\,2\,\f_{\bv\bz}^0\,(m-\bv_0+\delta \sin\theta)\,(m-\bv_0)\,+\,
\tfrac12 \,\f_{\bz\bz}^0\,((m-\bv_0+\delta)^2+(m-\bv_0-\delta)^2)\\[5pt]
\qquad\quad\,+\,\f_{\bw\bz}^0 \,(m-\bv_0+\delta)\,(m-\bv_0-\delta)
\,+\,{\mathcal O}(\delta^3)\,.
\end{array}$$
By \eqref{eq:devmgradm} this expansion reduces to
$$
\f=\f^0\,+\,\f_\bv^0\,\delta \sin\theta
\,+\,\Big(\,(\f_\bv^0+\,2\,\f_\bz^0)\,\frac{\frak{b}_0}{\frak{a}_0^2}\,+\, \f_{\bv\bv}^0\,\frac{\sin^2\theta}{2}\,+\,
\f_{\bz\bz}^0\,-\,\f_{\bw\bz}^0\Big)\,\delta^2\,+\,{\mathcal O}(\delta^3)\,.
$$
By integration over the interval $(-\pi/2,\pi/2)$, we 
thus obtain a rather simple expansion for $\nabla\Action$, which reads
$$
\frac1\pi\,\nabla\Action\,=\,
\frac1\pi\,\int_{-\pi/2}^{\pi/2}  \f\,\dif\theta
\,=\,
\f^0\,+\,\Big(\,(\f_\bv^0+\,2\,\f_\bz^0)\,\frac{\frak{b}_0}{\frak{a}_0^2}\,+\, \frac14\,\f_{\bv\bv}^0\,+\,
\f_{\bz\bz}^0\,-\,\f_{\bw\bz}^0\Big)\,\delta^2
\,+\,{\mathcal O}(\delta^3)\,.
$$

Now, differentiating with respect to $(\mu,\blambda,\speed)$ the 
${\mathcal O}(\delta^3)$ terms yields an ${\mathcal O}(\delta)$
remainder because of \eqref{eq:devgraddelta}. As to the leading order term $\f^0$, its differentiation yields two contributions, one coming from the dependence of $\Y$ and $\ZZ$ on $(\mu,\blambda,\speed)$ and one coming from their dependence on $\bv$, $\bw$, $\bz$ and that of $\bv_0$ on $(\blambda,\speed)$, since $\f^0$ is the product of $\Y$ and $\nabla\ZZ$ evaluated at $\bv=\bv_0$, $\bw=\bv_0$, $\bz=\bv_0$. Consequently, the Hessian of $\Action$ expands as follows $$\frac1\pi\,\nabla^2\Action\,=\,\nabla\f^0\,+\,
(\f_\bv^0+\,2\,\f_\bz^0)\,\otimes \,\nabla\bv_0\,+\,
\Big(\,(\f_\bv^0+\,2\,\f_\bz^0)\,\frac{\frak{b}_0}{\frak{a}_0^2}\,+\, \frac14\,\f_{\bv\bv}^0\,+\,
\f_{\bz\bz}^0\,-\,\f_{\bw\bz}^0\Big)\,\otimes 2\,\delta\,\nabla\delta
\,+\,{\mathcal O}(\delta)
$$
in the small amplitude limit $\delta\to 0$.
 
We are now ready to prove the following.
 
\begin{theorem} \label{thm:asharmhess}  
With notational conventions \ref{not:UVWZ}--\ref{not:abcp}, under Assumption \ref{as:main} the Hessian of the action has the asymptotic behavior
\begin{equation}\label{eq:asharmhess}
\frac{1}{\pi\Y^0}\,\nabla^2\Action\,=\,\begin{array}[t]{l}
\displaystyle
\alpha_0\,\bV_0\otimes \bV_0\,+\,\beta_0\,(\bV_0\otimes\bW_0\,+\,\bW_0\otimes\bV_0)\,+\,\frac{\frak{a}_0^2}{2}\,\bW_0\otimes\bW_0\\ [10pt]\displaystyle
\qquad\qquad\quad\,-\,\bT_0\otimes\bT_0\,+\,
\frac{\frak{a}_0^2}{4}\,(\bV_0\otimes\bZ_0\,+\,\bZ_0\otimes\bV_0) 
\,+\,{\mathcal O}(\delta)
\end{array}
\end{equation}
when $\delta=(\bv_3-\bv_2)/2 \to 0$, where 
$$
\begin{array}{rcl}
\alpha_0&:=&\displaystyle
\frac{1}{\Y^0}\,\Big(\,(\frak{b}_0\,(\Y_\bv^0+\,2\,\Y_\bz^0)\,+\,\frak{a}_0^2\,\big(\frac14\,\Y_{\bv\bv}^0\,+\,
\Y_{\bz\bz}^0\,-\,\Y_{\bw\bz}^0\big)\Big)\,,\\[10pt]
\beta_0&:=&\displaystyle
\frak{b}_0\,+\,\frac{\frak{a}_0^2}{2}\,\frac{\Y_\bv^0}{\Y^0}\,
\end{array}$$
may be computed explicitly in terms of derivatives of $\Potential$ and $\kappa$ at $\bv_0$ 
from \eqref{eq:redpot}-\eqref{eq:defY}.
\end{theorem}
The result is valid in both cases $N=1$ and $N=2$, the only difference being in the definitions \eqref{eq:vectstwo} and \eqref{eq:vectsone}
of vectors $\bV_0$, $\bW_0$, $\bZ_0$, the vector $\bT_0$ being equal to zero in the case $N=1$ and defined by the second relation in \eqref{eq:vectT} in the case $N=2$ - in both cases 
the first relation in \eqref{eq:vectT} stands, $\nabla^2\ZZ^0=-\bT_0\otimes \bT_0$.

\begin{proof}
Of course \eqref{eq:asharmhess} stems from the expansion derived above
and the one of $\nabla\delta$ in \eqref{eq:devgraddelta}. In addition, to achieve the proof, we must find expressions for $\nabla\bv_0$ and for the derivatives of $\f$ at $\bv=\bv_0$, $\bw=\bv_0$, $\bz=\bv_0$.

By definition we have $$\partial_\bv\Potential(\bv_{0}(\speed,\blambda);\speed,\blambda)=0$$
so that by the chain rule we get
\begin{equation}\label{eq:nablavzero}
\nabla\bv_0\,=\,-\,\frac{1}{\Potential_{\bv\bv}^0}\,\nabla \Potential_\bv^0\,=\,\frac{\frak{a}_0^2}{2}\,\bW_0\,.
\end{equation}

Now we have by the Leibniz formula
\begin{equation}\label{eq:nablafzero}
\nabla\f^0\,=\,\Y^0\,\nabla^2\ZZ^0\,+\,\nabla\ZZ^0\otimes \nabla\Y^0
\,=\,-\,\Y^0\,\bT_0\otimes \bT_0
\,+\,\bV_0\otimes \nabla\Y^0\,.
\end{equation}
Furthermore, we can show that
\begin{equation}
\label{eq:nablaYzero}
\nabla \Y^0\,=\,\frac{\frak{a}_0^2}{4}\,\Y^0\,\bZ_0\,.
\end{equation}
Indeed, since by definition 
$$\Y(\bv,\bw,\bz;\speed,\blambda)\,:=\,\sqrt{2 \cap(\bv)/|\Redpot(\bv,\bw,\bz;\speed,\blambda)|}\,,$$
we have 
$$\frac{\nabla\Y^0}{\Y^0}\,=\,-\,\frac{\nabla\Redpot^0}{2\,\Redpot^0}\,.$$
Furthermore, by differentiating under the integral sign in 
\eqref{eq:redpot} 
we have
$$\nabla\Redpot(\bv,\bw,\bz;\speed,\blambda)=-
\int_{0}^{1}\int_{0}^{1}t \nabla\ZZ_{\bv\bv}(\bw+t(\bz-\bw)+ts(\bv-\bz);\speed,\blambda)\,\dif s\dif t\,,$$
which gives in particular at $\bv=\bv_0$, $\bw=\bv_0$, $\bz=\bv_0$ the value 
$$\nabla\Redpot^0\,=\,-\,\nabla\ZZ_{\bv\bv}^0\,\int_{0}^{1}\int_{0}^{1} t\,\dif s \dif t\,=\,-\,\frac{\bZ_0}{2}\,.$$
Concerning the value of $\Redpot^0$, as we have used repeatedly, it is readily given by
$$\Redpot^0\,=\,\Potential_{\bv\bv}^0\,\int_{0}^{1}\int_{0}^{1}t\,\dif s\dif t\,=\,\frac{\Potential_{\bv\bv}^0}{2}\,=\,\frac{1}{\frak{a}_0^2}$$
by definition of $\frak{a}_0$. This eventually proves \eqref{eq:nablaYzero}.

Regarding the other derivatives of $\f$ we have by the Leibniz formula and the definition of $\bV_0$, $\bW_0$, $\bZ_0$ in Notation \ref{not:UVWZ}
\begin{equation}\label{eq:derfzero}
\left\{\begin{array}{l}
\f^0_{\bv}\,=\,\Y^0_{\bv}\,\bV_0\,+\,\Y^0\,\bW_0\,,\qquad
\f^0_{\bv\bv}\,=\,\Y^0_{\bv\bv}\,\bV_0\,+\,2\,\Y^0_{\bv}\,\bW_0\,+\,\Y^0\,\bZ_0\,,\\ [10pt]
\f^0_{\bz}\,=\,\Y^0_{\bz}\,\bV_0\,,\qquad
\f^0_{\bz\bz}\,=\,\Y^0_{\bz\bz}\,\bV_0\,,\qquad
\f^0_{\bw\bz}\,=\,\Y^0_{\bw\bz}\,\bV_0
\,.
\end{array}\right.
\end{equation}

Therefore, by using \eqref{eq:devgraddelta}, \eqref{eq:nablavzero}, \eqref{eq:nablafzero}, \eqref{eq:nablaYzero}, and \eqref{eq:derfzero}
in the original expansion
$$\frac1\pi\,\nabla^2\Action\,=\,\nabla\f^0\,+\, (\f_\bv^0+\,2\,\f_\bz^0)\,\otimes \,\nabla\bv_0\,+\, \Big(\,(\f_\bv^0+\,2\,\f_\bz^0)\,\frac{\frak{b}_0}{\frak{a}_0^2}\,+\, \frac14\,\f_{\bv\bv}^0\,+\, \f_{\bz\bz}^0\,-\,\f_{\bw\bz}^0\Big)\,\otimes 2\,\delta\,\nabla\delta
\,+\,{\mathcal O}(\delta)
$$
we arrive at 
$$\frac1\pi\,\nabla^2\Action\,=\,\begin{array}[t]{l}\Big(\,(\frak{b}_0\,(\Y_\bv^0+\,2\,\Y_\bz^0)\,+\,\frak{a}_0^2\,\big(\frac14\,\Y_{\bv\bv}^0\,+\,
\Y_{\bz\bz}^0\,-\,\Y_{\bw\bz}^0\big)\Big)\,\bV_0\otimes \bV_0\\ [10pt]
\,+\,\Y^0\,\frac{\frak{a}_0^2}{2}\,\bW_0\otimes\bW_0\,-\,\Y^0\,\bT_0\otimes\bT_0\,+\,
\Y^0\,\frac{\frak{a}_0^2}{4}\,(\bV_0\otimes\bZ_0\,+\,\bZ_0\otimes\bV_0) \\ [10pt]
\,+\,\big(\frak{b}_0\,\Y^0\,+\,\Y_\bv^0\,\frac{\frak{a}_0^2}{2}\big)\,\bW_0\otimes\bV_0
\,+\,\big(\frak{a}_0^2\,\Y_\bz^0\,+\,\Y_\bv^0\,\frac{\frak{a}_0^2}{2}\big)\,\bV_0\otimes\bW_0
\,+\,{\mathcal O}(\delta)\,.
\end{array}$$

To complete the proof of \eqref{eq:asharmhess}, we check that
\begin{equation}
\label{eq:bYzzero}
\frak{b}_0 \Y^0\,=\,\frak{a}_0^2\,\Y^0_{\bz}\,.
\end{equation}
We notice indeed that 
$$\frac{\Y^0_{\bz}}{\Y^0}\,=\,-\,\frac{\Redpot^0_{\bz}}{2\,\Redpot^0}\,,$$
and recalling that $\Redpot^0=1/\frak{a}_0^2$, we obtain 
$$\Redpot^0_{\bz}\,=\,\Potential_{\bv\bv\bv}^0\,\int_{0}^{1}\int_{0}^{1}t^2(1-s)\,\dif s\dif t\,=\,\frac{\Potential_{\bv\bv\bv}^0}{6}\,=\,-\,\frac{2\frak{b}_0}{\frak{a}_0^4}$$
by differentiating under the integral sign in \eqref{eq:redpot} and by definition of $\frak{a}_0$ and $\frak{b}_0$ in \eqref{eq:abcp0}.
\end{proof}

\subsubsection{Soliton limit}\label{ss:sollim}
The asymptotic behavior of $\nabla^2\Action$ is unsurprisingly more singular in the soliton limit than in the small amplitude, or harmonic limit. It will turn out that the expansion of  $\nabla^2\Action$ actually involves similar matrix-valued functions - evaluated at $\bv_s$ instead of $\bv_0$ - but also expressions evaluated at $\bv^s$ and a separation of scales that are not present in the harmonic limit.

Our starting point is the desingularization
$$
\nabla\Action
\,=\,\int_{0}^{1} \f\,\frac{\dif\sigma}{\sqrt{\sigma(\sigma+\rho)}}
\,=\,\int_{0}^{1} \frac{\fc}{\sqrt{1-\sigma}} \,\frac{\dif\sigma}{\sqrt{\sigma(\sigma+\rho)}}\,,$$
where $\f=\Y\,\nabla\ZZ$ and $\fc:=\Yc\,\nabla\ZZ$ are evaluated at $\bv=\bv_2+\sigma(\bv_3-\bv_2)$, $\bw=\bv_1$, $\bz=\bv_2$, and $\bzc=\bv_3$, and 
$$\rho=\frac{\bv_2-\bv_1}{\bv_3-\bv_2}\,.$$
We recall that $\Y$ is defined by \eqref{eq:redpot}-\eqref{eq:defY} and that $\Yc$ is given by 
$$\Yc(\bv,\bw,\bz,\bzc;\speed,\blambda)\,=\sqrt{\frac{\bzc-\bv}{\bzc-\bz}}\;\Y(\bv,w,z;\speed,\blambda)\,=\,
\frac{1}{\sqrt{\bzc-\bz}}\,\sqrt{\frac{2 \cap(\bv)}{|\Redpotc(\bv,\bw,\bz,\bzc;\speed,\blambda)|}}$$
where 
$$\Redpotc(\bv,\bw,\bz,\bzc;\speed,\blambda)=
\int_{0}^{1}\int_{0}^{1}\int_{0}^{1}t^2s\,\partial_\bv^3\Potential(\bw+t(\bz-\bw)+ts(\bzc-\bz)+tsr \,(\bv-\bzc);\speed,\blambda)\,\dif r\dif s\dif t\,,$$
is such that
$$\ZZ(\bv;\mu,\blambda,\speed)
\,=\,\mu-\Potential(\bv;\speed,\blambda)
\,=\,(\bv- \bv_1)(\bv-\bv_2) (\bv_3-\bv) \Redpotc(\bv,\bv_1,\bv_2,\bv_3;\speed,\blambda)\,.$$

To describe the soliton limit, 
we shall use $\rho$ as our small parameter - essentially as $\delta$ in the harmonic limit. Therefore we begin by recasting in terms of $\rho$ some of the previously obtained expansions. 

\begin{proposition}\label{prop:assol}
With notational conventions \ref{not:derivatives}--\ref{not:UVWZ}--\ref{not:abcp}, under Assumption \ref{as:main}, and with 
$$\rho\,=\,\frac{\bv_2-\bv_1}{\bv_3-\bv_2}
\qquad\mbox{and}\qquad
\h_s:=\bv^s-\bv_s$$
we have, in the limit $\rho\to0$,
\begin{equation}\label{eq:devmu}
\sqrt{\mu_s-\mu}\,=\,\frac{\h_s}{2\frak{a}_s}\,\left(\rho \,-\,\frac12 \rho^2\,+\,\frac14\,\Big(1\,-\,\frac{(\frak{b}_s+\frak{p}_s) h_s}{\frak{a}_s^2}\,-\,\frac{\frak{c}_s\h_s^2}{\frak{a}_s^3}\Big)\,\rho^3\right)\,+\,{\mathcal O}(\rho^4)\,,
\end{equation}
thus
$$
\left\{\begin{array}{rcl}
\mu_s-\mu&=&\displaystyle
\frac{\h_s^2}{4\frak{a}_s^2}\,(\rho^2\,-\,\rho^3)\,+\,{\mathcal O}(\rho^4)\,,\\[10pt]\displaystyle
\frac{1}{\sqrt{\mu_s-\mu}}&=&\displaystyle
\frac{2\frak{a}_s}{\rho\,\h_s}\,+\,\frac{\frak{a}_s}{\h_s}\,+\,\frac{\h_s}{2\frak{a}_s}\,\Big(\frac{\frak{b}_s+\frak{p}_s}{\h_s}\,+\,\frac{\frak{c}_s}{\frak{a}_s}\Big)\,\rho\,+\,{\mathcal O}(\rho^2)\,,
\end{array}\right.$$
\begin{equation}\label{eq:devvi}
\left\{\begin{array}{lll}\displaystyle
\bv_1=\bv_s\,-\,\frac{\h_s}{2} \,\rho\,&\displaystyle+\,\frac{\h_s}{4}\,\Big(1+\frac{\frak{b}_s\,\h_s}{\frak{a}_s^2}\Big)\,\rho^2\,-\,\frac{h_s}{8}\,\Big(1+\frac{(\frak{b}_s-\frak{p}_s)\,\h_s}{\frak{a}_s^2}\Big)\,\rho^3&+\,{\mathcal O}(\rho^4)\,,\\ [15pt]
\displaystyle
\bv_2=\bv_s\,+\,\frac{\h_s}{2} \,\rho\,&\displaystyle-\,\frac{\h_s}{4}\,\Big(1-\frac{\frak{b}_s\,\h_s}{\frak{a}_s^2}\Big)\,\rho^2\,+\,
\frac{h_s}{8}\,\Big(1-\frac{(3\frak{b}_s+\frak{p}_s)\,\h_s}{\frak{a}_s^2}\Big)\,\rho^3\,&+\,{\mathcal O}(\rho^4)\,,\\ [15pt]
\displaystyle
\bv_3=\bv^s\,&\displaystyle-\,\frac{\frak{p}_s\,\h_s^2}{4\frak{a}_s^2}\,
{
(\rho^2-\rho^3)}
\,&+\,{\mathcal O}(\rho^4)\,,
\end{array}
\right.
\end{equation}
\begin{equation}\label{eq:devgradvi}
\left\{\begin{array}{rcl}\displaystyle
\nabla\bv_1&=&\displaystyle
\frac{\frak{a}_s^2}{\rho\,\h_s}\,\bV_s
\,-\,\Big(\frak{b}_s-\frac{\frak{a}_s^2}{2\h_s}\Big)\,\bV_s\,-\,\frac{\frak{a}_s^2}{2}\,\bW_s\,
\\[10pt]
&&\displaystyle 
\quad\,+\,
{\frac{1}{4}\,\Big(\frak{b}_s+\frak{p}_s
\,+\,\frac{8\frak{c}_s\h_s}{\frak{a}_s}\Big)}
\,\rho\,\bV_s
\,+\, \frac34\h_s\frak{b}_s\,\rho\, \bW_s + \frac{\h_s\frak{a}_s^2}{8}\,\rho\, \bZ_s +\,{\mathcal O}(\rho^2)\,,\\ [15pt]
\displaystyle
\nabla\bv_2&=&\displaystyle
-\,\frac{\frak{a}_s^2}{\rho\,\h_s}\,\bV_s
\,-\,\Big(\frak{b}_s+\frac{\frak{a}_s^2}{2\h_s}\Big)\,\bV_s\,-\,\frac{\frak{a}_s^2}{2}\,\bW_s\,
\\[10pt]
&&\displaystyle 
\quad\,-\,
{\frac{1}{4}\,\Big(\frak{b}_s+\frak{p}_s
\,+\,\frac{8\frak{c}_s\h_s}{\frak{a}_s}\Big)}
\,\rho\,\bV_s
\,-\, \frac34\h_s\frak{b}_s\,\rho\, \bW_s - \frac{\h_s\frak{a}_s^2}{8}\,\rho\, \bZ_s +\,{\mathcal O}(\rho^2)\,,\\ [15pt]
\displaystyle
\nabla\bv_3&=& \qquad \displaystyle\,\frak{p}_s\,
\bV^s
\,+\,{\mathcal O}(\rho^2)\,,
\end{array}
\right.\end{equation}
and
\begin{equation}\label{eq:devgradrho}
\nabla\rho=\begin{array}[t]{l}
\displaystyle
-\Big[\frac{\frak{a}_s^2}{\h_s^2}\,\left(\frac{2}{\rho}+3\right)
\,+\,\frac{1}{\h_s}\,
{\Big(2\frak{b}_s\,+\,\frak{p}_s\,+\,
\frac{4\frak{c}_s}{\frak{a}_s}
\,+\,\frac{3\frak{a}_s^2}{2\h_s}\Big)}
\,\rho\Big]\,\bV_s\\ [15pt]
\displaystyle
\quad\,-\,\frac{\frak{p}_s}{\h_s}\rho\,\bV^s\,-\,\frac12\,\Big(3\frak{b}_s+\frac{\frak{a}_s^2}{\h_s}\Big)\,\rho\,\bW_s
\,-\,\frac{\frak{a}_s^2}{4}\,\rho\,\bZ_s
\,+\,{\mathcal O}(\rho^2)\,.
\end{array}
\end{equation}
\end{proposition}

\begin{proof}
From \eqref{eq:assolv} in Proposition~\ref{prop:asvgradv}, we derive
$$\rho\,=\,\frac{2\frak{a}_s\,\sqrt{\mu_s-\mu}\,+\, 2\frak{c}_s \,(\mu_s-\mu)^{3/2}\,+\,{\mathcal O}(\mu_s-\mu)^2}{\h_s\,-\,\frak{a}_s\,\sqrt{\mu_s-\mu}\,-\,(\frak{b}_s+\frak{p}_s)\,(\mu_s-\mu)\,+\,{\mathcal O}(\mu_s-\mu)^{3/2}}\,, $$
which yields
$$
\rho\,=\,
\frac{2\frak{a}_s}{\h_s}\Big(\sqrt{\mu_s-\mu}
+\frac{\frak{a}_s}{\h_s}({\mu_s-\mu})
+\Big(\frac{\frak{b}_s+\frak{p}_s}{\h_s}+\frac{\frak{a}_s^2}{\h_s^2}+\frac{\frak{c}_s}{\frak{a}_s}\Big)({\mu_s-\mu})^{3/2}
+{\mathcal O}(\mu_s-\mu)^{2}\Big)\,.
$$
From this we can derive an expansion of $\sqrt{\mu_s-\mu}$ in powers of $\rho$. We find successively
$$
\begin{array}{rcl}
\sqrt{\mu_s-\mu}&=&\displaystyle
\frac{\h_s}{2\frak{a}_s}\,\rho\,+\,{\mathcal O}(\rho^2)\,,\\[10pt]
\sqrt{\mu_s-\mu}&=&\displaystyle
\frac{\h_s}{2\frak{a}_s}\,(\rho \,-\,\frac12 \rho^2)\,+\,{\mathcal O}(\rho^3)\,,
\end{array}
$$
and finally \eqref{eq:devmu}, from which we infer expansions of $\mu_s-\mu$ and $1/\sqrt{\mu_s-\mu}$. Remaining expansions are then obtained by substitution in \eqref{eq:assolv}-\eqref{eq:assolgradv}.

For the convenience of the reader, let us provide some computational details on the derivation of \eqref{eq:devgradrho}. It follows from
$$
\frac{\nabla\rho}{\rho}\,=\,
\frac{\nabla(\bv_2-\bv_1)}{\bv_2-\bv_1}\,-\,\frac{\nabla(\bv_3-\bv_2)}{\bv_3-\bv_2}$$
by using that 
$$
\begin{array}{rcl}
\nabla(\bv_2-\bv_1)&=&\displaystyle
-\Big[\frac{\frak{a}_s^2}{\h_s}\,\left(\frac{2}{\rho}+1\right)\,\,+\,\frac{1}{2}\,\Big(\frak{b}_s+\,\frak{p}_s
\,+\,\frac{8\frak{c}_s\h_s}{\frak{a}_s}\Big)\,\rho\Big]\,\bV_s
\\[5pt]
&&\displaystyle
\qquad\qquad
\,-\, \frac32\h_s\frak{b}_s\,\rho\, \bW_s
\,-\,\frac{\h_s\frak{a}_s^2}{4}\,\rho\, \bZ_s\,+\,{\mathcal O}(\rho^2)\,,\\[10pt]
\displaystyle
\frac{1}{\bv_2-\bv_1}
&=&\displaystyle
\frac{1}{\h_s\rho}\Big(
1+\frac12\rho+\frac{(\frak{b}_s+\frak{p}_s)\,\h_s}{4\frak{a}_s^2}\rho^2
\Big)\,+\,{\mathcal O}(\rho^2)\,,\\[10pt]
\nabla(\bv_3-\bv_2)&=&\displaystyle
\frac{\frak{a}_s^2}{\rho\,\h_s}\,\bV_s\,+\,\big(\frak{b}_s+\frac{\frak{a}_s^2}{2\h_s}\big)\,\bV_s\,+\,\frac{\frak{a}_s^2}{2}\,\bW_s\,+\,\frak{p}_s\,\bV^s\,+\,{\mathcal O}(\rho)\,,\\[10pt]
\displaystyle
\frac{1}{\bv_3-\bv_2}
&=&\displaystyle
\frac{1}{\h_s}\Big(1+\frac{1}{2}\rho\Big)\,+\,{\mathcal O}(\rho^2)\,.
\end{array}
$$
\end{proof}

Again, arguments used to prove Proposition~\ref{prop:asactionperiod} - essentially Taylor expansions of integrands and repeated use of Proposition~\ref{prop:asol} - enable us to justify the form of the asymptotic expansion of $\nabla^2\Action$. Indeed from
$$
\begin{array}{rcccc}
\nabla^2\Action&=&\displaystyle
\int_{0}^{1}\frac{\nabla\fc}{\sqrt{1-\sigma}}\,\frac{\dif\sigma}{\sqrt{\sigma(\sigma+\rho)}}
&+&\displaystyle
\int_{0}^{1}\frac{\fc_\bw}{\sqrt{1-\sigma}}\,\frac{\dif\sigma}{\sqrt{\sigma(\sigma+\rho)}}\,\otimes\, \nabla\bv_1
\\[15pt]
&+&\displaystyle
\int_{0}^{1}\frac{\fc_\bz}{\sqrt{1-\sigma}}\,\frac{\dif\sigma}{\sqrt{\sigma(\sigma+\rho)}}\,\otimes\, \nabla\bv_2
&+&\displaystyle
\int_{0}^{1}
\frac{\fc_\bzc}{\sqrt{1-\sigma}}\,\frac{\dif\sigma}{\sqrt{\sigma(\sigma+\rho)}}\,\otimes\, \nabla\bv_3\\[15pt]
&+&\displaystyle
\int_{0}^{1}\frac{\fc_\bv}{\sqrt{1-\sigma}}\,\frac{(1-\sigma)\,\dif\sigma}{\sqrt{\sigma(\sigma+\rho)}}\,\otimes\, \nabla\bv_2
&+&\displaystyle
\int_{0}^{1}\frac{\fc_\bv}{\sqrt{1-\sigma}}\,\frac{\sigma\,\dif\sigma}{\sqrt{\sigma(\sigma+\rho)}}\,\otimes\, \nabla\bv_3
\\[15pt]
&-&\displaystyle
\frac12\,\int_{0}^{1}\frac{\fc}{\sqrt{1-\sigma}}\,\frac{\dif\sigma}{\sqrt{\sigma(\sigma+\rho)^3}}\,\otimes\, \nabla\rho
&&
\end{array}
$$
follows the existence of an expansion in $\rho^k$, $k\geq-2$, and $\rho^k\,\ln\rho$, $k\geq-1$. However, as in the harmonic limit, we find mind more convenient to derive coefficients of the expansion by differentiating a similar expansion for $\nabla\Action$.

As the process will be carried out at a relatively high order, it may be useful to explain it first in a rather abstract form. To do so, we use notation for asymptotic expansions explicitly described in Appendix~\ref{app:elas} and rely implicitly on Lemma~\ref{lem:derasexp}. We have already seen in the proof of Proposition~\ref{prop:assol} that 
$$\nabla \rho \sim \sum_{k\geqslant -1} \bY_{k+1}\,\rho^k$$
with
$$
\begin{array}{l}\displaystyle
\bY_0= -2\,\frac{\frak{a}_s^2}{\h_s^2}\,\bV_s\,,
\qquad\qquad
\bY_1=-3\,\frac{\frak{a}_s^2}{\h_s^2}\,\bV_s\,,\\[10pt]
\displaystyle
\bY_2=
-\frac{1}{\h_s}\,
\Big(2\frak{b}_s\,+\,\frak{p}_s\,+\,
\frac{4\frak{c}_s}{\frak{a}_s}
\,+\,\frac{3\frak{a}_s^2}{2\h_s}\Big)\,\bV_s
\,-\,\frac{\frak{p}_s}{\h_s}\,\bV^s\,-\,\frac12\,\Big(3\frak{b}_s+\frac{\frak{a}_s^2}{\h_s}\Big)\,\bW_s
\,-\,\frac{\frak{a}_s^2}{4}\,\bZ_s\,.
\end{array}
$$
Likewise we shall obtain for $\nabla\Action$ an expansion of the form
$$\nabla\Action \sim \sum_{k\geqslant 0} \Big( \bA_k\,\rho^k\,\ln\rho\,+\,\bB_k\,\rho^k\Big)\,,$$
where the vectors $\bA_k$ and $\bB_k$ depend on $(\speed,\blambda)$ \emph{per se}, through the 
solitary wave endstate $\bv_s$, and also through the solitary wave 
maximum $\bv^s$ as regards the $\bB_k$'s. From the latter we infer 
$$
\begin{array}{r}\displaystyle
\nabla^2\Action \sim \left(\sum_{k\geqslant -1} (k+1)\bA_{k+1}\,\rho^k\,\ln\rho\,+\,(\bA_{k+1}+(k+1)\bB_{k+1})\,\rho^k\right)\otimes \left(\sum_{k\geqslant -1} \bY_{k+1}\,\rho^k\right)\\[5pt]\displaystyle
\,+\,
\sum_{k\geqslant 0} \Big(\nabla_{\tot}\bA_k\,\rho^k\,\ln\rho\,+\,\nabla_{\tot}\bB_k\,\rho^k\Big)
\end{array}
$$
where, in order to avoid too heavy formulas at this stage, we have denoted  by $\nabla_{\tot}$ the `total gradient' with respect to $(\mu,\blambda,\speed)$, which involves the regular gradient $\nabla$ but also $\nabla\bv_s$ and $\nabla\bv^s$  by the chain rule. By collecting terms, this indeeds provides coefficients of the expansion
$$\nabla^2\Action \sim \bD_{0}\,\rho^{-2}\,+\,\sum_{k\geqslant -1} \Big(\bC_{k+2}\,\rho^k\,\ln\rho\,+\,\bD_{k+2}\,\rho^k\Big)\,.$$
For instance, 
$$
\begin{array}{l}\displaystyle
\bD_0=\bA_0\otimes \bY_{0}\,,\qquad
\bD_1=\bA_0\otimes \bY_1\,+\,(\bA_1+\bB_1)\otimes \bY_0\,,\\[5pt]\displaystyle
\bC_{1}=\bA_1\otimes \bY_0\,,\qquad
\bC_2=\bA_1\otimes \bY_1\,+\,2\,\bA_2\otimes \bY_0\,+\,\nabla_{\tot}\bA_0\,,\\[5pt]
\bD_2=\bA_0\otimes \bY_2\,+\,(\bA_1+\bB_1)\otimes \bY_1\,+\,(\bA_2+2\bB_2)\otimes \bY_0\,+\,\nabla_{\tot}\bB_0\,.
\end{array}
$$

Now we turn to the derivation of an asymptotic expansion for $\nabla\Action$, that is, to the explicit computation of $\bA_k$, $\bB_k$, $k\geq0$. The first step is to expand $\fc$ evaluated at $\bv=\bv_2+\sigma(\bv_3-\bv_2)$, $\bw=\bv_1$, $\bz=\bv_2$, and $\bzc=\bv_3$. To begin with, observe that, since $\bv_3$ is within a distance ${\mathcal O}(\rho^2)$ from $\bv^s$,
$$
\fc\begin{array}[t]{l}=\fc^s\,+\,(\bv_1-\bv_s)\,\fc^s_\bw\,+\,\,+\,(\bv_2-\bv_s)\,\fc^s_\bz\,+\,(\bv_3-\bv^s)\,\fc^s_\bzc
\,+\,((1-\sigma)(\bv_2-\bv_s)\,+\,\sigma(\bv_3-\bv^s))\,\fc^s_\bv\\ [10pt]
\quad
\,+\,\frac12 (\bv_1-\bv_s)^2\,\fc^s_{\bw\bw}\,+\,\frac12 (\bv_2-\bv_s)^2\,\fc^s_{\bz\bz}\,+\,\frac12 (1-\sigma)^2(\bv_2-\bv_s)^2\,\fc^s_{\bv\bv}
\,+\,(\bv_1-\bv_s)(\bv_2-\bv_s)\,\fc^s_{\bw\bz}
\\ [10pt]
\quad\,+\,(1-\sigma)\,(\bv_1-\bv_s)(\bv_2-\bv_s)\,\fc^s_{\bw\bv}\,+\,(1-\sigma)(\bv_2-\bv_s)^2\,\fc^s_{\bz\bv}\,
\,+\,{\mathcal O}(\rho^3)\,,\end{array}$$
where the superscript $s$ means that functions are evaluated at 
$\bv=\bv_s+\sigma(\bv^s-\bv_s)$, $\bw=\bv_s$, $\bz=\bv_s$, and $\bzc=\bv^s$. When we need to evaluate functions at $\bv=\bv_s$, which corresponds to $\sigma=0$, we will add $0$ as a subscript, consistently with convention of Remark~\ref{rem:superscripts} in Section~\ref{ss:actionperiod}. By replacing the $\bv_i$'s with their expansions \eqref{eq:devvi} we get a genuine expansion for $\fc$,
\begin{equation}\label{eq:devfc-first}
\fc\begin{array}[t]{l}=
\fc^s\,+\,\dfrac{\rho\,\h_s}{2}\,(-\,\fc^s_\bw\,+\,\fc^s_\bz\,+\,(1-\sigma)\,\fc^s_\bv)\\ [10pt]
\ \,+\,\dfrac{\rho^2\,\h_s^2}{4}\,\Big(
-\dfrac{\frak{p}_s}{\frak{a}_s^2}\,\Big(\fc^s_\bzc+\sigma \fc^s_\bv\Big)\,+\,\dfrac{\frak{2b}_s}{\frak{a}_s^2}\,\fc^s_\bw
\,-\,\Big(\dfrac{1}{\h_s}-\dfrac{\frak{b}_s}{\frak{a}_s^2}\Big)\,(-\fc^s_\bw+\fc^s_\bz+(1-\sigma)\fc^s_\bv)
\Big)\\ [10pt]
\ \,+\,\dfrac{\rho^2\,\h_s^2}{4}\,\Big( 
\frac12 \fc^s_{\bw\bw}\,+\,\frac12 \fc^s_{\bz\bz}\,+\,\frac12 (1-\sigma)^2 \fc^s_{\bv\bv}
\,-\,\fc^s_{\bw\bz}\,+\,(1-\sigma)\,(\fc^s_{\bz\bv}\,-\,\fc^s_{\bw\bv})
\Big)
\,+\,{\mathcal O}(\rho^3)\,.\end{array}
\end{equation}

There are some simplifications to be made by using the relationship between $\fc$ and $\f$ together with the symmetries of $\f^s$. First, since 
$$\fc(\bv,\bw,\bz,\bzc;\speed,\blambda)=\sqrt{\frac{\bzc-\bv}{\bzc-\bz}}\;\f(\bv,w,z;\speed,\blambda)\,,$$
by differentiation we find that
$$
\left\{\begin{array}{l}\displaystyle
\fc_\bv\,=\,\sqrt{\frac{\bzc-\bv}{\bzc-\bz}}\;\f_\bv\,-\,\frac{\f}{2\sqrt{(\bzc-\bv)(\bzc-\bz)}\;}\,,\qquad\qquad
\fc_\bw\,=\,\sqrt{\frac{\bzc-\bv}{\bzc-\bz}}\;\f_\bw\,,\\ [15pt]
\displaystyle
\fc_\bz\,=\,\sqrt{\frac{\bzc-\bv}{\bzc-\bz}}\;\f_\bz\,+\,\frac{\sqrt{\bzc-\bv}}{(\bzc-\bz)^{3/2}}\;\,\frac{\f}{2}\,,
\qquad\qquad
\fc_\bzc\,=\,\left(\frac{1}{\sqrt{(\bzc-\bv)(\bzc-\bz)}}\,-\,\frac{\sqrt{\bzc-\bv}}{(\bzc-\bz)^{3/2}}\right)\,\frac{\f}{2}\,.
\end{array}\right.
$$
When evaluated at $\bv=\bv_2+\sigma(\bv_3-\bv_2)$, $\bw=\bv_1$, $\bz=\bv_2$, and 
$\bzc=\bv_3$, these relations yields
\begin{equation}\label{eq:simpf}
\left\{\begin{array}{l}\displaystyle
\frac{\fc_\bv}{\sqrt{1-\sigma}}\,=\,\f_\bv\,-\,\frac{\f}{2(1-\sigma)(\bv_3-\bv_2)}\,,\qquad\qquad
\frac{\fc_\bw}{\sqrt{1-\sigma}}\,=\,\f_\bw\,,
\\ [15pt]
\displaystyle
\frac{\fc_\bz}{\sqrt{1-\sigma}}\,=\,\f_\bz\,+\,\frac{\f}{2(\bv_3-\bv_2)}\,,\qquad\qquad
\frac{\fc_\bzc}{\sqrt{1-\sigma}}\,=\,\left(\frac{1}{1-\sigma}\,-\,1\right)\,\frac{\f}{2(\bv_3-\bv_2)}\,.
\end{array}\right.
\end{equation}
In a similar way, by differentiating once more, we obtain
$$
\left\{\begin{array}{l}\displaystyle
\fc_{\bv\bv}\,=\,\sqrt{\frac{\bzc-\bv}{\bzc-\bz}}\;\f_{\bv\bv}\,-\,\frac{\f_\bv}{\sqrt{(\bzc-\bv)(\bzc-\bz)}\;}\,-\,\frac{\f}{4\sqrt{(\bzc-\bv)^3(\bzc-\bz)}\;}\,,
\\ [20pt]\displaystyle 
\fc_{\bv\bw}\,=\,\sqrt{\frac{\bzc-\bv}{\bzc-\bz}}\;\f_{\bv\bw}\,-\,\frac{\f_\bw}{2\sqrt{(\bzc-\bv)(\bzc-\bz)}\;}\,,
\qquad\qquad\qquad
\fc_{\bw\bw}\,=\,\sqrt{\frac{\bzc-\bv}{\bzc-\bz}}\;\f_{\bw\bw}\,,
\\ [20pt]\displaystyle 
\fc_{\bv\bz}\,=\,\sqrt{\frac{\bzc-\bv}{\bzc-\bz}}\;\f_{\bv\bz}\,-\,\frac{\f_\bz}{2\sqrt{(\bzc-\bv)(\bzc-\bz)}\;}\,+\,\frac{\sqrt{\bzc-\bv}}{(\bzc-\bz)^{3/2}}\,\frac{\f_\bv}{2}\,-\,\frac{\f}{4\sqrt{(\bzc-\bv)(\bzc-\bz)^3}\;}\,,
\\ [20pt]\displaystyle
\fc_{\bw\bz}\,=\,\sqrt{\frac{\bzc-\bv}{\bzc-\bz}}\;\f_{\bw\bz}\,+\,\frac{\sqrt{\bzc-\bv}}{(\bzc-\bz)^{3/2}}\;\,\frac{\f_\bw}{2}\,,\\[20pt]
\displaystyle
\fc_{\bz\bz}\,=\,\sqrt{\frac{\bzc-\bv}{\bzc-\bz}}\;\f_{\bz\bz}\,+\,\frac{\sqrt{\bzc-\bv}}{(\bzc-\bz)^{3/2}}\;\,{\f_\bz}\,+\,\frac34\frac{\sqrt{\bzc-\bv}}{(\bzc-\bz)^{5/2}}\;\,\f\,,
\end{array}\right.
$$
hence after evaluation at $\bv=\bv_2+\sigma(\bv_3-\bv_2)$, $\bw=\bv_1$, $\bz=\bv_2$, and $\bzc=\bv_3$
\begin{equation}\label{eq:simpfder}
\left\{\begin{array}{l}\displaystyle
\frac{\fc_{\bv\bv}}{\sqrt{1-\sigma}}\,=\,\f_{\bv\bv}\,-\,\frac{\f_\bv}{(1-\sigma)\,(\bv_3-\bv_2)\;}\,-\,\frac{\f}{4(1-\sigma)^2\,(\bv_3-\bv_2)^2\;}\,,
\\[20pt]
\displaystyle
\frac{\fc_{\bv\bw}}{\sqrt{1-\sigma}}\,=\,\f_{\bv\bw}\,-\,\frac{\f_\bw}{2(1-\sigma)\,(\bv_3-\bv_2)\;}\,,\qquad\qquad
\frac{\fc_{\bw\bw}}{\sqrt{1-\sigma}}\,=\,\displaystyle  \f_{\bw\bw}\,,
\\ [20pt]
\displaystyle
\frac{\fc_{\bv\bz}}{\sqrt{1-\sigma}}\,=\,\f_{\bv\bz}\,-\,\frac{\f_\bz}{2(1-\sigma)\,(\bv_3-\bv_2)\;}\,+\,\frac{\f_\bv}{2(\bv_3-\bv_2)}\,-\,\frac{\f}{4(1-\sigma)\,(\bv_3-\bv_2)^2\;}\,,
\\ [20pt]
\displaystyle
\frac{\fc_{\bw\bz}}{\sqrt{1-\sigma}}\,=\,\f_{\bw\bz}\,+\,\frac{\f_\bw}{2\,(\bv_3-\bv_2)}\,,
\qquad\ 
\frac{\fc_{\bz\bz}}{\sqrt{1-\sigma}}\,=\,\f_{\bz\bz}\,+\,\frac{\f_\bz}{\bv_3-\bv_2}\,+\,\frac{3\,\f}{4\,(\bv_3-\bv_2)^2}\,.
\end{array}\right.
\end{equation}
Now the symmetry of $\f$ inherited from that of $\Redpot$ - see Lemma~\ref{lem:symmetry} - implies that $\f^s_{\bw}\,=\,\f^s_{\bz}$, $\f^s_{\bv\bw}=\f^s_{\bv\bz}$, $\f^s_{\bw\bw}=\f^s_{\bz\bz}$, so that \eqref{eq:simpf} and \eqref{eq:simpfder} imply
$$
\begin{array}{l}\displaystyle
\frac{-\,\fc^s_\bw\,+\,\fc^s_\bz\,+\,(1-\sigma)\,\fc^s_\bv}{\sqrt{1-\sigma}}\,=\,(1-\sigma)\,\f^s_\bv\,,
\qquad\qquad
\frac{\fc^s_\bzc+\sigma\,\fc^s_\bv}{\sqrt{1-\sigma}}\,=\,\sigma\,
\f^s_\bv\,,
\\[5pt]
\displaystyle
\frac{\fc^s_\bz+(1-\sigma)\,\fc^s_\bv}{\sqrt{1-\sigma}}\,=\,\f^s_\bz+(1-\sigma)\,\f^s_\bv\,,
\qquad\qquad
\frac{\fc^s_\bw}{\sqrt{1-\sigma}}\,=\,\f^s_\bw\,=\,\f^s_\bz\,,
\end{array}
$$
and
$$
\begin{array}{r}\displaystyle
\frac{1}{\sqrt{1-\sigma}}\,\Big(\tfrac12 \fc^s_{\bw\bw}\,+\,\tfrac12 \fc^s_{\bz\bz}\,+\,\tfrac12 (1-\sigma)^2 \fc^s_{\bv\bv}
\,-\,\fc^s_{\bw\bz}\,+\,(1-\sigma)\,(\fc^s_{\bz\bv}\,-\,\fc^s_{\bw\bv})\Big)\qquad\qquad
\\[5pt]
\displaystyle
\,=\,\f^s_{\bz\bz}\,+\,\tfrac12 (1-\sigma)^2 \f^s_{\bv\bv}\,-\,\f^s_{\bw\bz}\,.
\end{array}
$$
By substituting these expressions for those appearing in \eqref{eq:devfc-first} we eventually arrive at
\begin{equation}\label{eq:devfc}
\frac{\fc}{\sqrt{1-\sigma}}\begin{array}[t]{l}=
\f^s\,+\,\dfrac{\rho\,\h_s}{2}\,(1-\sigma)\,\f^s_\bv\\ [10pt]
\quad\,+\,\dfrac{\rho^2\,\h_s^2}{4}\,\Big(
-\Big(\sigma\,\dfrac{\frak{p}_s}{\frak{a}_s^2}\,+\,(1-\sigma)\,\Big(\dfrac{1}{\h_s}-\dfrac{\frak{b}_s}{\frak{a}_s^2}\Big)\Big)\,\f^s_\bv\,+\,
\dfrac{2\frak{b}_s}{\frak{a}_s^2}\,\f^s_\bz
\Big)\\ [10pt]
\quad
\,+\,\dfrac{\rho^2\,\h_s^2}{4}\,\Big( 
\f^s_{\bz\bz}\,+\,\frac12 (1-\sigma)^2 \f^s_{\bv\bv}
\,-\,\f^s_{\bw\bz}\Big)
\,+\,\dfrac{{\mathcal O}(\rho^3)}{\sqrt{1-\sigma}}\,.\end{array}
\end{equation}

Now, to proceed, we invoke the crucial Proposition~\ref{prop:asol}. Stated in rather abstract words, it provides us with coefficients of asymptotic expansions
$$\int_{0}^{1} \frac{g(x)}{\sqrt{x(x+\rho)}}\,\dif x\sim\sum_{k\geqslant 0} (a_k[g] \rho^k\,\ln\rho \,+\,b_k[g]\, \rho^k)\,,$$
where the brackets $[\cdot]$ are here to emphasize the function to which the expansion corresponds.  As a consequence, by combining with \eqref{eq:devfc}, we receive coefficients of the expansion of $\nabla\Action$. Namely, we find that the first vectors $\bA_k$'s are
$$
\begin{array}{l}\displaystyle
\bA_0\,=\,a_0[\f^s]\,=\,\,-\f^s_0\,=\,-\,\Y^s_0\,\bV_s\,,\\[5pt]
\displaystyle
\bA_1\,=\,a_0[\tfrac{\h_s}{2}\,(1-\sigma)\,\f^s_\bv]\,+\,a_1[\f^s]\,=\,-\,\tfrac{\h_s}{2}\,\f^s_{\bv,0}\,+\,\tfrac{1}{2}\,(\f^s)'(0)\,=\,0\,,
\end{array}$$
and
$$
\begin{array}[t]{rcl}
\bA_2&=&
\displaystyle 
\dfrac{\h_s^2}{4}\,a_0\left[-\Big(\sigma\,\dfrac{\frak{p}_s}{\frak{a}_s^2}\,+\,(1-\sigma)\,\Big(\dfrac{1}{\h_s}-\dfrac{\frak{b}_s}{\frak{a}_s^2}\Big)\Big)\,\f^s_\bv\,+\,
\dfrac{2\frak{b}_s}{\frak{a}_s^2}\,\f^s_\bz\right]
\\[10pt]
&&\displaystyle
\quad\,+\,\dfrac{\h_s^2}{4}\,a_0\left[\f^s_{\bz\bz}\,+\,\frac12 (1-\sigma)^2 \f^s_{\bv\bv}\,-\,\f^s_{\bw\bz}\right]
\,+\,\frac{\h_s}{2}\,a_1[(1-\sigma)\,\f^s_\bv]\,+\,a_2[\f^s]\\ [15pt]
&=&
\displaystyle
\dfrac{\h_s^2}{4}\Big(\dfrac{1}{\h_s}-\dfrac{\frak{b}_s}{\frak{a}_s^2}\Big)\,\f^s_{\bv,0}\,-\,
\dfrac{\h_s^2\frak{b}_s}{2\frak{a}_s^2}\,\f^s_{\bz,0}\,-\, \dfrac{\h_s^2}{4} \f^s_{\bz\bz,0}\,-\,\dfrac{\h_s^2}{8} \f^s_{\bv\bv,0}\\[10pt]
&&
\displaystyle
\quad\,+\,\dfrac{\h_s^2}{4} \f^s_{\bw\bz,0}
\,+\,\frac{\h_s}{4}\,((\f^s_\bv)'(0)\,-\,\f^s_{\bv,0})\,-\,\frac{3}{16}\,(\f^s)''(0)\\[15pt]
&=&
\displaystyle
-\,\dfrac{\h_s^2}{4}\,\left(\dfrac{\frak{b}_s}{\frak{a}_s^2}\,\f^s_{\bv,0}\,+\,
\dfrac{2\frak{b}_s}{\frak{a}_s^2}\,\f^s_{\bz,0}\,+\,  \f^s_{\bz\bz,0}\,-\, \f^s_{\bw\bz,0}\,+\, \frac{1}{4} \f^s_{\bv\bv,0}\right)\,,
\end{array}$$
where $'$ denotes the differentiation with respect to $\sigma$ and we have used that $(\g^s)'\,=\,\h_s\,\g^s_\bv$, for any $\g$.  We recall that the subscript $0$ stands for evaluation at $\sigma=0$. As for the first $\bB_k$'s, they read
$$\bB_0\,=\,b_0[\f^s]\,=\,\f^s_0\,\ln 4\,+\,\int_{0}^{1}\,\frac{\f^s-\f^s_0}{\sigma}\,\dif\sigma\,,$$
$$
\begin{array}[t]{rcl}
\bB_1&=&
\displaystyle 
b_0[\tfrac{\h_s}{2}\,(1-\sigma)\,\f^s_\bv]\,+\,b_1[\f^s]\\[10pt]
&=&\displaystyle
\frac{1}{2}\,(\f^s)'(0)\,\ln 4
\,+\,\frac{1}{2}\,\int_{0}^{1}\,\frac{(1-\sigma)\,(\f^s)'(\sigma)\,-\,(\f^s)'(0)}{\sigma}\,\dif\sigma
\\[5pt]
&&\displaystyle
\quad\,+\,\frac{1}{2}\left(\f^s_0+(1-\ln 4)\,(\f^s)'(0)\right)
-\,\frac{1}{2}\,
\int_{0}^{1}\,\frac{\f^s\,-\,\f^s_0\,-\,\sigma\,(\f^s)'(0)}{\sigma^2}\,\dif\sigma\\[10pt]
&=&
\displaystyle
\frac{1}{2}\left(\f^s_0+(\f^s)'(0)\right) -\,\displaystyle\frac12\,
\int_{0}^{1}\,\frac{\f^s\,-\,\f^s_0\,-\,\sigma\,(1-\sigma)\,(\f^s)'(\sigma)}{\sigma^2}\,\dif\sigma\\[10pt]
&=&
\displaystyle
\frac{1}{2}\left(\f^s_0+(\f^s)'(0)\right)\,+\,\frac{1}{2}\left[(1-\sigma)\,\frac{\f^s(\sigma)\,-\,\f^s(0)}{\sigma} \right]_{\sigma=0}^{\sigma=1}\\[10pt]
&=&\displaystyle
\frac{1}{2}\f^s_0\,=\,\frac{1}{2}\,\Y^s_0\,\bV_s\,,
\end{array}$$
and
$$\begin{array}[t]{rcl}\bB_2&=&
\displaystyle \dfrac{\h_s^2}{4}\,b_0\left[-\Big(\sigma\,\dfrac{\frak{p}_s}{\frak{a}_s^2}\,+\,(1-\sigma)\,\Big(\dfrac{1}{\h_s}-\dfrac{\frak{b}_s}{\frak{a}_s^2}\Big)\Big)\,\f^s_\bv\,+\,
\dfrac{2\frak{b}_s}{\frak{a}_s^2}\,\f^s_\bz\right]\\[10pt]
&&\displaystyle
\quad\,+\,\dfrac{\h_s^2}{4}\,b_0\left[\f^s_{\bz\bz}\,+\,\tfrac12 (1-\sigma)^2 \f^s_{\bv\bv}\,-\,\f^s_{\bw\bz}\right]
\,+\,\frac{\h_s}{2}\,b_1[(1-\sigma)\,\f^s_\bv]\,+\,b_2[\f^s]\,.\\ [20pt]
\displaystyle
\\
\end{array}$$
We do not make the latter more explicit here because we eventually `only' need its contribution to the one `coefficient' of $\bD_2$ that we will compute and this contribution is shown to vanish by an argument independent of the precise form of $\bB_2$. 

We are now ready to prove the following.
\begin{theorem}\label{thm:ashess}\label{thm:ashessun}
With notational conventions \ref{not:derivatives}--\ref{not:UVWZ}--\ref{not:abcp}, under Assumption \ref{as:main}, the Hessian of the action has the asymptotic behavior 
\begin{equation}\label{eq:ashess}
\frac{\nabla^2\Action}{\Y^s_0}
\begin{array}[t]{l}\displaystyle\,=\,\frac{2\frak{a}_s^2}{\h_s^2}\,\frac{1+\rho}{\rho^2}\;\bV_s\otimes \,\bV_s
\,+\,\left( \alpha_s\,\bV_s\otimes \,\bV_s\,+\,\beta_s\,\left(\bV_s\otimes \bW_s\,+\,\bW_s\otimes \bV_s\right)\right)\,\ln\rho\\ [15pt]
\displaystyle\,+\left(\dfrac{\frak{a}_s^2}{2}\,\bW_s\otimes \bW_s\,+\,\bT_s\otimes \bT_s
\,+\,\frac{\frak{a}_s^2}{4}\,(\bZ_s\otimes \bV_s\,+\,\bV_s\otimes \bZ_s)\right) \ln\rho\\ [15pt]
\displaystyle\,+\;\bD_s\,+\, {\mathcal O}\big(\rho\ln\rho\big)
\end{array}
\end{equation}
when $\rho\,=\,{(\bv_2-\bv_1)}/{(\bv_3-\bv_2)}$ goes to zero, where 
$$
\begin{array}{rcl}
\alpha_s&:=&\displaystyle
\frac{1}{\Y^s_0}\,\Big(\,\frak{b}_s\,(\Y_{\bv,0}^s\,+\,2\,\Y_{\bz,0}^s)\,+\,\frak{a}_s^2\,\Big(\frac14\,\Y_{\bv\bv,0}^s\,+\,
\Y_{\bz\bz,0}^s\,-\,\Y_{\bw\bz,0}^s\Big)\Big)\,,\\[5pt]
\beta_s&:=&\displaystyle
\frak{b}_s\,+\,\frac{\frak{a}_s^2}{2}\,\frac{\Y_{\bv,0}^s}{\Y^s_0}\,
\end{array}
$$
may be computed explicitly in terms of derivatives of $\Potential$ and $\kappa$ at $\bv_s$ from \eqref{eq:redpot}-\eqref{eq:defY}, and the symmetric matrix $\bD_s$ is such that 
$$\bS_s\cdot \bD_s \,\bS_s\,=\,\Mom''\,,$$
where $\Mom''$ is the second order derivative of the Boussinesq momentum $\Mom$ with respect to the solitary wave speed $\speed$ at fixed endstate $\bU_s$ and 
$$\bS_s:=
\left(\begin{array}{c} -\,\impulse_s \\
\bB^{-1}\bU_s \\ -1\end{array}\right)\,.$$
\end{theorem}

We recall that $\bT_s=0$ and $\bB^{-1}\bU_s=\bv_s$ in the case $N=1$.

\begin{remark}
The motivation for expanding $\nabla^2\Action$ up to the ${\mathcal O}\big(1\big)$ term $\D_s$ is that the expansion up to order $\ln\rho$ has rank $N+1$ (at best). Moreover, the computed $\bS_s\cdot \bD_s \bS_s$ is precisely the `coefficient' whose non vanishing ensures invertibility of the expansion up to order $1$ - when $\rho$ is sufficiently small. This will become clearer in Corollary \ref{cor:stabsol} hereafter.
\end{remark}

\begin{remark}
By definition of $\rho$ we see that
$$
\frac{1+\rho}{\rho^2}\,=\,\frac{(\bv_3-\bv_1)\,(\bv_3-\bv_2)}{(\bv_2-\bv_1)^2}$$
so that the principal part of the expansion of $\nabla^2\Action$ is just proportional to
$$\frac{(\bv_3-\bv_1)\,(\bv_3-\bv_2)}{(\bv_2-\bv_1)^2}\,\bV_s\otimes \,\bV_s\,.$$
\end{remark}

\begin{remark}
Remarkably enough, despite the scale separation, the expression of the term of order $\ln\rho$ in \eqref{eq:ashess} for the soliton limit  is analogous to the leading order term in \eqref{eq:asharmhess} for the harmonic limit, with the state $\bv_s$ replacing $\bv_0$.
\end{remark}

\begin{proof}
In terms introduced above we need to compute the first $\bC_k$, $\bD_k$. From foregoing computations we readily obtain the first three of them,
$$\bD_0=\bA_0\otimes \bY_{0}\,=\,2\,\Y^s_0\,\frac{\frak{a}_s^2}{\h_s^2}\,\bV_s\otimes \bV_s\,,
\qquad\qquad\qquad
\bC_{1}=\bA_1\otimes \bY_0\,=\,0\,,$$
and
$$\bD_1=\bA_0\otimes \bY_1\,+\,(\bA_1+\bB_1)\otimes \bY_0
\,=\,2\,\Y^s_0\,\frac{\frak{a}_s^2}{\h_s^2}\,\bV_s\otimes \bV_s\,.
$$
This already gives the expansion of $\nabla^2\Action$ up to ${\mathcal O}(\ln\rho)$ remainders.

The explicitation of 
$$
\bC_2=2\,\bA_2\otimes \bY_0\,+\,\nabla_{\tot} \bA_0
$$
requires further work. To make $\bA_2$ more explicit we first observe that by the Leibniz rule
$$
\begin{array}{l}\displaystyle
\f^s_{\bv,0}\,=\,\Y^s_{\bv,0}\,\bV_s\,+\,\Y^s_0\,\bW_s\,,
\qquad\qquad
\f^s_{\bv\bv,0}\,=\,\Y^s_{\bv\bv,0}\,\bV_s\,+\,2\,\Y^s_{\bv,0}\,\bW_s\,+\,\Y^s_{0}\,\bZ_s\,,\\[5pt]
\f^s_{\bz,0}\,=\,\Y^s_{\bz,0}\,\bV_s\,,\qquad\qquad
\f^s_{\bz\bz,0}\,=\,\Y^s_{\bz\bz,0}\,\bV_s\,,\qquad\qquad
\f^s_{\bw\bz,0}\,=\,\Y^s_{\bw\bz,0}\,\bV_s\,.
\end{array}
$$
Given that $\bA_0\,=\,-\,\Y^s_0\,\bV_s$ and $\bV_s\,=\,\nabla\ZZ^s_0$, we also have 
$$
\nabla_{\tot}\bA_0\,=\,-\,\bV_s\otimes \nabla_{\tot}\Y^s_0\,-\,\Y^s_0\,\nabla_{\tot}\bV_s\,,$$
with
$$
\begin{array}{rcl}\displaystyle
\nabla_{\tot}\Y^s_0&=&\displaystyle
\nabla\Y^s_0\,+\,(\Y^s_{\bv,0}\,+\,2\,\Y^s_{\bz,0})\,\nabla\bv_s\,,\\[5pt]
\nabla_{\tot}\bV_s&=&\,-\,\bT_s\otimes \bT_s\,+\,\bW_s\otimes \nabla\bv_s\,.
\end{array}
$$
Besides, we may proceed as we have derived $\nabla\bv_0$ in the harmonic limit and differentiate $\partial_\bv\Potential(\bv_s(\speed,\blambda);\speed,\blambda)\,=\,0$ to obtain 
\begin{equation}
\label{eq:gradvs}
\nabla \bv_s\,=\,-\,\frac{\frak{a}_s^2}{2}\,\bW_s\,.
\end{equation}
Finally, we also need to express $\nabla\Y^s$. Again we may proceed as in the harmonic limit to obtain 
\begin{equation}
\label{eq:nablaY}
\frac{\nabla \Y^s_0}{\Y^s_{0}}
\,=\,-\,\frac{\nabla\Redpot^s_0}{2\,\Redpot^s_0}
\,=\,-\,\frac{\frak{a}_s^2}{4}\,\,\bZ_s\,.
\end{equation}
In a similar way, we check that 
\begin{equation}
\label{eq:bYz}
\frac{\Y^s_{\bz,0}}{\Y^s_0}\,=\,-\,\frac{\Redpot^s_{\bz,0}}{2\,\Redpot^s_0}
\,=\,\frac{\frak{b}_s}{\frak{a}_s^2}\,.
\end{equation}
\emph{In fine}, we may gather all the terms involved in $\bC_2$ since on the one hand
$$\begin{array}{rcl}\displaystyle
2\,\bA_2\otimes \bY_0&=&\displaystyle
\frak{a}_s^2\,\left(\dfrac{\frak{b}_s}{\frak{a}_s^2}\,\f^s_{\bv,0}\,+\,
\dfrac{2\frak{b}_s}{\frak{a}_s^2}\,\f^s_{\bz,0}\,+\,  \f^s_{\bz\bz,0}\,-\, \f^s_{\bw\bz,0}\,+\, \tfrac{1}{4} \f^s_{\bv\bv,0}\right)\,\otimes \bV_s\\[10pt]
&=&\displaystyle
\frak{a}_s^2\,\left(\dfrac{\frak{b}_s}{\frak{a}_s^2}\,\Y^s_{\bv,0}\,+\,
\dfrac{2\frak{b}_s}{\frak{a}_s^2}\,\Y^s_{\bz,0}\,+\,  \Y^s_{\bz\bz,0}\,-\, \Y^s_{\bw\bz,0}\,+\, \tfrac{1}{4} \Y^s_{\bv\bv,0}\right)\,\bV_s\otimes \bV_s\\[5pt]
&&\qquad\,+\,
\frak{a}_s^2\,\left(\dfrac{\frak{b}_s}{\frak{a}_s^2}\,\Y^s_{0}\,+\,\tfrac{1}{2} \Y^s_{\bv,0}\right)\,\bW_s\otimes \bV_s\,+\,
\frac{\frak{a}_s^2}{4}\,\Y^s_{0}\,\bZ_s\otimes \bV_s\,,
\end{array}
$$
and on the other hand
$$
\begin{array}{l}\displaystyle
\nabla_{\tot}\bA_0\,=\,
-\,\bV_s\otimes \nabla\Y^s_0\,-\,(\Y^s_{\bv,0}\,+\,2\,\Y^s_{\bz,0})\,\bV_s\otimes \nabla\bv_s
\,+\,\Y^s_0\,\bT_s\otimes \bT_s\,-\,\Y^s_0\,\bW_s\otimes \nabla\bv_s
\\[10pt]
\quad=\displaystyle
\frac{\frak{a}_s^2}{4}\Y^s_{0}\bV_s\otimes \bZ_s
+\frac{\frak{a}_s^2}{2}\big(\Y^s_{\bv,0}
+\frac{2\frak{b}_s}{\frak{a}_s^2}\Y^s_0\big)\bV_s\otimes \bW_s
+\Y^s_0\bT_s\otimes \bT_s
+\frac{\frak{a}_s^2}{2}\Y^s_0\bW_s\otimes \bW_s.
\end{array}$$
Eventually this yields the claimed expansion up to ${\mathcal O}(1)$ remainders.

Unsurprisingly, the calculation of
$$\bD_2=\bA_0\otimes \bY_2\,+\,(\bA_1+\bB_1)\otimes \bY_1\,+\,(\bA_2+2\bB_2)\otimes \bY_0\,+\,\nabla_{\tot}\bB_0$$
is even more demanding. The first two terms are not too complicated. They reduce to
$$
\begin{array}{rcl}
\bA_0\otimes \bY_2\,+\,(\bA_1+\bB_1)\otimes \bY_1
&=&\displaystyle
\frac{\Y^s_0}{\h_s}\,
\Big(2\frak{b}_s\,+\,\frak{p}_s\,+\,
\frac{4\frak{c}_s}{\frak{a}_s}
\Big)\,\bV_s\otimes\bV_s\\[5pt]
&+&\displaystyle
\frac{\Y^s_0}{\h_s}\,\bV_s\otimes
\Big(\frak{p}_s\,\bV^s\,+\,\frac12\,\Big(3\frak{b}_s\h_s+\frak{a}_s^2\Big)\,\bW_s
\,+\,\frac{\h_s\frak{a}_s^2}{4}\,\bZ_s\Big)\,.
\end{array}
$$
At this stage we leave to the interested reader the derivation of a complete - and lengthy! - formula for $\bD_2$ but stress that this may be carried out along the lines of computations explicitly performed above.

Instead, we focus now on the value of $\bS_s\cdot \bD_2 \bS_s$. The main computational gain is that the vector 
$$\bS_s\,=\,\left(\begin{array}{c} -\,\impulse_s \\
\bB^{-1}\bU_s \\ -1\end{array}\right)$$
is orthogonal to $\bV_s$ as may be checked by direct inspection and is explicitly contained in Lemma~\ref{lem:orthVWZT} below. As a consequence, $\bS_s$ is also orthogonal to $\bY_0$, $\bY_1$, and $\bA_0$ so that
$$\bS_s\cdot \bD_2\, \bS_s\,=\,\bS_s\cdot (\nabla_{\tot}\bB_0)\,\bS_s\,.$$
In addition, we point out the fact - also straightforward to establish and stated in Lemma~\ref{lem:orthVWZT} - that $\bS_s$ is also orthogonal to $\bW_s$ and $\bT_s$. At this stage the reader could rightfully wonder whether all these cancellations are accidental. Exactly the opposite is true, they are the reason why we are interested in $\bS_s$ and $\bS_s\cdot\Action\bS_s$. A direct consequence of these orthogonalities is
$$
\bS_s\cdot (\nabla_{\tot}\bV_s)\,\bS_s\,=\,0\,.
$$
Since 
$$
\begin{array}{rcl}
\nabla_{\tot}\bB_0&=&\displaystyle
(\ln 4)\,\nabla_{\tot}\f^s_0\,+\,\int_{0}^{1}\,\frac{\nabla_{\tot}\f^s-\nabla_{\tot}\f^s_0}{\sigma}\,\dif\sigma\,,\\[5pt]
\nabla_{\tot}\f^s_0&=&
\bV_s\otimes \nabla_{\tot}\Y^s_0\,+\,\Y^s_0\,\nabla_{\tot}\bV_s\,,
\end{array}
$$
this leaves
$$\bS_s\cdot \bD_2\,\bS_s\,=\,\int_{0}^{1}\,\,\bS_s\cdot (\nabla_{\tot}\f^s)\,\bS_s\,\frac{\dif\sigma}{\sigma}\,.$$
To go even further, note that $\bS_s$ is also orthogonal to $\nabla v_s$ so that we are left with
$$\bS_s\cdot \bD_2\,\bS_s\,=\,\int_{0}^{1}\,\,\bS_s\cdot \Big(\nabla \f^s\,+\,\sigma \,\f^s_\bv\otimes \nabla\bv^s
\Big)\,\bS_s\,\frac{\dif\sigma}{\sigma}\,,$$
still to be identified with $\Mom''$.

In the order direction, we now provide - for comparison - an explicit formula for $\Mom''$. In order to avoid confusion with the partial derivative with respect to $\speed$ at constant $\blambda$ and $\mu$ - the one that is involved in the notation $\nabla$ - or with the differentiation with respect to $\sigma$, let us introduce the local convention that $\dif_\speed$ denotes the differentiation operator with respect to $\speed$ at fixed endstate $\bU_s$. To start with, we recall from Remark~\ref{rk:def-Mom} that 
$$
\Mom(\speed,\bU_s)=\int_{-\infty}^{+\infty} (\Ham[\ubU^s]+\speed \Impulse(\ubU^s) + \blambda_s \cdot \ubU^s +\mu_s)\,\dif \xi\,,$$
with 
$$\blambda_s = - \nabla_{\bU}(\Ham+\speed\Impulse)(\bU_s,0)\,,\qquad\quad
\mu_s = - \blambda_s\cdot \bU_s\,-\,(\Ham+\speed\Impulse)(\bU_s,0)\,.$$
A direct differentiation yields
$$
\dif_\speed\blambda_s\,=\,
- \nabla_{\bU}\Impulse(\bU_s)\,=\,-\,\bB^{-1}\bU_s\,,
\qquad\qquad\qquad
\dif_\speed\mu_s\,=\,
-\,\bU_s\cdot\bB^{-1}\bU_s
\,-\,\Impulse(\bU_s)\,.
$$
Likewise, by using that $\Euler (\Ham+\speed\Impulse)[\ubU^s] \,+\,\blambda_s\,=\,0$, one obtains after a change of variables
$$
\begin{array}{rcl}
\dif_\speed\Mom&=&\displaystyle
\int_{-\infty}^{+\infty} \big( \Impulse(\ubU^s) - \Impulse(\bU_s) - \bU_s\cdot \bB^{-1} (\ubU^s-\bU_s)\big)\,\dif \xi\\[10pt]
&=&\displaystyle
\int_{0}^{1} \big( \Impulse(\U_s) - \Impulse(\bU_s) - \bU_s\cdot \bB^{-1}  (\U_s-\bU_s)\big)\,\Y^s\,\frac{\dif\sigma}{\sigma}
\end{array}
$$
where in the latter integral functions of $\bv$ are evaluated at $\bv=\bv_s+\sigma (\bv^s-\bv_s)$, and $\U(v):=\bv$ in the case $N=1$ whereas, in the case $N=2$, $\transp{\U(v)}:=\left(\bv\,,g(\bv;\speed,\blambda)
\right)$. Note that in performing the foregoing change of variables we have used
$$
\frac{\Y(\bv,\bv_s,\bv_s;\speed,\blambda)}{\bv-\bv_s}
\,=\,\sqrt{\frac{2\kappa(\bv)}{\mu-\Potential(\bv;\speed,\blambda)}}\,.
$$
Now observe that, since
$$\f^s\,=\,\Y^s\,\transp{\left(1,\transp{\U_s},\Impulse(\U_s)
\right)}
\qquad\mbox{and}\qquad
\bS_s=
\transp{\left(-\,\impulse_s,\transp{(\bB^{-1}\bU_s)},-1\right)}
$$
one finds by direct inspection that 
$$ \big( \Impulse(\U_s) - \Impulse(\bU_s) - \bU_s\cdot \bB^{-1}  (\U_s-\bU_s)\big)\,\Y^s\,=\,-\,\bS_s\cdot \f^s$$
so that 
$$
\dif_\speed\Mom\,=\,-\,\int_{0}^{1} \bS_s\cdot \f^s\; \frac{\dif\sigma}{\sigma}\,.
$$
To differentiate once more notice that from the value of $\dif_c\blambda_s$ computed above follows that for any function $a$ of $(\speed,\blambda)$
$$
\dif_\speed a\,=\,\partial_\speed a\,+\,\nabla_{\blambda} a\cdot \dif_\speed\blambda_s\,=\,-\,\nabla a\cdot \bS_s\,.
$$ 
As a consequence
$$
\dif_\speed^2\Mom\,=\,\int_{0}^{1} 
\bS_s\cdot \Big(\nabla \f^s\,+\,\sigma \,\f^s_\bv\otimes \nabla\bv^s
\Big)\,\bS_s\,\frac{\dif\sigma}{\sigma}\,=\,\bS_s\cdot \bD_2\,\bS_s\,.$$
This completes the proof of Theorem~\ref{thm:ashess}, in the statement of which the matrix $\D_2$ is denoted by $\D_s$ for notational consistency. 
\end{proof}

\section{Proof of stability/instability results}\label{s:conclusions}

\subsection{Orthogonality}\label{s:ortho}

In order to exploit Theorem~\ref{thm:stabcrit} in the harmonic and soliton limits, we need a way to deduce the negative signature of $\nabla^2\Action$ from its expansions in these limits given by Theorems~\ref{thm:asharmhess} and~\ref{thm:ashess}. It will be based on a - linear - change of variables for the quadratic form associated with $\nabla^2\Action$, built up thanks to a symmetric $(N+2)\times (N+2)$ matrix $\Sb$ associated with another quadratic form, namely the impulse.

This symmetric matrix is non singular and defined as
$$\Sb\,:=\,\left(\begin{array}{c|ccc|c} 0 &  & 0 & & -1 \\ \hline & & & & \\ 0 & & \bB^{-1} & & 0 \\ & &  & & \\ \hline -1 & & 0 & & 0 \end{array}\right)\,.$$
It is closely linked to the impulse $\Impulse$ in that for any vector $\bX\in \R^{N+2}$ of the form
$$\bX=\left(\begin{array}{c} 1 \\ \bU \\ q \end{array}\right)$$
we have 
\begin{equation}\label{def:Sb}
\frac12\,\bX\cdot \Sb \bX\,=\,\Impulse(\bU)-q\,.
\end{equation}

The important role of the matrix $\Sb$ 
is already clearly apparent in the modulated equations associated with periodic waves under consideration (see \cite{BNR-GDR-AEDP}). More importantly for us here - but not unrelated - $\Sb$ is associated with some crucial orthogonality properties for the vectors $\bV_i$, $\bW_i$, $\bZ_i$, and $\bT_i$ involved in the asymptotic expansion of $\nabla^2\Action$  in the harmonic limit ($i=0$) and in the soliton limit ($i=s$). As a consequence, the crucial role played by the vector $\bS_s$ introduced in the previous section actually hinges on 
$$
\bS_s\,=\,\Sb\,\bV_s\,.
$$

\begin{lemma}\label{lem:orthVWZT}
For $i=0$ or $s$ the vectors  $\bV_i$, $\bW_i$, $\bZ_i$, and $\bT_i$ from Notation  \ref{not:UVWZ} are such that
\begin{equation}\label{eq:orthVWZT}
\left\{\begin{array}{l}
\bV_i \cdot \Sb\,\bV_i\,=\,0\,,\ \ 
\bV_i\,\cdot \Sb\,\bW_i\,=\,0\,,\ \ 
\bV_i\,\cdot \Sb\,\bT_i\,=\,0\,,\ \ 
\bV_i\,\cdot \Sb\,\bZ_i\,=\,-\,\bW_i\cdot \Sb\,\bW_i\,,\\ [10pt]
\bT_i \cdot \Sb\,\bT_i\,=\,0\,,\qquad
\bT_i \cdot \Sb\,\bZ_i\,=\,0\,,\\ [10pt]
\bE\cdot \bV_i\,=\,1\,,\qquad
\bE\cdot \bW_i\,=\,0\,,\qquad
\bE\cdot \bZ_i\,=\,0\,,\qquad
\bE\cdot \bT_i\,=\,0\,,
\end{array}\right.
\end{equation}
where we have denoted by $\bE$ the vector
$$\bE:= \left(\begin{array}{c} 1 \\ 0 \\ 0 \end{array}\right)\,=\,\Sb\,\bF\,,\qquad\qquad\bF:=\left(\begin{array}{c} 0 \\ 0 \\ -1 \end{array}\right)\,.$$
\end{lemma}

\begin{proof}
Of course these relations may be checked by direct inspection. It is however more instructive to relate them to fundamental properties of $\Sb$ and $\ZZ$. Most of them are consequences of the identity
$$\nabla\ZZ \cdot \Sb\,\nabla\ZZ\,=\,0\,,$$
which readily stems from \eqref{def:Sb} and \eqref{eq:derZNtwo}/\eqref{eq:derZNone}. By differentiation and thanks to the symmetry of $\Sb$ this first identity yields
$$\nabla\ZZ \cdot \Sb\,\nabla\ZZ_\bv\,=\,0\,,\qquad\qquad
\nabla\ZZ \cdot \Sb\,\nabla^2\ZZ\,=\,0\,,$$
and also
$$\nabla\ZZ \cdot \Sb\,\nabla\ZZ_{\bv\bv}\,=\,-\,\nabla\ZZ_\bv \cdot \Sb\,\nabla\ZZ_\bv\,.$$
When evaluated at $\bv=\bv_i$ these four relations give
$$\bV_i \cdot \Sb\,\bV_i\,=\,0\,,\quad
\bV_i\,\cdot \Sb\,\bW_i\,=\,0\,,\quad
\bV_i\,\cdot \Sb\,\bT_i\,=\,0\,,\quad
\bV_i\,\cdot \Sb\,\bZ_i\,=\,-\,\bW_i\cdot \Sb\,\bW_i\,.$$
Likewise, from
$$
\bE\cdot \nabla\ZZ\,=\,1\,,
$$
which is essentially the definition of $\d_\mu$, one obtains
$$
\bE\cdot \bV_i\,=\,1\,,\qquad
\bE\cdot \bW_i\,=\,0\,,\qquad
\bE\cdot \bZ_i\,=\,0\,,\qquad
\bE\cdot \bT_i\,=\,0\,.$$
Furthermore, since the quadratic form induced by $\Sb$ vanishes on $\{(0,0)\}\times\R^N$, from the special form of \eqref{eq:derZNtwo}/\eqref{eq:derZNone} we also infer 
$$\nabla\ZZ_{vv}\cdot \Sb\,\nabla\ZZ_{vv}\,=\,0\,,
\qquad 
\nabla\ZZ_{vv}\cdot \Sb\,\nabla^2\ZZ\,=\,0\,,
$$
hence
$$\bT_i \cdot \Sb\,\bT_i\,=\,0\,,\qquad
\bT_i \cdot \Sb\,\bZ_i\,=\,0\,.$$
\end{proof}

\begin{remark}\label{rem:bases}
A useful observation for the sequel is that the vectors
$(\Sb\bV_i,\Sb\bW_i,\Sb\bT_i,\bE)$ form a basis of $\R^4$ in the case $N=2$, and 
$(\Sb\bV_i,\Sb\bW_i,\bE)$ a basis of $\R^3$ in the case $N=1$.
This is due to the invertibility of $\Sb$ and the fact that both $(\bV_i\;\bW_i\;\bT_i\;\bF)$ in the case $N=2$ and $(\bV_i\;\bW_i\;\bF)$ in the case $N=1$ are lower triangular matrices with nonzero entries in the diagonal - in fact $1$ or $-1$ as diagonal entries except for the one coming from $\bT_i$ given by $1/\sqrt{\tau(v_i)}$.
\end{remark}

\subsection{Stability in the harmonic limit}\label{s:stabharm}

From Eq.~\eqref{eq:asharmhess} in Theorem \ref{thm:asharmhess} we know that there exist real numbers $\alpha_0$, $\beta_0$, and a positive number $\gamma_0$ such that 
$$\frac{1}{\pi\Y^0}\,\nabla^2\Action\,=\,\begin{array}[t]{l}\alpha_0\,\bV_0\otimes \bV_0\,+\,\beta_0\,(\bV_0\otimes\bW_0\,+\,\bW_0\otimes\bV_0)\,-\,\bT_0\otimes\bT_0\\ [10pt]
\,+\,2\,\gamma_0\,\bW_0\otimes\bW_0\,+\,
\gamma_0\,(\bV_0\otimes\bZ_0\,+\,\bZ_0\otimes\bV_0)
\,+\,{\mathcal O}(\delta)
\\ [10pt]
\end{array}$$
when $\delta$ goes to zero.
By Remark \ref{rem:bases} the matrix $\Pb_0$ defined as
$$\Pb_0:=(\bE\;\;\Sb\bV_0\;\;\Sb\bT_0\;\;\Sb\bW_0) \mbox{ for } N=2\,,\qquad\qquad
\Pb_0:=(\bE\;\;\Sb\bV_0\;\;\Sb\bW_0) \mbox{ for } N=1\,,$$
is nonsingular. Now, thanks to \eqref{eq:orthVWZT} in Lemma~\ref{lem:orthVWZT} the symmetric matrix $\Sigma_0$ found in 
$$\transp{\Pb}_0\,\nabla^2\Action\,\Pb_0\,=\,\pi\Y^0\,\Sigma_0
\,+\,{\mathcal O}(\delta)\,,
$$
is given by
$$\Sigma_0:=
\left(\begin{array}{c|c|c|c} \alpha_0 &  -\,\gamma_0\,\bW_0\cdot\Sb\bW_0& \beta_0\,\bT_0\cdot\Sb\bW_0 &  \beta_0\,\bW_0\cdot\Sb\bW_0 \\
 &  &  &  +\gamma_0\,\bW_0\cdot\Sb\bZ_0 \\  \hline
-\,\gamma_0\,\bW_0\cdot\Sb\bW_0 & 0 & 0 &  0 \\ \hline 
\beta_0\,\bT_0\cdot\Sb\bW_0 &  0 & 2 \gamma_0\,(\bT_0\cdot\Sb\bW_0)^2 &  2 \gamma_0\,(\bT_0\cdot\Sb\bW_0) \\ 
& & & \times  (\bW_0\cdot\Sb\bW_0) \\ \hline
 \beta_0\,\bW_0\cdot\Sb\bW_0 & 0 & 2 \gamma_0\,(\bT_0\cdot\Sb\bW_0) & 2 \gamma_0\,(\bW_0\cdot\Sb\bW_0)^2 \\
 +\, \gamma_0\,\bW_0\cdot\Sb\bZ_0 & & \times (\bW_0\cdot\Sb\bW_0) & - (\bT_0\cdot\Sb\bW_0)^2
 \end{array}\right)
$$
in the case $N=2$ and 
$$\Sigma_0:=
\left(\begin{array}{c|c|c} \alpha_0 &  -\,\gamma_0\,\bW_0\cdot\Sb\bW_0&  \beta_0\,\bW_0\cdot\Sb\bW_0 \\
 &  &    +\gamma_0\,\bW_0\cdot\Sb\bZ_0 \\  \hline
-\,\gamma_0\,\bW_0\cdot\Sb\bW_0 & 0 & 0  \\ \hline 
 \beta_0\,\bW_0\cdot\Sb\bW_0 & 0 & 2 \gamma_0\,(\bW_0\cdot\Sb\bW_0)^2 \\
 +\, \gamma_0\,\bW_0\cdot\Sb\bZ_0 &  & 
 \end{array}\right)
$$
in the case $N=1$.

\begin{corollary}
\label{cor:stabharm}
Under Assumption \ref{as:main} and those of Theorem \ref{thm:stabcrit}, if 
$$\Y^0\,\alpha_0=\frak{b}_0\,(\Y_\bv^0+\,2\,\Y_\bz^0)\,+\,\frak{a}_0^2\,\Big(\frac14\,\Y_{\bv\bv}^0\,+\,
\Y_{\bz\bz}^0\,-\,\Y_{\bw\bz}^0\Big)\,,$$
is nonzero and, in addition if $N=2$ $g_\bv^0\neq0$, then small amplitude waves about $\bU_0$ are conditionally orbitally stable with respect to co-periodic perturbations.
\end{corollary}

\begin{proof}
The point is to apply Theorem \ref{thm:stabcrit} and thus to prove that $\negsign(\nabla^2\Action)=N$ for $\delta$ small enough.
Since $\Y^0>0$ by definition, the computation made above shows that this amounts to proving that the negative signature of any matrix of the form
$\Sigma(\delta)=\Sigma_0+{\mathcal O}(\delta)$
is equal to $N$ for $\delta$ small enough. 

By Sylvester's rule, $\negsign(\Sigma(\delta))$ is equal to the number of sign changes in the sequence starting with $+1$ and continuing with the principal minors of $\Sigma(\delta)$.
The first principal minor of $\Sigma(\delta)$ is just $\alpha_0
+{\mathcal O}(\delta)$, 
of indeterminate sign. The second one is 
$$-\gamma_0^2\,(\bW_0\cdot\Sb\bW_0)^2
\,+\,{\mathcal O}(\delta)
$$ and is therefore negative for $\delta$ small enough, since $\gamma_0$ is nonzero and $\bW_0\cdot\Sb\bW_0$ as well - indeed, $\bW_0\cdot\Sb\bW_0= b$ in the case $N=1$ and $\bW_0\cdot\Sb\bW_0=2 b g_\bv^0$ in the case $N=2$.

To complete the computation of $\negsign(\Sigma(\delta))$ we need to distinguish between the case $N=1$ and the case $N=2$. The simplest one is the former. From the definition of $\Sigma_0$ in this case, we see that
$$\det\Sigma(\delta)\,=\,-\,2\,\gamma_0^3\,(\bW_0\cdot\Sb\bW_0)^4
\,+\,{\mathcal O}(\delta)
$$
is negative for $\delta$ small enough since $\gamma_0$ is positive and $\bW_0\cdot\Sb\bW_0$ is nonzero. So the complete sequence of signs is $(+,\mbox{sign}\alpha_0,-,-)$, which involves exactly one change of sign. This implies that $\negsign(\Sigma(\delta))=1$ in the case $N=1$.

From the definition of $\Sigma_0$ in the case $N=2$ we see that the third principal minor of $\Sigma(\delta)$ is
$$\,-\,2\,\gamma_0^3\,(\bT_0\cdot\Sb\bW_0)^2\,(\bW_0\cdot\Sb\bW_0)^2
\,+\,{\mathcal O}(\delta)
\,.$$
This is negative for $\delta$ small enough because $\bT_0\cdot\Sb\bW_0=b/\sqrt{\tau(\bv_0)}$ is nonzero. Finally, in the case $N=2$
$$
\det\Sigma(\delta)\,=\,2\,\gamma_0^3\,(\bW_0\cdot\Sb\bW_0)^2\,(\bT_0\cdot\Sb\bW_0)^4
\,+\,{\mathcal O}(\delta)\,,
$$
which is positive for $\delta$ small enough. 
So the complete sequence of signs is $(+,\mbox{sign}\alpha_0,-,-,+)$, which involves exactly two changes of sign, hence $\negsign(\Sigma(\delta))=2$ in the case $N=2$.
\end{proof}

We can actually derive a more explicit formula for $\alpha_0$. Indeed, computing as in the proof of Theorem~\ref{thm:asharmhess}
from the definitions in \eqref{eq:redpot}-\eqref{eq:defY} we see that
$$\frac{\Y_{\bv}^0}{\Y^0}\,=\,-\frac{\Redpot_{\bv}^0}{2\Redpot^0}\,+\,\frac{\cap_{\bv}^0}{2\cap^0}
\,=\,\frac{\frak{b}_0}{\frak{a}_0^2}
\,+\,\frac{\cap_{\bv}^0}{2\cap^0}\,.$$
We recall that we have already established likewise in \eqref{eq:bYzzero} that
$$
\frac{\Y^0_{\bz}}{\Y^0}
\,=\,-\frac{\Redpot_{\bz}^0}{2\Redpot^0}
\,=\,\frac{\frak{b}_0 }{\frak{a}_0^2}\,.$$
so that
$$
\frac{\Y_{\bv}^0+2\Y^0_{\bz}}{\Y^0}
\,=\,
\frac{3\frak{b}_0}{\frak{a}_0^2}
\,+\,\frac{\cap_{\bv}^0}{2\cap^0}\,.
$$
Differentiating once more we get on the one hand
$$
\begin{array}{rcl}\displaystyle
\frac{\Y_{\bv\bv}^0}{\Y^0}&=&\displaystyle
\left(\frac{\Y_{\bv}^0}{\Y^0}\right)^2
\,+\,\frac12\left(\frac{\cap_\bv}{\cap}\right)_{\!\bv}^{\!0}
+\frac12\,\left(\frac{\Redpot_{\bv}^0}{\Redpot^0}\right)^2-\frac12\,\frac{\Redpot_{\bv\bv}^0}{\Redpot^0}\,,\\[5pt]
&=&\displaystyle
\frac{3\frak{b}_0^2}{\frak{a}_0^4}
\,+\,\frac{\frak{b}_0}{\frak{a}_0^2}\frac{\cap_{\bv}^0}{\cap^0}
\,+\,\frac14\left(\frac{\cap_{\bv}^0}{\cap^0}\right)^2
\,+\,\frac12\left(\frac{\cap_\bv}{\cap}\right)_{\!\bv}^{\!0}
-\frac12\,\frac{\Redpot_{\bv\bv}^0}{\Redpot^0}\,,
\end{array}
$$
with
$$
\Redpot_{\bv\bv}^0\,=\,\Potential_{\bv\bv\bv\bv}^0\,\int_{0}^{1}\int_{0}^{1}t^3 s^2\,\dif s\dif t\,=\,\frac{\Potential_{\bv\bv\bv\bv}^0}{12}
\,=\,
10\,\frac{\frak{b}_0^2}{\frak{a}_0^6}
-4\frac{\frak{c}_0}{\frak{a}_0^5}\,,
$$
so that
$$
\frac{\Y_{\bv\bv}^0}{\Y^0}
\,=\,
-\,\frac{2\frak{b}_0^2}{\frak{a}_0^4}
\,+\,2\frac{\frak{c}_0}{\frak{a}_0^3}
\,+\,\frac{\frak{b}_0}{\frak{a}_0^2}\frac{\cap_{\bv}^0}{\cap^0}
\,+\,\frac14\left(\frac{\cap_{\bv}^0}{\cap^0}\right)^2
\,+\,\frac12\left(\frac{\cap_\bv}{\cap}\right)_{\!\bv}^{\!0}\,,
$$
On the other hand, thanks to the symmetry of $\Redpot$ proved in Lemma~\ref{lem:symmetry}, we obtain
$$
\begin{array}{rclcl}\displaystyle
\frac{\Y_{\bz\bz}^0}{\Y^0}
&=&\displaystyle
\left(\frac{\Y_{\bz}^0}{\Y^0}\right)^2\,+\,\frac12\left(\frac{\Redpot_{\bz}^0}{\Redpot^0}\right)^2\,-\,\frac12\frac{\Redpot_{\bz\bz}^0}{\Redpot^0}
&=&\displaystyle
\left(\frac{\Y_{\bz}^0}{\Y^0}\right)^2\,+\,\frac12\left(\frac{\Redpot_{\bz}^0}{\Redpot^0}\right)^2\,-\,\frac12\frac{\Redpot_{\bv\bv}^0}{\Redpot^0}\,,
\\[10pt]
\displaystyle
\frac{\Y_{\bw\bz}^0}{\Y^0}
&=&\displaystyle
\frac{\Y_{\bz}^0\Y_{\bw}^0}{(\Y^0)^2}\,+\,\frac12 \frac{\Redpot_{\bz}^0\Redpot_{\bw}^0}{(\Redpot^0)^2}\,-\,\frac12\frac{\Redpot_{\bw\bz}^0}{\Redpot^0}
&=&\displaystyle
\left(\frac{\Y_{\bz}^0}{\Y^0}\right)^2\,+\,\frac12\left(\frac{\Redpot_{\bz}^0}{\Redpot^0}\right)^2\,-\,\frac12\frac{\Redpot_{\bw\bz}^0}{\Redpot^0}\,,
\end{array}
$$
with 
$$\Redpot_{\bw\bz}^{0}\,=\,\Potential_{\bv\bv\bv\bv}^0\,\int_{0}^{1}\int_{0}^{1}t^2(1-t)(1-s)\,\dif s\dif t\,=\,\frac{\Potential_{\bv\bv\bv\bv}^0}{24}\,=\,\frac12\,\Redpot_{\bv\bv}^0$$
so that
$$
\frac{\Y_{\bz\bz}^0-\Y_{\bw\bz}^0}{\Y^0}
\,=\,-\,\frac14\frac{\Redpot_{\bv\bv}^0}{\Redpot^0}
\,=\,
-\frac52\,\frac{\frak{b}_0^2}{\frak{a}_0^4}
+\frac{\frak{c}_0}{\frak{a}_0^3}\,.
$$
This eventually implies that
$$
\begin{array}{rcl}
\alpha_0
&=&\displaystyle
\frac{3\frak{b}_0^2}{\frak{a}_0^2}
\,+\,\frak{b}_0\,\frac{\cap_{\bv}^0}{2\cap^0}
-\frac52\,\frac{\frak{b}_0^2}{\frak{a}_0^2}
+\frac{\frak{c}_0}{\frak{a}_0}
\,-\,\frac{\frak{b}_0^2}{2\frak{a}_0^2}
\,+\,\frac{\frak{c}_0}{2\frak{a}_0}
\,+\,\frac{\frak{b}_0}{4}\frac{\cap_{\bv}^0}{\cap^0}
\,+\,\frac{\frak{a}_0^2}{16}\left(\frac{\cap_{\bv}^0}{\cap^0}\right)^2
\,+\,\frac{\frak{a}_0^2}{8}\left(\frac{\cap_\bv}{\cap}\right)_{\!\bv}^{\!0}
\end{array}
$$
hence
\begin{equation}\label{eq:alpha-formula}
\alpha_0
\,=\,
\frac{3\frak{c}_0}{2\frak{a}_0}
\,+\,\frac{3\frak{b}_0}{4}\frac{\cap_{\bv}^0}{\cap^0}
\,+\,\frac{\frak{a}_0^2}{8}\frac{(\cap_{\bv\bv})^{\!0}}{\cap^{\!0}}\,.
\end{equation}

\begin{remark} 
Since 
$$
\partial_\mu\Xi=\partial_\mu^2\Action=\bE\cdot\Hess\Action\bE
=\pi\,\Y^0\,\alpha_0\,+\,{\mathcal O}(\delta)
$$ 
the condition $\alpha_0\neq0$ is a condition on the nondegeneracy of the dependence of the period $\Xi$ on $\mu$, an obviously natural condition. A large body of the literature on orbital stability relies on the even more stringent assumption $\partial_\mu\Xi>0$ ; see discussion in \cite{BMR}.

In the special case when \emph{$\cap$ is constant}, we see from the computation here above that $\alpha_0$ is nonzero if and only if $\frak{c}_0$ is nonzero. This is also a rather natural condition since it means that the expansion of the amplitude of the wave $\bv_3-\bv_2$ in powers of $\sqrt{\mu-\mu_0}$ is nondegenerate - the coefficient $\frak{a}_0$ of the leading order term being forced to be nonzero by Assumption \ref{as:main}.

In particular, for a scalar equation with $\cap$ constant, we have $$\Potential (v;\lambda,\speed)= \free(\bv) \,-\,\tfrac{\speed}{2b} v^2\,-\,\lambda\,v\,,
$$
with
$$\en(\bv,\bv_x)= \frac12\cap\,\bv_x^2\,+\,\free(\bv)$$
and the fact that $\alpha_0$ is nonzero is equivalent to
$$\frac53 (\free^{'''}(\bv_0))^2\,\neq\,\left(\free''(\bv_0)-\frac{\speed}{b}\right)\,\free^{(4)}(\bv_0)\,.$$
Since $\free''(\bv_0)-\speed/b>0$ by assumption, it is automatically satisfied when $\free^{(4)}(\bv_0)<0$. It also includes the Korteweg--de Vries case where $\free(\bv)=v^3/6$. Otherwise, this may exclude some exceptional points form the range of validity of Corollary~\ref{cor:stabharm}.
\end{remark}

\begin{remark}
The condition $g_\bv^0\neq 0$ if $N=2$ is also rather natural. It expresses that the slaving of $\bu$ to $\bv$ induced by traveling wave profile equations is nondegenerate. Here this is explicitly written as
$$
\left(\lambda_2+\frac{c}{b}\,\bv_0\right)\,\tau'(\bv_0)\neq
\frac{c}{b}\,\tau(\bv_0)\,.
$$
When $\tau\equiv1$, this reduces to $\speed\neq0$ and when $\tau(\bv)=\bv$ it amounts to $\lambda_2\neq0$.

In particular, for the Euler--Korteweg system - either {\rm (EKE)} or {\rm (EKL)} -, it is automatically satisfied whenever the wave is a dynamic one, in the sense that there is some mass transfer across it. Using the notation $j$ for the mass transfer flux across the wave - as in \cite{SBG-DIE} - we see indeed that $g_\bv^0$ is equal to $j$ for {\rm (EKL)}, and equal to $-j/\rho_0^2$ for  {\rm (EKE)}.
\end{remark}

\subsection{Stability or instability in the soliton limit}

By relying on arguments similar to those used to derive Corollary \ref{cor:stabharm} from Theorems~\ref{thm:stabcrit} and~\ref{thm:asharmhess}, now we draw neat conclusions on the stability of periodic waves close to solitary waves from Theorem~\ref{thm:ashess}.

\begin{corollary}
\label{cor:stabsol} Under Assumption \ref{as:main} and those of Theorem \ref{thm:stabcrit}, we have the following properties for periodic waves of sufficiently large period.
\begin{itemize}
\item Their period $\Xi$ is monotonically increasing with $\mu$, that is, $\partial_\mu\Xi=\partial_\mu^2\Action>0$.
\item If, for the limiting solitary wave, the second derivative at fixed endstate with respect to velocity  $\Mom''$ of the Boussinesq momentum $\Mom$ is non zero, then the matrix $\Hess\Action$ is nonsingular.
\item If in addition $\Mom''>0$ then 
the negative signature of  $\Hess\Action$ is equal to $N$ and the corresponding wave is conditionally orbitally stable.
\item If on the contrary $\Mom''<0$ then 
the negative signature of  $\Hess\Action$ is equal to $N+1$ and the corresponding wave is spectrally unstable.
\end{itemize}
\end{corollary}

\proof
We have 
$$\partial_\mu^2\Action\,=\,\transp{\bE}\,\Hess\Action\,\bE\,,$$
where, as in Lemma~\ref{lem:orthVWZT}, $\bE$ is the first vector of the canonical basis of $\R^{N+2}$. By Theorem~\ref{thm:ashess} (case $N=2$), we thus see that in the solitary wave limit
$$\partial_\mu^2\Action\,=\,\Y^s_0\,\frac{2\frak{a}_s^2}{\h_s^2}\,\left(\frac{1}{\rho^2}\,+\,\frac{1}{\rho}\right)\,+\,{\mathcal O}(\ln\rho)\,,$$
which implies indeed that $\partial_\mu^2\Action$ is positive for $\rho$ small enough.

In the same spirit as in Section \ref{s:stabharm}, we may look at the quadratic form associated with $\Hess\Action$ in the convenient basis suggested by Lemma~\ref{lem:orthVWZT} and Remark~\ref{rem:bases}. Namely we introduce
$$\Pb_s:=(\bE\;\;\Sb\bV_s\;\;\Sb\bT_s\;\;\Sb\bW_s) \mbox{ for } N=2\,,\qquad\qquad
\Pb_s:=(\bE\;\;\Sb\bV_s\;\;\Sb\bW_s) \mbox{ for } N=1\,.$$
We now consider the principal minors of 
$$
\Sigma(\rho)
:=\frac{1}{\Y^s_0}\,\transp{\Pb}_s\,\Hess\Action\,\Pb_s\,.
$$
In contrast with the small amplitude case, a good choice in the basis ordering is here immaterial as it is scale separation that allows an easy computation of the signature.

Let us start with the case $N=2$. Then from \eqref{eq:ashess}  in Theorem~\ref{thm:ashess} and orthogonality relations in \eqref{eq:orthVWZT} (Lemma~\ref{lem:orthVWZT}) we have
$$
\Sigma(\rho)\,=\,
\left(
\begin{array}{c|c|c|c}\dfrac{2\frak{a}_s^2}{\h_s^2}\,\left(\dfrac{1}{\rho^2}\,+\,\dfrac{1}{\rho}\right)& 
{\mathcal O}\big(\ln\rho\big) & {\mathcal O}\big(\ln\rho\big) & {\mathcal O}\big(\ln\rho\big)
\\ [10pt]
\,+\,{\mathcal O}(\ln\rho)  & & & \\ [5pt] 
\hline 
{\mathcal O}\big(\ln\rho\big) &\Mom''& {\mathcal O}\big(1\big) & {\mathcal O}\big(1\big) \\ [5pt]
& + \,{\mathcal O}\big(\rho\ln\rho\big) & & \\ [5pt] 
\hline
{\mathcal O}\big(\ln\rho\big) & {\mathcal O}\big(1\big) & \dfrac{\frak{a}_s^2}{2}\,(\bT_s\cdot \Sb\bW_s)^2 \ln\rho & \dfrac{\frak{a}_s^2}{2}\,(\bT_s\cdot \Sb\bW_s)\\ [5pt]
& & +\, {\mathcal O}\big(1\big) &\,\times(\bW_s\cdot \Sb\bW_s)\,\ln\rho \\  [5pt]
 & & & +\, {\mathcal O}\big(1\big)\\ [5pt] 
\hline
{\mathcal O}\big(\ln\rho\big) & {\mathcal O}\big(1) &\dfrac{\frak{a}_s^2}{2}\,(\bT_s\cdot \Sb\bW_s) & \dfrac{\frak{a}_s^2}{2}\,(\bW_s\cdot \Sb\bW_s)^2\ln\rho\\ [5pt]
& & \,\times(\bW_s\cdot \Sb\bW_s)\,\ln\rho &+ \,(\bT_s\cdot \Sb\bW_s)^2\,\ln\rho \\  [5pt]
& &  +\, {\mathcal O}\big(1\big)& \,+\,{\mathcal O}\big(1\big) \\
\end{array}\right)\,.$$
The first minor is $\partial_\mu^2\Action/\Y^0$ already expanded above. The last one satisfies
$$\det\Sigma(\rho)\,=\,
\frac{\frak{a}_s^4}{\h_s^2}\,(\bT_s\cdot \Sb\bW_s)^4\,\Mom''
\frac{(\ln\rho)^2}{\rho^2}
\,+\,{\mathcal O}\Big(\frac{\ln\rho}{\rho^2}\Big)\,.
$$
In particular, since $\bT_s\cdot \Sb\bW_s=b/\sqrt{\tau(\bv_s)}$ is nonzero, $\det\Hess\Action$ is therefore also nonzero for $\rho$ small enough provided that $\Mom''$ is nonzero. Other principal minors are as follows. The $2\times 2$ one reads
$$\Delta_2\,=\,
\frac{2\frak{a}_s^2}{\h_s^2}\,\Mom''\,\frac{1}{\rho^2}
\,+\,{\mathcal O}\Big(\frac{\ln\rho}{\rho}\Big)
$$
and the $3\times 3$ one is
$$\Delta_3\,=\,
\frac{\frak{a}_s^4}{\h_s^2}\,(\bT_s\cdot \Sb\bW_s)^2\,\Mom''\,
\frac{\ln\rho}{\rho^2}\,+\,{\mathcal O}\Big(\frac{1}{\rho^2}\Big)\,.$$
Therefore, when $\Mom''\neq0$, by Sylvester's rule, for $\rho$ small enough the negative signature of $\Hess\Action$ equals the number of sign changes in $(+,+,\mbox{sign}\Mom'',-\mbox{sign}\Mom'',\mbox{sign}\Mom'')$, that is, two if $\Mom''$ is positive and three if $\Mom''$ is negative. This completes the proof of Corollary~\ref{cor:stabsol} in the case $N=2$. 

Let us now look at the case $N=1$. We first observe that 
$$\Sigma(\rho)\,=\,
\left(
\begin{array}{c|c|c}\dfrac{2\frak{a}_s^2}{\h_s^2}\,\left(\dfrac{1}{\rho^2}\,+\,\dfrac{1}{\rho}\right)&  {\mathcal O}\big(\ln\rho\big) & {\mathcal O}\big(\ln\rho\big)
\\ [10pt]
\,+\,{\mathcal O}(\ln\rho)  &  & \\ [5pt] \hline
{\mathcal O}\big(\ln\rho\big) & 
\Mom'' & {\mathcal O}\big(1\big) \,\\ [5pt] 
 &  +\, {\mathcal O}\big(\rho\ln\rho\big) & \\ [5pt] \hline
{\mathcal O}\big(\ln\rho\big)  & {\mathcal O}\big(1) & 
\dfrac{\frak{a}_s^2}{2}\,(\bW_s\cdot \Sb\bW_s)^2\ln\rho\\ [5pt]
& &  \,+\,{\mathcal O}\big(1\big) \\
\end{array}\right)\,.$$
From this we find that
$$\det\Sigma(\rho)\,=\,\frac{\frak{a}_s^4}{\h_s^2}\,(\bW_s\cdot \Sb\bW_s)^2\,\Mom''\,
\frac{\ln\rho}{\rho^2}\,+\,{\mathcal O}\Big(\frac{1}{\rho^2}\Big)\,.$$
Note that $\bW_s\cdot \Sb\bW_s=b$ is nonzero so that $\Mom''\neq0$ implies that $\det\Sigma(\rho)\neq0$ for sufficiently small $\rho$. As in the former case, the expansion of the first principal minor of $\Sigma(\rho)$ is already known and the one of its $2\times 2$ principal minor is
$$\Delta_2\,=\,
\frac{2\frak{a}_s^2}{\h_s^2}\,\Mom''\,\frac{1}{\rho^2}
\,+\,{\mathcal O}\Big(\frac{\ln\rho}{\rho}\Big)
\,.$$
Therefore, when $\Mom''\neq0$, by Sylvester's rule, for $\rho$ small enough the negative signature of $\Hess\Action$ equals the number of sign changes in $(+,+,\mbox{sign}\Mom'',-\mbox{sign}\Mom'')$, that is, one if $\Mom''$ is positive and two if $\Mom''$ is negative.
\endproof

\paragraph{Acknowledgement.} This work has been partly supported by the ANR project BoND (ANR-13-BS01-0009-01), and by the LABEX MILYON (ANR-10-LABX-0070) of Université de Lyon, within the program `\emph{Investissements d'Avenir}' (ANR-11-IDEX-0007) operated by the French National Research Agency (ANR). The doctoral scholarship of the second author was directly supported by ANR-13-BS01-0009-01. 

\appendix
\section*{Appendix}
\addtocontents{toc}{\protect\contentsline{section}{\appendixname}{}{}}
\renewcommand\thesubsection{\Alph{subsection}}
\setcounter{lemma}{0}
\renewcommand\thelemma{\Alph{subsection}.\roman{lemma}}
\setcounter{proposition}{0}
\renewcommand\theproposition{\Alph{subsection}.\roman{proposition}}
\setcounter{remark}{0}
\renewcommand\theremark{\Alph{subsection}.\roman{remark}}

\subsection{Spectra at the harmonic limit}\label{app:harmonic}

To shed some light on Theorem~\ref{thm:stabsmall}, we provide here outcomes of some spectral computations for constant steady states. In what follows we use the same notation for constant-valued functions as for their actual values.
 
By the special form of the Hamiltonian $\Ham$ and the matrix  $\bB$ (Assumption \ref{as:ham}), the Lagrangian $\Lag$ reads
$$\Lag(\bU,\bU_x;\speed,\blambda)\,=\,\en(\bv,\bv_x)\,+\,\frac{\tau(\bv)}{2} \,\bu^2\,+\,\frac{\speed}{2b}\,\bv\,\bu\,\,$$
in the case $N=2$, and merely
$$\Lag(\bv,\bv_x;\speed,\blambda)\,=\,\en(\bv,\bv_x)\,+\,\frac{\speed}{2b}\,\bv^2\,\,$$
in the case $N=1$, so that its second variational derivative - or Hessian - is the differential operator
$$\hess \Lag[\bU]= \left(\begin{array}{cc}\displaystyle \hess \en[\vol]  + {\tau''(\bv)} \,\bu^2 /2 & \displaystyle {\speed}/{b} \,+\,\tau'(\bv)\,\bu\\ [5pt]
\displaystyle {\speed}/{b} \,+\,\tau'(\bv)\,\bu & \tau(\bv) \end{array}\right)\,\quad \mbox{when } N=2\,,$$
$$\hess \Lag[\bv]=\hess \en[\vol]  + {\speed}/{b}\,\quad \mbox{when } N=1\,.$$

For a traveling wave solution to $\ubU$ of \eqref{eq:absHamb} to be - neutrally - spectrally stable with respect to localized perturbations we need that the spectrum of the operator $\partial_x \,\bB\,\hess \Lag[\ubU]$ in $L^2(\R)$ lie in the left half-plane. When $\ubU$ is actually a constant $\bU_*$ this can be checked by Fourier transform. With
$\en(\bv,\bv_x)\,=\,\frac12 \cap(\bv)\bv_x^2\,+\,\free(\bv)$ we have
$$\hess \en[\vol_*]\,=\,\free''(\bv_*)\,-\,\cap(\bv_*)\,\partial_x^2\,.$$
So by Fourier transform applied to the eigenvalue problem
$$\partial_x \,(\hess \en[\vol_*] \bv \,+\,\speed \,\bv/b)\,=\,z\,\bv$$
we see that for $z$ to be in the $L^2$ spectrum of $\partial_x \;\hess \Lag[\bv_*]$ there must exist $\xi\in \R$ such that
$$z=i\,\xi\,(\free''(\bv_*)\,+\,\speed/b \,+\,\cap(\bv_*)\,\xi^2)\,,$$
which obviously forces $z$ to be purely imaginary $z$. This implies the neutral stability of $\bv_*$ in the case $N=1$. In the case $N=2$
we find in the same way that $z$ is in the $L^2$ spectrum of $\partial_x \;\hess \Lag[\bU_*]$
 if and only if $z=i\xi\,b\,\sigma$ for some $\xi\in \R$ and $\sigma$ an eigenvalue of the matrix
$$\bA_*(\xi):=
\left(\begin{array}{cc}  {\speed}/b \,+\,\tau'(\bv_*)\,\bu &  \tau(\bv_*) \\ [5pt]
 \displaystyle \free''(\bv_*)\,+\,\cap(\bv_*)\,\xi^2 + {\tau''(\bv_*)} \,\bu_*^2 /2 &  {\speed}/b \,+\,\tau'(\bv_*)\,\bu_* \end{array}\right)\,.
$$
Since $ \tau(\bv_*)$ is positive, a necessary and sufficient condition for $\sigma$ to be real and thus for $z$ to be purely imaginary is 
$$\free''(\bv_*)\,+\,\cap(\bv_*)\,\xi^2 + {\tau''(\bv_*)} \,\bu_*^2 /2\geqslant 0\,.$$
Knowing that $\cap(\bv_*)$ is positive we thus find that neutral stability of the constant steady state $\bU_*=(\bv_*,\bu_*)$ is equivalent to 
$$\free''(\bv_*)\,+ \,{\tau''(\bv_*)} \,\bu_*^2 /2\geqslant 0\,.$$
This is precisely the condition for hyperbolicity at $\bU_*$ of the underlying dispersionless system - when $\cap$ is set equal to zero
(which yields the usual barotropic Euler equations for \EKE, and the so-called $p$-system for \EKL).

Stability with respect to periodic perturbations is of a slightly different nature. Let us recall from \cite{BMR} that orbital stability with respect to periodic perturbations can be inferred from \emph{variational stability under constraints}. For a constant steady state this amounts to requiring that the spectrum of $\hess \Lag[\bU_*]$ in the space of locally square integrable, $\Xi$-periodic zero-mean functions be nonnegative - plus some additional conditions about the possible kernel. This spectrum is discrete and can be found by decomposition in Fourier series.

In the scalar case $N=1$ it is made of the numbers
$$\free''(\bv_*)\,+\,\speed/b \,+\,\cap(\bv_*)\,(2\pi\,\ell /\Xi)^2$$
for $\ell\in\Z\setminus\{0\}$.
These are all positive as soon as $\free''(\bv_*)\,+\,\speed/b\geqslant 0$. This is however not the interesting case if we think of 
$\bv_*$ as the limiting state of a small amplitude periodic wave. Indeed, such a state is characterized in terms of the potential
$$\Potential(\bv;\speed,\lambda)=\,-\,\free(\bv)\,-\, \frac{\speed}{2 b}\, \bv^2\,-\,\lambda\,\bv$$
by
$$0<\partial_{\bv}^2\Potential(\bv_*;\speed,\lambda) =\,-\,\free''(\bv_*)\,-\,\speed/b \,.$$
In this case there is a condition for $\bv_*$ to be stable with respect to $\Xi$-periodic zero-mean perturbations. Indeed, we must have
$$-\,\partial_{\bv}^2\Potential(\bv_*;\speed,\lambda) \,+\,\cap(\bv_*)\,(2\pi\,\ell /\Xi)^2\,\geqslant 0 $$
for all $\ell\in\Z\setminus\{0\}$,
which is equivalent to
$$\Xi 
\,\leqslant \,2\pi\sqrt{\frac{\cap(\bv_{*})}{\partial_{\bv}^2\Potential(\bv_*;\speed,\lambda)}}
=:\Xi_{*}\,,$$
where $\Xi_{*}$ is precisely the harmonic period. With the strict inequality we only have positive spectrum and thus stability.
In the limiting case $\Xi=\Xi_{*}$ we find a two-dimensional kernel for $\hess \Lag[\vol_*]$, corresponding to $\ell=1$ or $-1$ and thus only to harmonic waves, as expected. 

In the case $N=2$, the potential is 
$$\Potential(\bv;\speed,\blambda)=\,-\,\free(\bv)\,+\,\tfrac12 \tau(\bv)\,g(\bv;\speed,\lambda_2)^2\,-\,\lambda_1\,\bv\,,$$
with 
$$\tau(\bv)\,g(\bv;\speed,\lambda)\,=\,-\,c\,\bv/b\,+\,\lambda\,.$$
By differentiating twice this identity we readily see that
$$\tau''(\bv)\,g(\bv;\speed,\lambda)\,+\,2\,\tau'(\bv)\,\partial_\bv g(\bv;\speed,\lambda)\,+\,\tau(\bv)\,\partial_\bv^2 g(\bv;\speed,\lambda)\,=\,0\,,$$
from which we find that
$$\free''(\bv)\,+\,\tfrac12 {\tau''(\bv)} \,g(\bv;\speed,\lambda_2)^2\,=\,-\,\partial_\bv^2\Potential(\bv;\speed,\blambda)\,+\,
\frac{1}{\tau(\bv)}\,\left(\tau'(\bv)\,g(\bv;\speed,\lambda_2)+\frac{\speed}{b}\right)^2\,.$$
Therefore, we may observe that the spectrum of the matrices $\bA_*(2\pi \ell/\Xi)$ - as defined above with 
$\bu_*=g(\bv_*;\speed,\lambda_2)$
- is real for all $\ell\in\Z\setminus\{0\}$ if and only if, either $\partial_\bv^2\Potential(v_*;\speed,\blambda)\leqslant 0$ or 
$$\partial_\bv^2\Potential(v_*;\speed,\blambda)\,=\,
\frac{4\pi^2 \cap(\bv_{*})}{\Xi_*^2}
>0 \quad \mbox{and}\quad 
\left(\frac{\Xi_{*}}{\Xi}\right)^2
\geqslant 1-\frac{1}{\tau(v_*)\,\partial_\bv^2\Potential(v_*;\speed,\blambda)}\left(\tau'(v_*)u_*+\frac{\speed}{b}\right)^2\,.
$$
The latter holds in particular for $\Xi=\Xi_*$, hence the neutral stability of $\bU_*$ with respect to harmonic perturbations, even when the underlying dispersionless system is not hyperbolic. 

As regards variational stability under $\Xi$-periodic zero-mean perturbations, it requires that the symmetric matrices $\Sigma_*(2\pi\ell/\Xi)$ defined by
$$\Sigma_*(\xi):=
\left(\begin{array}{cc}  \displaystyle \free''(\bv_*)\,+\,\cap(\bv_*)\,\xi^2 + {\tau''(\bv_*)} \,\bu_*^2 /2 &  {\speed}/b \,+\,\tau'(\bv_*)\,\bu_*  \\ [5pt]
 {\speed}/b \,+\,\tau'(\bv_*)\,\bu_*  &  \tau(\bv_*) \end{array}\right)\,.
$$
be nonnegative for all $\ell\in\Z\setminus\{0\}$.
Rewriting 
$$\Sigma_*(\xi)=
\left(\begin{array}{cc}  \displaystyle -\,\partial_\bv^2\Potential(v_*;\speed,\blambda)\,+\,\cap(\bv_*)\,\xi^2\,+\,({\speed}/b \,+\,\tau'(\bv_*)\,\bu_*)^2/\tau(\bv_*) &  {\speed}/b \,+\,\tau'(\bv_*)\,\bu_*  \\ [5pt]
 {\speed}/b \,+\,\tau'(\bv_*)\,\bu_*  &  \tau(\bv_*) \end{array}\right)
$$
we see that for all $\bU=\transp{(\bv , \bu)}$
$$\bU\cdot \Sigma_*(\xi) \bU\,=\,(-\,\partial_\bv^2\Potential(v_*;\speed,\blambda)\,+\cap(\bv_*)\,\xi^2)\,\bv^2\,+\,
\frac{1}{\tau(\bv_*)}\,\left(\tau(\bv_*)\,\bu\,+\,({\speed}/b \,+\,\tau'(\bv_*)\,\bu_*)\,\bv\right)^2\,.
$$
Therefore, the matrices $\Sigma_*(2\pi\ell/\Xi)$ are nonnegative for all $\ell\in\Z\setminus\{0\}$ if and only if $\partial_\bv^2\Potential(v_*;\speed,\blambda)\leqslant \cap(\bv_*)\,2\pi/\Xi$, which is equivalent to $\Xi\leqslant \Xi_*$. As in the scalar case, we find a two-dimensional kernel corresponding to harmonic waves in the limiting case $\Xi=\Xi_{*}$.

\newpage

\subsection{Symmetry of the reduced potential}\label{app:alg}

\begin{lemma}\label{lem:symmetry}
Let $\ZZ$ be a smooth function of one variable and
$$\Redpot(\bv,\bw,\bz):=\int_{0}^{1}\!\!\!\int_{0}^{1}t\ZZ''(\bw+t(\bz-\bw)+ts(\bv-\bz))\,\dif s\dif t\,.$$
Then $\Redpot$ is a symmetric function of $(\bv,\bw,\bz)$.
\end{lemma}

\begin{proof}
By continuity of $\Redpot$, it is sufficient to establish symmetry relations on the dense and symmetric set of $(\bv,\bw,\bz)$s such that $(\bv-\bw)(\bw-\bz)(\bz-\bv)\neq 0 $. Then, for such a $(\bv,\bw,\bz)$, symmetry follows from 
$$
\begin{array}{rcl}
\Redpot(\bv,\bw,\bz)&=&\displaystyle
\int_{0}^{1}\int_{0}^{1}t\ZZ''(\bw+t(\bz-\bw)+ts(\bv-\bz))\,\dif s\dif t\\[10pt]
&=&\displaystyle
\int_{0}^{1}\frac{\ZZ'(\bw+t(\bv-\bw))-\ZZ'(\bw+t(\bz-\bw))}{\bv-\bz}\,\dif t\\[10pt]
&=&\displaystyle
\frac{\ZZ(\bv)-\ZZ(\bw)}{(\bv-\bz)(\bv-\bw))}
-\frac{\ZZ(\bz)-\ZZ(\bw)}{(\bv-\bz)(\bz-\bw)}\\[10pt]
&=&\displaystyle
\frac{\ZZ(\bv)}{(\bv-\bw)(\bv-\bz)}\,+\,\frac{\ZZ(\bw)}{(\bw-\bz)(\bw-\bv)}\,+\,\frac{\ZZ(\bz)}{(\bz-\bv)(\bz-\bw)}
\end{array}
$$
obtained by integrating twice.
\end{proof}

\subsection{Elementary asymptotics}\label{app:elas}
\begin{lemma}\label{lem:as}
Assume that $W$ is a smooth function such that 
$$W(0)=0\,,\qquad\qquad
W'(0)=0\,,\qquad\mbox{and}\qquad
W''(0)>0\,.$$
Then for small $\varepsilon>0$ there are two zeros $z_\pm$ of $W-\varepsilon$ in the neighborhood of $0$, which have the following expansions
$$z_\pm(\varepsilon)=\pm \sqrt{\frac{2\varepsilon}{\alpha}}\,-\,\frac{\beta\varepsilon}{3\alpha^2}\,\pm\,\eta\,\varepsilon^{3/2}\,+\,{\mathcal O}(\varepsilon^2)\,,$$
where $$\alpha:=W''(0)\,,\quad \beta:=W'''(0)\,,\quad \gamma:=W''''(0)\,,\quad \eta:= \frac{1}{6\alpha^3\sqrt{2\alpha}}\,(\tfrac53 \beta^2-\alpha\gamma)\,.$$
Furthermore, their derivatives expand as
$$z_\pm'(\varepsilon)=\pm \frac{1}{\sqrt{2\alpha\varepsilon}}\,-\,\frac{\beta}{3\alpha^2}\,\pm\,\frac32 \eta\,\sqrt{\varepsilon}\,+\,{\mathcal O}(\varepsilon)\,.$$
\end{lemma}
\begin{proof}
By the Taylor formula
$$
W(z)\,=\,z^2\int_{0}^{1} W''(t\,z)\,(1-t)\dif t
$$
we know that, in a neighborhood of $0$, $W$ takes nonnegative values and the function $f$ defined by
$$w(z):=z \sqrt{W(z)/z^2}$$
is smooth. Then from
$$W(z)=\frac{\alpha}{2} z^2+ \frac{\beta}{6} z^3 + \frac{\gamma}{24} z^4+ {\mathcal O}(z^5)\,,$$
the Taylor expansion of $f$ is seen to be
$$w(z)\,=\,
\sqrt{\frac{\alpha}{2}}\,\left(z + \frac{\beta}{6\alpha}z^2 + \frac{1}{24\alpha^2} \left(\gamma\alpha- \frac{\beta^2}{3}\right) z^3 + {\mathcal O}(z^4)\right) .$$
Now, since $W(z)=\varepsilon>0$ is equivalent to $w(z)=\pm \sqrt{\varepsilon}$ and $w'(0) = \sqrt{\alpha/2}\neq 0$,
the implicit function theorem implies that $W(z)=\varepsilon$ is equivalent  for $z$ close to zero  to
$z=\varphi(\pm\sqrt{\varepsilon})$ for a smooth function $\varphi$ of $x:=\pm\sqrt{\varepsilon}$ that vanishes at zero. Furthermore, we find its asymptotic expansion
$$\varphi(x)=a x + b x^2 + c x^3 + {\mathcal O}(x^4)$$ by merely equating terms of same order in
$$w(\varphi(x))= x\,.$$
This gives
\begin{itemize}
\item at order one, $a=\sqrt{2/\alpha}$,
\item at order two, $b+\frac{\beta}{6\alpha} a^2=0$, hence $b=-\frac{\beta}{3\alpha^2}$,
\item at order three $c+ \frac{\beta}{3\alpha} ab + \frac{1}{24\alpha^2} (\gamma\alpha-\frac{\beta^2}{3}) a^3=0$, hence 
$c=\frac{1}{6\alpha^3\sqrt{2\alpha}}\,(\frac53 \beta^2-\alpha\gamma)$.
\end{itemize}
This yields the claimed expansion for $z_\pm(\varepsilon):=\varphi(\pm\sqrt{\varepsilon})$.

The expansion of $z_\pm'$ turns out to correspond to term-by-term differentiation of the expansion of $z_\pm$. This 
may be justified here by first differentiating $z_\pm(\varepsilon)=\varphi(\pm\sqrt{\varepsilon})$ to receive 
$$
z_\pm'(\varepsilon)=\frac{\mp1}{2\sqrt{\varepsilon}}\,\varphi'(\pm\sqrt{\varepsilon})
$$
and then expanding $\varphi'$.
\end{proof}

The next proposition is a little bit more sophisticated in that it provides asymptotic series  that are not power series for integrals depending on parameters. These asymptotic series are of the form 
$$\sum_{k\geqslant n} (a_k \rho^k\,\ln\rho \,+\,b_k\, \rho^k)\qquad $$
for some rational 
integer $n$, when $\rho \searrow 0$.

For concision and precision, we adopt the convention that 
$$
G(\rho)\,\sim\,\sum_{k\geqslant n} g_k(\rho)
$$
denotes the fact that $f$ admits $\sum_{k\geqslant n} g_k$ as asymptotic series in the limit $\rho\to0$, that is, for any $J\geq n$
$$
G(\rho)\,=\,\sum_{n\leqslant k\leqslant J} g_k(\rho)
\,+\,{\mathcal O}(g_{J+1}(\rho))
$$
as $\rho\to0$. In the foregoing, $(g_k)_{k\geq n}$ is a sequence of functions indexed by integers.
It will be useful to have the following result at hand, the proof of which is straightforward, using that, 
when $k\geq0$, the antiderivative of $x\mapsto x^k \ln x$ that vanishes at $x=0$ is $x\mapsto x^{k+1} (\ln x - 1/(k+1))/(k+1)$.

\begin{lemma}\label{lem:derasexp}
If a differentiable function $f:(0,+\infty)\to \R$ and its derivative admit asymptotic expansions
$$\begin{array}{l}\displaystyle
G(\rho)\,\sim\,\sum_{k\geqslant 0} (a_k \rho^k\,\ln\rho \,+\,b_k\, \rho^k)\\[10pt]\displaystyle
G'(\rho)\,\sim\,\frac{d_0}{\rho}\,+\,\sum_{k\geqslant 0} (c_{k+1} \rho^k\,\ln\rho \,+\,d_{k+1}\, \rho^k)
\end{array}
$$
when $\rho \searrow 0$, then the expansion of $f'$ is obtained by formal differentiation of the expansion of $f$, and more precisely, we have
$$a_0=d_0\,,\qquad\qquad\qquad a_k=\frac{c_k}{k}\,,\qquad b_k=\frac{d_k}{k}-\frac{c_k}{k^2}\,,\quad k\geqslant 1\,,$$
or equivalently
$$d_0=a_0\,,\qquad\qquad\qquad c_k=k a_k\,,\qquad d_k= k b_k + a_k\,,\quad k\geqslant 1\,.$$
\end{lemma}

Note that if $G$ is such that 
$$
G'(\rho)\,\sim\,\sum_{k\geqslant 0} (c_{k+1} \rho^k\,\ln\rho \,+\,d_{k+1}\, \rho^k)
$$
when $\rho\to 0$ then by a direct integration one concludes that as $\rho\to0$
$$
G(\rho)\,\sim\,b_0\,+\,\sum_{k\geqslant 1} (a_k \rho^k\,\ln\rho \,+\,b_k\, \rho^k)
$$
with $b_0=\lim_{\rho\to0}G(\rho)$ and the former lemma may then be applied.

\begin{proposition}\label{prop:asol}
Let $f$, $g$, and $h$ be smooth integrable functions on $[0,1)$, and consider the following integrals depending on a parameter $\rho>0$,
$$F(\rho):=\int_{0}^{1} f(x) \,\sqrt{x(x+\rho)}\,\dif x,\qquad
G(\rho):=\int_{0}^{1} \frac{g(x)\,\dif x}{\sqrt{x(x+\rho)}},\qquad
H(\rho):=\int_{0}^{1} \frac{h(x)\,\dif x}{\sqrt{x(x+\rho)^3}}\,.$$
They have asymptotic expansions of the form
$$F(\rho)\sim B_0+B_1 \rho+\sum_{k\geqslant 2} (A_k \rho^k\,\ln\rho \,+\,B_k\, \rho^k)\,,\qquad
G(\rho)\sim\sum_{k\geqslant 0} (a_k \rho^k\,\ln\rho \,+\,b_k\, \rho^k)\,,$$
and
$$H(\rho)\sim \frac{d_0}{\rho}\,+\,\sum_{k\geqslant 0} (c_{k+1} \rho^k\,\ln\rho \,+\,d_{k+1}\, \rho^k)$$
when $\rho \searrow 0$, in which the first coefficients
for $F$ and $G$ 
are given respectively by
$$B_0=\int_{0}^{1} xf(x)\,\dif x\,,\qquad\quad
B_1=\frac{1}{2}\,\int_{0}^{1} f(x)\,\dif x\,,\qquad\quad 
A_2=\frac18 f(0)\,,$$
and by
$$
\begin{array}{l}\displaystyle
a_0=-\,g(0)\,,\qquad\quad b_0=
\int_{0}^{1}\!\int_{0}^{1}\ g'(sx)\,\dif s\,\dif x\,+\,g(0)\,\ln 4\,,\qquad\quad a_1=\,g'(0)/2\,,\\[5pt]\displaystyle
b_1=\frac12\left(g(0)\,+\,(1-\ln4)\,g'(0)
\,-\, 
\int_{0}^{1}\!\int_{0}^{1}(1-s)\,g''(sx)\,\dif s\,\dif x\right)\,,\quad a_2=-\frac{3}{16}g''(0)\,,\\[5pt]\displaystyle
b_2=\frac{1}{16}\,\left(-3g(0)-6g'(0)+\tfrac12 g''(0)(-7+6\ln4)\,+\,6\,\int_{0}^{1}\!\int_{0}^{1}\frac{(1-s)^2}{2}\,g'''(sx)\,\dif s\,\dif x\right)\,.
\end{array}$$
\end{proposition}

For practical use, note that all double integrals appearing above may be reformulated as simple integrals by using Taylor formulas. For instance, 
$$
\int_{0}^{1}\!\int_{0}^{1}\ g'(sx)\,\dif s\,\dif x
\,=\,\int_{0}^{1}\frac{g(x)-g(0)}{x}\,\dif x\,.
$$

\begin{remark}\label{rem:asol}
Note that the assumptions on $f$, $g$ and $h$ are tailored to include functions given as $g(x)=\check{g}(x)/\sqrt{1-x}$ with $\check{g}$ a smooth function on $[0,1]$.
\end{remark}

\begin{proof} We start with the central case of $G$.
Let us define by induction $g_0(x):=g(x)$ and $g_{j+1}(x)=(g_j(x)-g_j(0))/x$ for all $j\geqslant 0$ and $x\neq 0$. All these functions have smooth continuations up to $x=0$ if $g$ is ${\mathcal C}^\infty$ at $x=0$. 
Those continuations are smooth and integrable on $[0,1)$. 
By induction we obtain, for any $J\geq0$,
\begin{equation}\label{eq:expG}
G(\rho)=\sum_{j=0}^{J} g_j(0)\,\int_{0}^{1} \frac{x^j}{\sqrt{x(x+\rho)}}\,\dif x\,+\,\int_{0}^{1} g_{J+1}(x)\,\frac{x^{J+1}}{\sqrt{x(x+\rho)}}\,\dif x\,,
\end{equation}
with $g_j(0):=g^{(j)}(0)/j!$.
In \eqref{eq:expG} the factors of $g_j(0)$ may be computed explicitly. Indeed, after a series of standard changes of variables, which eventually amounts to introducing the variable 
$$
s\,=\,\exp\left(\mbox{argcosh}\left(\frac{x}{\rho/2}+1\right)\right)
\,=\,\frac{2x+\rho+2\sqrt{x(x+\rho)}}{\rho}\,,
$$
we find that
$$
\int_{0}^{1} \frac{x^j}{\sqrt{x(x+\rho)}}\,\dif x\,=\,
\left(\frac{\rho}{4}\right)^j\int_{1}^{(2+\rho+2\sqrt{1+\rho})/\rho} (s^{1/2}-s^{-1/2})^{2j}\,s^{-1}\,\dif s
$$
thus
$$\int_{0}^{1} \frac{x^j}{\sqrt{x(x+\rho)}}\,\dif x\begin{array}[t]{l}\,=\,\begin{array}[t]{l}\displaystyle
{2j \choose j}\,(-1)^j\,(\rho/4)^j\,
\left[\ln\left((2x+\rho+2\sqrt{x(x+\rho)})/\rho\right)\right]_{x=0}^{x=1} \\ [\jot]
+ \displaystyle\sum_{\underset{k\neq j}{0\leqslant k \leqslant 2j}} {2j \choose k}\,
\frac{(-\rho)^{2j-k}}{(k-j)\,4^j}\,
\left[(2x+\rho+2\sqrt{x(x+\rho)})^{(k-j)}\right]_{x=0}^{x=1}
\end{array}\\
\,=\,\begin{array}[t]{l}\displaystyle
{2j \choose j}\,(-1)^j\,(\rho/4)^j\,\left(\ln(2+\rho+2\sqrt{1+\rho})\,-\,\ln\rho\right) \\ [\jot]
+ \displaystyle\sum_{\underset{k\neq j}{0\leqslant k \leqslant 2j}} {2j \choose k}\,
\frac{(-\rho)^{2j-k}}{(k-j)\,4^j}\,
\left((2+\rho+2\sqrt{1+\rho})^{(k-j)}\,-\,\rho^{(k-j)}\right).
\end{array}
\end{array}
$$
Since both the function $\rho\mapsto\ln(2+\rho+2\sqrt{1+\rho})$ and, for $n$ any rational integer, the function $\rho\mapsto(2+\rho+2\sqrt{1+\rho})^{n}$  have power series expansions as $\rho\to0$, this shows that
\begin{equation}\label{eq:expxj}
\int_{0}^{1} \frac{x^j}{\sqrt{x(x+\rho)}}\,=\,*\,\rho^j\,\ln\rho \,+\,\sum_{k=0}^{J} *\,\rho^k \,+\,{\mathcal O}(\rho^{J+1})\,,
\end{equation}
where we have denoted for simplicity by $*$ the coefficients that can be derived from the explicit formula above and the expansion of the aforementioned functions. 

We must also deal with the last integral in \eqref{eq:expG}, which is of the form
$$I_j(\rho)=\int_{0}^{1} \varphi(x)\,\frac{x^{j}}{\sqrt{x(x+\rho)}}\,\dif x\,,$$
with $\varphi=g_{J+1}$ smooth and integrable on $[0,1)$ and $j=J+1\geqslant 1$. For any $k\leqslant j-1$ the $k^{\mbox{\tiny th}}$ derivative of $I_j$ may be obtained by differentiating under the $\int$ sign. This yields for $k$ as above and any $\rho\geq0$
$$I_j^{(k)}(\rho)=\left(\tfrac{-1}{2}\right)\left(\tfrac{-3}{2}\right)\cdots\left(\tfrac{-2k+1}{2}\right)\,\int_{0}^{1} \varphi(x)\,\frac{x^{j-1/2}}{(x+\rho)^{k+1/2}}\,\dif x\,,$$
hence in particular
$$I_j^{(k)}(0)=
\frac{(-1)^k\times(2k)!}{4^k\times k!}
\,\int_{0}^{1} \varphi(x)\,x^{j-1-k}\,\dif x\,.$$
At any $\rho>0$ we can differentiate $I_j^{(j-1)}$ once more, and we receive
$$I_j^{(j)}(\rho)=\left(\tfrac{-1}{2}\right)\left(\tfrac{-3}{2}\right)\cdots\left(\tfrac{-2j+1}{2}\right)\,\int_{0}^{1} \varphi(x)\,\left(\frac{x}{x+\rho}\right)^{j}\,\frac{\dif x}{\sqrt{x(x+\rho)}}\,.$$
Provided that $\varphi$ is bounded near zero and integrable on $(0,1)$, this last integral is an 
$${\mathcal O}\left(\int_{0}^{1} \frac{\dif x}{\sqrt{x(x+\rho)}}\right)\,,$$
which is itself an ${\mathcal O}(\ln \rho)$. Therefore, $I_j^{(j)}$ is integrable on $[0,1]$. This enables us to write by Taylor's formula that
$$I_j(\rho)\,\begin{array}[t]{l}
\displaystyle =\,\sum_{k=0}^{j-1} \frac{I_j^{(k)}(0)}{k!}\,\rho^k\,+\,\int_{0}^{\rho} \frac{(\rho-r)^{j-1}}{(j-1)!}\,I_j^{(j)}(r)\,\,\dif r\\ [\jot]
\displaystyle =\,\sum_{k=0}^{j-1} \frac{I_j^{(k)}(0)}{k!}\,\rho^k\,+\,{\mathcal O}(\rho^j \,\ln \rho)\,.\end{array}$$
Thus we have
$$\int_{0}^{1} g_{J+1}(x)\,\frac{x^{J+1}}{\sqrt{x(x+\rho)}}\,\dif x\,=\,
\sum_{k=0}^{J}\frac{(-1)^k\,\rho^k\times(2k)!}{4^k\times (k!)^2}
\,\int_{0}^{1} g_{J+1}(x)\,x^{J-k}\,\dif x
\,+\,
{\mathcal O}(\rho^{J+1}\,\ln\rho)\,,$$
where the estimate depends only on a bound of $g_{J+1}$ near zero and on $\|g_{J+1}\|_{L^1}$.

In combination with \eqref{eq:expG} and \eqref{eq:expxj}, this shows altogether the existence of coefficients $a_j$, $b_j$ such that
\begin{equation}\label{eq:expGcut}
G(\rho)=\sum_{j=0}^{J}  (a_j \rho^j\,\ln\rho \,+\,b_j\, \rho^j)\,+\,{\mathcal O}(\rho^{J+1}\,\ln\rho)\,.
\end{equation}
The actual computation of $a_j$, $b_j$ is technical but not difficult. 
In order to compute just  $a_j$, $b_j$  for $j=0,1,2$, we make explicit the foregoing process with $J=2$. On the one hand this yields in the limit $\rho\to0$
$$
\begin{array}{rcl}\displaystyle
\int_{0}^{1} \frac{\dif x}{\sqrt{x(x+\rho)}}&=&\displaystyle
-\ln\rho \,+\,\ln 4\,+\,\dfrac{\rho}{2}\,-\,\dfrac{3\rho^2}{16}\,+\,{\mathcal O}(\rho^3)\\[5pt]\displaystyle
\int_{0}^{1} \frac{x\,\dif x}{\sqrt{x(x+\rho)}}&=&\displaystyle
\frac\rho2\ln\rho \,-\,\frac\rho2\ln 4\,+\,1\,+\,\dfrac{\rho}{2}\,-\,\dfrac{3\rho^2}{8}\,+\,{\mathcal O}(\rho^3)\\[5pt]\displaystyle
\int_{0}^{1} \frac{x^2\,\dif x}{\sqrt{x(x+\rho)}}&=&\displaystyle
-\dfrac{3\rho^2}{8}\ln\rho \,+\,\dfrac{3\rho^2}{8}\ln 4\,+\,\frac12\,-\,\dfrac{\rho}{2}\,-\,\dfrac{7\rho^2}{16}\,+\,{\mathcal O}(\rho^3)
\end{array}
$$
since
$$
\begin{array}{rcl}\displaystyle
\ln(2+\rho+2\sqrt{1+\rho})&=&\displaystyle
\ln 4\,+\,\dfrac{\rho}{2}\,-\,\dfrac{3\rho^2}{16}\,+\,{\mathcal O}(\rho^3)\\[5pt]\displaystyle
2+\rho+2\sqrt{1+\rho}&=&\displaystyle
4\,+\,2\rho-\frac{\rho^2}{4}\,+\,{\mathcal O}(\rho^3)\,.
\end{array}
$$
On the other hand, from the definitions of $g_1$, $g_2$ and $g_3$, as $\rho\to0$,
$$
\begin{array}{rcl}\displaystyle
\int_{0}^{1}\frac{g_{3}(x)\,x^{3}}{\sqrt{x(x+\rho)}}\,\dif x
&=&\displaystyle
-\,g_1(0)\,-\,\frac12\,g_2(0)\,+\,\frac\rho2\,g_2(0)
\,+\,\int_{0}^{1}g_{1}(x)\,\dif x
\\[5pt]\displaystyle
&&\displaystyle
\,-\,\frac\rho2\int_{0}^{1}g_{2}(x)\,\dif x
\,+\,\frac{3\rho^2}{8}\int_{0}^{1}g_{3}(x)\,\dif x
\,+\,{\mathcal O}(\rho^3\,\ln\rho)\,.
\end{array}
$$
This eventually yields
$$G(\rho)\,=\begin{array}[t]{l}\displaystyle
\,-\,g(0)\,\ln\rho \,+\,\int_{0}^{1} g_1(x)\,\dif x\,+\,g(0)\,\ln 4\,+\,g_1(0)\,\dfrac{\rho}{2}\,\ln\rho\\
[15pt] \,+\,\displaystyle
\frac{\rho}{2}\,\left(g(0)\,+\,(1-\ln4)\,g_1(0)\,-\, \int_{0}^{1} g_2(x)\,\dif x\right)
\,-\,\frac{3}{8} \,g_2(0)\,\rho^2\,\ln \rho \\ [15pt]
\displaystyle
\,+\,\frac{\rho^2}{16}\,\left(-3g(0)-6g_1(0)+g_2(0)(-7+6\ln4)\,+\,6\,\int_{0}^{1}g_3(x)\,\dif x\right)
\,+\,{\mathcal O}(\rho^3\ln\rho)\,.\end{array}$$
We complete the computation of the first coefficients in the expansion of $G$ by substituting 
for $g_1$, $g_2$ and $g_3$ their expressions in terms of derivatives of $g$ obtained from the Taylor formulas.

As regards the expansion of $F$, we may note that for $\rho>0$, 
$$F'(\rho)\,=\,\frac12 \int_{0}^{1} f(x)\,\sqrt{\frac{x}{x+\rho}}\,\dif x\,,$$
which is nothing but $G(\rho)$ for $g(x)=xf(x)/2$.
Expanding this specific $G$ 
and applying Lemma~\ref{lem:derasexp} as in the comment following it we conclude to the existence of the claimed expansion for $F$ with 
$$
B_0=\lim_{\rho\to0}F(\rho)=\int_{0}^{1} f(x)x\,\dif x\,,\qquad\qquad
A_1=a_0=-g(0)=0\,,
$$
and
$$
B_1=b_0-a_0=\frac12\int_{0}^{1} f(x)\,\dif x\,,\qquad\qquad\qquad
A_2=\frac{a_1}{2}=\frac{f(0)}{8}\,.
$$

Regarding the expansion of $H$, it can be justified as the one of $G$ by decomposing as
$$
H(\rho)=\sum_{j=0}^{J} h_j(0)\,\int_{0}^{1} \frac{x^j}{\sqrt{x(x+\rho)^3}}\,\dif x\,+\,\int_{0}^{1} h_{J+1}(x)\,\frac{x^{J+1}}{\sqrt{x(x+\rho)^3}}\,\dif x\,,
$$
for any $J\geq0$, where $h_j$s are defined by induction as $h_0:=h$ and for all $j\geqslant 0$, $h_{j+1}(0)=h_j'(0)$ and $h_{j+1}(x)=(h_j(x)-h_j(0))/x$ for all $x\neq 0$.
Indeed the existence of the claimed expansion then follows by arguing recursively from on the one hand, for $j\geq1$
$$
\begin{array}{rcl}\displaystyle
\int_{0}^{1} \frac{x^j}{\sqrt{x(x+\rho)^3}}\,\dif x&=&\displaystyle
\int_{0}^{1} \frac{x^{j-1}}{\sqrt{x(x+\rho)}}\,\dif x\,-\,\rho\,
\int_{0}^{1} \frac{x^{j-1}}{\sqrt{x(x+\rho)^3}}\,\dif x\\[10pt]
\displaystyle
\int_{0}^{1} \varphi(x)\, \frac{x^{j}\,\dif x}{\sqrt{x(x+\rho)^3}}&=&\displaystyle
\int_{0}^{1} \varphi(x)\, \frac{x^{j-1}\,\dif x}{\sqrt{x(x+\rho)}}\,-\,\rho\,
\int_{0}^{1} \varphi(x)\,\frac{x^{j-1}\,\dif x}{\sqrt{x(x+\rho)^3}}
\end{array}
$$
and, on the other hand,
$$
\begin{array}{rcl}\displaystyle
\int_{0}^{1} \frac{\dif x}{\sqrt{x(x+\rho)^3}}\,&=&\displaystyle
\frac{4}{\rho}\int_{1}^{(2+\rho+2\sqrt{1+\rho})/\rho} \frac{\dif s}{(s+1)^2}\,=\,\frac{2}{\rho}\,-\,\frac{2}{1+\rho+\sqrt{1+\rho}}
\\[10pt]
\displaystyle
\int_{0}^{1}\frac{\varphi(x)\,\dif x}{\sqrt{x(x+\rho)^3}}&=&\displaystyle
{\mathcal O}\left(\frac1\rho\right)
\end{array}
$$
and the fact that the function $\rho\mapsto1/(1+\rho+\sqrt{1+\rho})$ has a power series expansion.
\end{proof}

Note that if we were willing to compute the coefficients of the expansion of $H$ we could rely on Lemma~\ref{lem:derasexp}
and use that $G'\,=\,H$ provided $g=-2h$.

\bibliographystyle{plain}

\end{document}